\documentclass[11pt]{smfart}
\usepackage{amssymb,latexsym}
\usepackage{graphicx}
\usepackage{version}
\usepackage[all]{xy} 
\usepackage{lscape}
\entrymodifiers={!!<0pt,0.7ex>+}

\usepackage{geometry}
\usepackage[applemac]{inputenc} 
\usepackage{amsfonts}
\usepackage{amssymb}
\usepackage{epstopdf}

\newtheorem{theorem}{Theorem}[section]

\newtheorem{proposition}{Proposition}[section]

\newtheorem{lemme}{Lemma}[section]
\newtheorem{definition}{Definition} [section]
\newtheorem{notation}{Notation}[section]
\newtheorem{remarque}{Remark}[section]

\newtheorem{result}{Result}

\begin{document}
\setcounter{tocdepth}{1}
\title[Partitions and geometry]{Combinatorial theory of permutation-invariant random matrices~I: \\ Partitions, geometry and renormalization.}
\author[Franck Gabriel]{Franck Gabriel \\ E-mail: \parbox[t]{0.45\linewidth}{\texttt{franck.gabriel@normalesup.org}}}
\email{franck.gabriel@normalesup.org}
\address{Université Pierre et Marie Curie (Paris 6)\\ Laboratoire de Probabilités et Modèles Aléatoires\\ 4, Place Jussieu\\  F-75252 Paris Cedex 05. }
\address{Mathematics Institute\\ 
Zeeman Building\\ 
University of Warwick\\ 
Coventry CV4 7AL.}

\begin{abstract}Using a natural distance, we define and study a family of orders on partitions of a given set $X$. When the set $X$ is the disjoint union of two copies of $\{1,...,k\}$, there exists a canonical choice of order in the family we constructed. We show that the two ordered sets of partitions and non-crossing partitions of $k$ elements can be seen as subsets of the ordered set $\mathcal{P}_k$ of partitions on $X$. We generalize the notion of Kreweras complement to the set of partitions $\mathcal{P}_k$. These notions allow us to define new structures on the linear forms on partitions: some triangular transformations, two convolutions, a multiplicative bi-albegra structure and an Hopf algebra structure and some natural projections. We study the set of characters on partitions and their interaction with the newly constructed structures. 
At last we show that the abstract structures appear naturally when one considers the notion of convergence in moments for sequence of partitions. The notion of convergence is generalized in order to study the algebraic fluctuations. 
\end{abstract}

\maketitle

\tableofcontents

\section{Introduction}

This article is the first of a serie of three in which we generalize the notions of independence and freeness in order to define a notion of $\mathcal{A}$-freeness in the setting of $\mathcal{A}$-tracial algebras. This setting unifies classical and free probabilities and allows us to study random matrices which are not asymptotically invariant in law by conjugation by the unitary group. In this article, the reader will find the needed combinatorial tools; in the article \cite{Gab2}, he will find the study of $\mathcal{A}$-tracial algebras and applications to random matrices; the article \cite{Gab3} uses the previous result and focuses on the study of general random walks on the symmetric group and the construction of the $\mathfrak{S}(\infty)$-master field.

The set of partitions of $k$ elements, denoted by ${\sf P}_k$ and the set of non-crossing partitions of $k$ elements, denoted by ${\sf NC}_k$, both endowed with the finer-order $\trianglelefteq$, are two important ordered sets in classical probabilities and free probabilities. Their importance comes from the fact that one can define the notions of cumulants and independence or freeness using these ordered sets. For example, let $X_1$,...,$X_k$ be random variables which have all moments bounded. The cumulants of $X_1$, ..., $X_k$ are defined by the fact that for any integer $l$ and any $i_1,...,i_l \in \{1,...,k\}$:
\begin{align*}
\mathbb{E}[X_{i_1}...X_{i_l}] = \sum_{ \pi \in {\sf P}_l} \prod_{b \in \pi} {\sf cum}((X_{i_u})_{u \in b}).
\end{align*}
Besides, two random variables $X$ and $Y$ are independent if and only if their mixed classical cumulants vanish. This means that for any integers $k$, $l \geq 1$, ${\sf cum}(X,...,X,Y,...,Y) = 0$ where we wrote $k$ times $X$ and $l$ times $Y$.

One can define the free cumulants and freeness by considering elements in a non-commutative algebra $A$ endowed with a tracial state $\phi$, and using the set of non-crossing partitions ${\sf NC}_l$ instead of the set ${\sf P}_l$. The set of non-crossing partitions ${\sf NC}_l$ is also endowed with an interesting involution, the Kreweras complement involution \cite{kreweras}. Actually, in this paper we prove the two following results. 

\begin{result}
For any integer $k$, there exists an order $\leq$ on ${\sf P}_{2k}$ such that the sets $({\sf P}_k, \trianglelefteq)$ and $({\sf NC}_k, \trianglelefteq)$ can be seen as subsets of $({\sf P}_{2k}, \leq)$.
\end{result}

\begin{result}
For any integer $k$, there exists a notion of Kreweras complement on ${\sf P}_{2k}$ which generalizes the notion of Kreweras complement on ${\sf NC}_k$. 
\end{result}

In order to prove this, we define for any set $X$ a distance on the set of partitions $\mathcal{P}(X)$ which allows us to define a family of geodesic orders that we thoroughly study in Section~\ref{sec:geometry}. The main results of this section are summarized below. 

\begin{result}
The function defined on $\mathcal{P}(X) \times \mathcal{P}(X)$: 
\begin{align*}
d(p,p')= \frac{1}{2} (\# p + \# p') - \# (p \vee p'),
\end{align*}
where $p \vee p'$ is the finest partition which is coarser than $p$ and $p'$, is a distance.
\end{result}

Let $b$ be a partition in $\mathcal{P}(X)$ and let say that $p' \leq_b p$ if $d(b,p')+d(p',p) = d(b,p)$.

\begin{result}
The order $\leq_b$ is fully characterized: 
\begin{enumerate}
\item there exists a ``decomposition'' of $\leq_b$ using two simpler well-understood orders, 
\item the Hasse diagram is described, 
\item the Möbius function of $\leq_b$ is computed. 
\end{enumerate}
\end{result}

In Section \ref{sec:sectionPk}, we apply these two last results to the set $\mathcal{P}(\{1,...,k,1',...,k'\}) = \mathcal{P}_{k}$ which is in bijection with ${\sf P}_{2k}.$ This allows us to prove in particular the Result $1$ which is partly a consequence of the following result.

\begin{result}
The order $\leq$ on $\mathcal{P}_k$ is a natural generalization of the Bruhat order on permutations. 
\end{result}

The order in this special case exhibits some other interesting properties, like the geodesic factorization, that we study further. In Section \ref{sec:Kreweras}, Result $2$ is proved as a consequence of the new Inequality (\ref{autreequation}) which links the distance and the multiplication operations on $\mathcal{P}_k$. The main result of this section is Theorem $2.4$ which links different notions of defect for partitions in $\mathcal{P}_k$ and which allows us to prove the following result. 

\begin{result}
The notion of Kreweras complement can be used in order to define a new order on partitions, denoted $\prec$ which satisfies a nice factorization property. 
\end{result}

In Section \ref{sec:structure}, we follow some ideas of \cite{mastnaknica}: using the previous results, we define some structures on the set of linear forms $(\oplus_{k=0}^{\infty}\mathbb{C}[\mathcal{P}_k])^{*}$.

\begin{result}
The set $(\oplus_{k=0}^{\infty}\mathbb{C}[\mathcal{P}_k])^{*}$ can be endowed with: 
\begin{enumerate}
\item two convolutions which allow us to define: 
\begin{itemize}
\item a structure of graded connected Hopf algebra,
\item a structure of associative, co-associative bi-algebra, 
\end{itemize}
\item notions of characters and infinitesimal characters which are compatible with one another through the two notions of convolutions, 
\item some triangular transformations which nicely interact with the notions of character and infinitesimal characters, 
\item three natural projections: the cumulant-projection, the moment-projection and the exclusive-projection.
\end{enumerate}
\end{result}

In Section \ref{sec:obsconv}, we explain how the structures defined on $(\oplus_{k=0}^{\infty}\mathbb{C}[\mathcal{P}_k])^{*}$ appear naturally when one considers a special notion of convergence for sequences of elements in $\mathbb{C}[\mathcal{P}_k]$. Actually, we emulate the theory of random matrices in a combinatorial framework: for any parameter $N$, we introduce a family of linear forms on the partition algebras which allows us to define a notion of weak convergence similar to the convergence in moments in random matrices theory. This notion of convergence is linked with a notion of moments, yet we can link it with the asymptotics of the coordinates. 

\begin{result}
A sequence $(E_N)_{N \in \mathbb{N}}$ converges if and only if the coordinates of $E_N$ satisfy a specific asymptotic behaviour as $N$ goes to infinity. 
\end{result}

We also study a notion of exclusive moments and show that the convergence of these exclusive moments is equivalent to the convergence of the moments. Besides, it is well known that for any integer $N$, there exists a natural multiplication on $\mathbb{C}[\mathcal{P}_k]$ which depends on $N$ \cite{Halv}. Let us denote it by $\times_N$: the convergence is ``compatible" with this family of multiplications.

\begin{result}
If $(E_N)_{N \in \mathbb{N}}$ and  $(F_N)_{N \in \mathbb{N}}$ converge for the notion of convergence defined in Section  \ref{sec:obsconv}, then  $E_N \times_N F_N$ converges and the limit of $E_N \times_N F_N$ is linked with the multiplicative convolution on  $(\oplus_{k=0}^{\infty}\mathbb{C}[\mathcal{P}_k])^{*}$. 

If for any integer $N$, $(E_N^{t})_{t \geq 0})$ is a semi-group for $\times_N$, if the sequence of generators of $(E_N^{t})_{t \geq 0})$ converges then for any $t \geq 0$, $E_N^{t}$ converges.
\end{result}

Gathering all the results in Section \ref{sec:obsconv}, we obtain the following result.
\begin{result}
There exists a natural notion of convergence for sequences in  $\mathbb{C}[\mathcal{P}_k]$ such that the structures defined on $(\oplus_{k=0}^{\infty}\mathbb{C}[\mathcal{P}_k])^{*}$ can be approximated by natural structures on $\mathbb{C}[\mathcal{P}_k]$.
\end{result}

In Section \ref{sec:Fluctuations}, we generalize the results obtained about the convergence of sequences in $\mathbb{C}[\mathcal{P}_k]$ in order to deal with algebraic fluctuations of these sequences.

\section{Geometry and orders on partitions}
\label{sec:geometry}
Let us consider a finite set $X$. The set of partitions of $X$ is the set: 
\begin{align*}
\mathcal{P}(X) = \{ \{b_1,...,b_k\} | \emptyset \neq b_1, ..., b_k \subset X ; \cup_{i=1}^{k} b_i = X ; \forall i\neq j\in \{1,...,k\}, b_i \cap b_j = \emptyset\}, 
\end{align*}
Let $p$ be an element of $\mathcal{P}(X)$. Let $b \in p$: it is called a block of $p$. We denote by ${\sf nc}(p)$ the number of blocks of $p$. The set $\mathcal{P}(X)$ can be endowed with a first order $\trianglelefteq$: $p \trianglelefteq p'$ if and only if $p$ is finer than $p'$: for any $b \in p$, there exists $b' \in p'$ such that $b \subset b'$. For the opposite order we will say that $p'$ is coarser than $p$. For any partitions $p$ and $p'$, we denote by $p \vee p'$ the smallest partition for $\trianglelefteq$ which is coarser than $p$ and $p'$. 

Any partition $p \in \mathcal{P}(X)$ can be represented by a graph. For this we consider some vertices which represent $X$: any edge between two vertices means that the labels of the two vertices are in the same block of the partition $p$. An example is given in Figure \ref{fig:exemplepartition}.  Using this graphical representation, one can recover a diagram representing $p \vee p'$ by putting a diagram representing $p'$ over one representing $p$.  

\begin{figure}[h!]
 \centering
  \includegraphics[width=80pt]{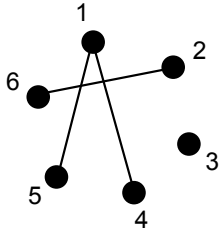}
 \caption{The partition $\{\{1,4,5\}, \{2,6\}, \{3\}\}$ in $\mathcal{P}_{\{1,2,3,4,5,6\}}$.}
 \label{fig:exemplepartition}
\end{figure}

\subsection{Cayley graph of $\mathcal{P}(X)$}

In this section we define a natural graph on $\mathcal{P}(X)$ which will allow us to define a family of distances on $\mathcal{P}(X)$. 

\begin{definition}
\label{Cayley}
The Cayley graph $\mathbb{G} = (\mathbb{V}, \mathbb{E})$ is given by: 
\begin{itemize}
\item the set of vertices $\mathbb{V}$ is $\mathcal{P}(X)$, 
\item there exists an edge in $\mathbb{E}$ between $p$ and $p'$, two elements of $\mathcal{P}(X)$, if and only if one can go from one to the other by gluing two blocks. 
\end{itemize} 
\end{definition}

Actually, it is almost the Hasse diagram of $(\mathcal{P}(X), \trianglelefteq)$: we only forget about the orientation of the diagram. Using this graph, we can define a geodesic distance on $\mathcal{P}(X)$.  

\begin{definition}
Let $p$ and $p'$ be two elements of $\mathcal{P}(X)$. Let $C_{\mathbb{G}}(p,p')$ be the set of paths $\pi$ in $\mathbb{G}$ which begin in $p$ and finish in $p'$: it is always non-empty and $C_{\mathbb{G}}(p,p)$ has only one element, the constant path of lenght $0$. The geodesic distance on $\mathcal{P}(X)$ between $p$ and $p'$ is: 
\begin{align*}
d(p,p') = \frac{1}{2} \min_{\pi\in C_{\mathbb{G}}(p,p')} \# \pi,
\end{align*}
where $\#\pi$ is the length of $\pi$. A path $\pi$ such that $\# \pi = d(p,p')$ is called a geodesic between $p$ and $p'$. 
\end{definition}

The geodesic distance can be computed easily using the following result. 

\begin{theorem}
\label{def:dist}
For any $p$ and $p'$ in $\mathcal{P}(X)$: 
\begin{align*}
d(p,p') = \frac{1}{2} \big({\sf nc}( p) + {\sf nc}(p')\big)-{\sf nc}(p\vee p'). 
\end{align*}
\end{theorem}

\begin{proof} 
For any partitions $p$ and $p'$, we set: 
\begin{align*}
d'(p,p') = \frac{1}{2} \big({\sf nc}( p) + {\sf nc}(p')\big)-{\sf nc}(p\vee p'). 
\end{align*}
It satisfies that $d'(p,p) = 0$ for any partition $p$. Let $p$ and $p'$ be two elements of $\mathcal{P}(X)$. Let us see what happens to $d'(p,p')$ when one moves from $p'$ to one neighborhood of $p'$ in $\mathbb{G}$. Suppose that we glue two blocks of $p'$, then ${\sf nc}(p )$ is constant, ${\sf nc}(p' )$ decreases by $1$ and ${\sf nc}(p \vee p')$ stays constant or decreases by $1$. In this case $d'(p,p')$ will increase or decrease by $0.5$. Suppose now that we cut one block of $p'$, then ${\sf nc}(p )$ is constant, ${\sf nc}(p' )$ increases by $1$ and ${\sf nc}(p \vee p')$ stays constant or increases by $1$. In this case $d'(p,p')$ will also increase or decrease by $0.5$. Thus a gluing/cutting can at most increase the value of $d(p,p')$ by $0.5$. It implies that $d'(p,p') \leq d(p,p')$. 

We have to show that $d(p,p') \leq d'(p,p')$. Let us remark that  $p\vee p'$ is coarser than $p$: we can go from $p$ to $p \vee p'$ by doing ${\sf nc}( p) - {\sf nc}(p \vee p')$ gluing of blocks. The same holds for $p'$: we can go from $p'$ to $p \vee p'$ by doing ${\sf nc}(p') - {\sf nc}(p \vee p')$ gluing of blocks. Thus one can go from $p$ to $p \vee p'$ and then from $p\vee p'$ to $p'$ in ${\sf nc}(p)+{\sf nc}(p') - 2 {\sf nc}(p \vee p')$ steps in $\mathbb{G}$. Thus $d(p,p') \leq \frac{1}{2}\left[{\sf nc}(p')+{\sf nc}(p') - 2 {\sf nc}(p \vee p')\right] = d'(p,p')$.
\end{proof}

Using this distance, we can define a notion of segments in $\mathcal{P}(X)$. 
\begin{definition}
Let $p_1$ and $p_2$ be in $\mathcal{P}(X)$. The segment $[p_1,p_2]$ is given by: 
\begin{align*}
[p_1,p_2] = \{p \in \mathcal{P}(X), d(p_1,p)+d(p,p_2)=d(p_1,p_2) \}.
\end{align*}
\end{definition}
An other geometric interpretation is to say that any element $p$ is in the segment $[p_1,p_2]$ if and only if $p$ is on a geodesic path between $p_1$ and $p_2$.

\subsection{Family of orders}
\subsubsection{Definition}
Using the geodesic distance $d$, we can define a family of geodesic orders. In order to understand this family, we introduce also two new families of orders which are sligth modifications of the coarser and finer orders. To define these orders, we consider a {\sf base partition} $b \in \mathcal{P}(X)$. 

\begin{definition}
Let $p$ and $p'$ be two elements of $\mathcal{P}({X})$. We define three new orders $ \leq_{b}, \dashv_b$ and $\sqsupset_b$: 
\begin{description}
\item[geodesic order] $p' \leq_b p$ if $p' \in [b,p]$, 
\item[coarser-compatible order] $p'\dashv_b p $ if $p'$ is coarser than $p$ and ${\sf nc}(p' \vee b) = {\sf nc}(p \vee b)$, 
\item[finer-compatible order] $p'\sqsupset_b p $ if $p'$ is finer than $p$ and ${\sf nc}(p')- {\sf nc}(p' \vee b) = {\sf nc}(p) - {\sf nc}(p \vee b)$.
\end{description}
\end{definition}
 
Using the fact that $d$ is a distance, it is easy to see that $\leq_b$ is an order on $\mathcal{P}(X)$  for which $b$ is the smallest partition: it is the geodesic order with base partition $b$. 

\begin{remarque}
The order $\trianglelefteq$ is the geodesic order with base partition ${\sf 0}_{X} = \{ \{ x\}, x \in X\}$.  Indeed, using Theorem \ref{def:dist}, we can see that $p' \leq_{{\sf 0}_{X}} p$ if and only if ${\sf nc}(p) = {\sf nc}(p \vee p')$ which is equivalent to the fact that $p'$ is finer than $p$. 

The geodesic order with base partition ${\sf 1}_{X} = \{ \{ x, x \in X\}\}$ is the order $\trianglerighteq$ on $\mathcal{P}(X)$.
\end{remarque}

\begin{remarque}
\label{remarque:finer}
When the partition $p$ gets finer, the quantities ${\sf nc}(p)$, ${\sf nc}(p\vee b)$ and ${\sf nc}(p)-{\sf nc}(p\vee b$) increase. In particular if one such quantity is the same at the beginning and at the end of a chain of finer and finer partitions, it must stay constant all along this chain. 
\end{remarque}

It will be useful to denote the defect of $p'$ from not being on $[b, p]$ by: 
\begin{align}
\label{definitiondefect}
{\sf df}_b{(p',p)} = d(b,p')+d(p',p)- d(b,p) = {\sf nc}(p') - {\sf nc}(p' \vee b) - {\sf nc}(p \vee p') + {\sf nc}(p \vee b).  
\end{align}

\subsubsection{Study of the coarser and finer-compatible orders}
\begin{lemme}
\label{lemme:carac1}
Let $p$ and $p'$ be two elements of $\mathcal{P}(X)$. We have the following characterization of the coarser-compatible and finer-compatible order: 
\begin{enumerate}
\item $p'\dashv_b p$ if and only if  $p'$ is coarser than $p$ and $p' \leq_b p$. 
\item $p'\sqsupset_b p$ if and only if $p'$ is finer than $p$ and $p' \leq_b p$.  
\end{enumerate}
\end{lemme}

\begin{proof}
This is a straightforward consequence of the fact that: 
\begin{enumerate}
\item if $p'$ is coarser than $p$, ${\sf df}_{b}(p',p) = -{\sf nc}(p' \vee b) + {\sf nc}(p \vee b)$,
\item if $p'$ is finer than $p$, ${\sf df}_{b}(p',p) = {\sf nc}(p')-{\sf nc}(p' \vee b) - {\sf nc}(p) + {\sf nc}(p \vee b)$.
\end{enumerate}
\end{proof}

It would be interesting to have a better understanding of the orders $\dashv_b$ and $\sqsupset_b$. In order to do so, we introduce the notion of pivotal blocks, admissible splits and admissible gluings for a partition $p \in \mathcal{P}(X)$.

\begin{definition}
\label{def:pivotal}
A {\em pivotal block} (for the base partition $b$) for $p$ is a block of $p$ such that there exists a way to cut it into two blocks in order to cut a block of $p\vee b$ into two blocks. We denote by ${\sf Piv}_b(p)$ the set of pivotal blocks for $p$.

We denote by $\Delta_b(p)$ the set of all partitions $p'$ which are obtained by cutting in $p$ a pivotal block for $p$ into two blocks in such way that $p'\vee b$ has one more block than $p\vee b$. This defines a function $\Delta_b$ from $\mathcal{P}(X)$ to the subsets of $\mathcal{P}(X)$. The {\em admissible splits} of $p$ are ${\sf Sp}_b(p) = \bigcup\limits_{k=0}^{\infty} \Delta_b^{k}(p).$ 
\end{definition}

\begin{definition}
\label{definitioncoarsercomp} 
Let ${\sf Gl}_{b}(p )$ be the set of partitions $p'$ in $\mathcal{P}(X)$ such that $p'$ is obtained by gluing blocks of $p$ in a way such that ${\sf nc}(p \vee b) = {\sf nc}(p' \vee b)$. It is the set of {\em admissible gluings} of $p$. 
\end{definition}

We can better understand the orders $\dashv_b$ and $\sqsupset_b$. 

\begin{lemme}\label{fin}
Let $p$ and $p'$ be two elements of $\mathcal{P}(X)$. We have the following equivalences: 
\begin{enumerate}
\item $p' \dashv_b p$ if and only if $p' \in {\sf Gl}_{b}( p)$,  
\item $p' \sqsupset_b p$ if and only if $p' \in {\sf Sp}_{b}( p)$. 
\end{enumerate}
\end{lemme}

\begin{proof} The first equivalence is straightforward. Let us prove that the second equivalence holds. We can suppose that $p'$ is finer than $p$ since it is implied by both conditions. We have to prove that ${\sf nc}(p' ) - {\sf nc}(p' \vee b)$ is equal to  ${\sf nc}(p ) - {\sf nc}(p \vee b)$ if and only if there exists a path $p_0,...,p_k$ in the Cayley graph of $\mathcal{P}(X)$ such that $p_0=p$, $p_k=p'$ and $p_{i+1} \in \Delta_{b}(p_{i})$ for any $i \in \{0,...,k-1\}$. 

Let us suppose that such a path exists: by definition of a pivotal block, we see that for any $i \in \{0,...,k-1\}$, ${\sf nc}(p_i ) -{\sf nc}(p_i \vee b) = {\sf nc}(p_{i+1} ) -{\sf nc}(p_{i+1} \vee b) $ and thus, ${\sf nc}(p ) - {\sf nc}(p \vee b) = {\sf nc}(p_0 ) - {\sf nc}(p_0 \vee b) = {\sf nc}(p_k ) - {\sf nc}(p_k \vee b)={\sf nc}(p' ) - {\sf nc}(p' \vee b)$. 

Let us suppose instead that ${\sf nc}(p ) - {\sf nc}(p \vee b)={\sf nc}(p' ) - {\sf nc}(p' \vee b)$. Since $p'$ is finer than $p$, there exists a path $\pi_0,...,\pi_l$ in the Cayley graph of $\mathcal{P}(X)$ such that $\pi_0=p$, $\pi_k=p'$ and $\pi_{i+1}$ is obtained by cutting a block of $\pi_i$ for any $i \in \{0,...,k-1\}$. At each step the number of blocks of $\pi_i$ goes up by one and the number of blocks of $\pi_i\vee b$ is either constant or goes up by one. Since ${\sf nc}(\pi_0 ) - {\sf nc}(\pi_0 \vee b)={\sf nc}(\pi_l ) - {\sf nc}(\pi_l \vee b)$, $\left({\sf nc}(\pi_i) - {\sf nc}(\pi_i \vee b)\right)_{i=0}^{l-1}$ must be constant. This means that at each step the number of blocks of $\pi_i\vee b$ must go up by one: $\pi_{i+1} \in \Delta_{b}(\pi_i)$ for any $i \in \{0,...,l-1\}$.
\end{proof}

\subsection{The matrices of the orders}
In the following, by matrice, we understand a triple $(I,J,M)$ where $I$, the departure set, and $J$, the arrival set, are finite sets and $M$ is an application from $I\times J$ in $\mathbb{R}$. Given $(I_1,I_2,M_1)$ and $(I_2,I_3,M_2)$ two matrices, we can multiply them and $(I_1,I_2,M_1) (I_2,I_3,M_2) = (I_1,I_3,M_1M_2)$ where: 
\begin{align*}
M_1M_2(i,k) = \sum_{j \in I_2} M_1(i,j)M_{2}(j,k).
\end{align*}

Let us consider a finite set $I$ endowed with an order $\prec$. A matrix $(I,I,M)$ is lower triangular if for any $i$ and $i'$ in $I$, if $M(i,i') \neq 0$ then $i' \prec i$. Let us remark that any matrix $(I,I,M)$ which is strictly lower triangular is nilpotent : there exists a positive integer $n$ such that $M^{n}=0$. Indeed, it is enough to consider $n=\#I$. The matrices that we will consider in the following have departure and arrival sets equal to $\mathcal{P}(X)$: we will omit to specify it in the following.

\begin{definition}
\label{matriceorder}
 The matrices of the partial orders $\leq_{b}, \dashv_b, \sqsupset_b$ are: 
 \begin{itemize}
\item for the geodesic order $\leq_b$:\ \ \ \ \ \ \  \ \ \ \ \ \  $(G_b)_{p,p'} = \delta_{p' \leq_b p}$,
\item for the coarser-compatible order $\dashv_b$: $(C_b)_{p,p'} = \delta_{p' \dashv_b p}$, 
\item for the finer-admissible order $\sqsupset_b$: \ \ \ $(S_b)_{p,p'} = \delta_{p' \sqsupset_b p}$. 
 \end{itemize}
\end{definition}

Let us remark that these matrices are lower triangular when their departure set is endowed with the corresponding order. These matrices satisfy the next important identity.  

\begin{theorem}
\label{th:lienmatriciel}
The partial orders $\leq_b, \dashv_b$, and $\sqsupset_b$ are linked by the following equality: 
\begin{align*}
G_b = C_bS_b. 
\end{align*} 
\end{theorem}

\begin{proof}
Let us consider $p$ and $p'$ in $\mathcal{P}(X)$. We have: 
\begin{align*}
(C_bS_b)_{p,p'} = \sum_{p'' \in \mathcal{P}_k} \delta_{p'' \dashv_b p} \delta_{p' \sqsupset_b p'' }. 
\end{align*}
Thus it is enough to show that $p' \leq_b p$ if and only if there exists $p'' \in \mathcal{P}(X)$ such that $p'' \dashv_b p$ and $p' \sqsupset_b p''$. Besides we need to show that such partition $p''$ is unique: we will show that it is equal to $p \vee p'$. 

We have the following equalities: 
\begin{align*}
{\sf df}_b(p',p) &= {\sf nc}(p \vee b) - {\sf nc}(p' \vee b) + {\sf nc}(p') - {\sf nc}(p \vee p')\\
&= [{\sf nc}(p \vee b) - {\sf nc} (p \vee p' \vee b)] + [{ \sf nc } ( p \vee p' \vee b )  - {\sf nc}(p' \vee b) + {\sf nc}(p') - {\sf nc}(p \vee p')]
\\&= {\sf df}_b (p \vee p' , p) + {\sf df}_b(p', p \vee p'). 
\end{align*}
Thus, $p' \leq_b p$ if and only if $p \vee p' \leq_b p$ and $p' \leq_b p \vee p'$. Yet, $p \vee p' $ is coarser than $p$ and $p'$ is finer than $p\vee p'$. Using Lemma \ref{lemme:carac1}, we see that $p' \leq_b p$ if and only $p \vee p'  \dashv_b p$ and $p' \sqsupset_b p\vee p'$. The proof of the theorem will be completed if we prove that if $p''$ satisfies $p'' \dashv_b p$ and $p' \sqsupset_b p''$ then $p' \leq_{b} p$ and $ p'' = p \vee p'$. 

Let us consider such a partition $p''$, we have:
\begin{align}
\label{eq:1}
{\sf nc}(p'' \vee b) &= {\sf nc}(p \vee b),\\
\label{eq:2}
{\sf nc}(p'') - {\sf nc}(p'' \vee b) &= {\sf nc}(p') - {\sf nc}(p' \vee b). 
\end{align}
Using Lemma \ref{lemme:carac1}, $p' \leq_b p'' \leq_b p$, thus $p' \leq_b p$: 
\begin{align}
\label{eq:3}
{\sf nc}(p \vee p')- {\sf nc}(p \vee b) = {\sf nc}(p') - {\sf nc}(p' \vee b). 
\end{align}
Thus, we have the following equalities: 
\begin{align*}
{\sf nc}(p \vee p') = {\sf nc}(p \vee b)+{\sf nc}(p') - {\sf nc}(p' \vee b) 
&= {\sf nc}(p \vee b)+{\sf nc}(p'') - {\sf nc}(p'' \vee b)\\
&={\sf nc}(p''),
\end{align*}
where we applied successively the Equations $(\ref{eq:3})$, $(\ref{eq:2})$ and $(\ref{eq:1})$.
Since $p''$ is coarser than $p$ and than $p'$, it is coarser than $p \vee p'$ and thus the last equation implies that $p''$ is equal to $p \vee p'$. 
\end{proof}

An other version of \ref{th:lienmatriciel} is given in the following theorem. 

\begin{theorem}
\label{th:geocharact}
The partition $p'$ is in $[b, p]$ if and only there exists $p'' \in \mathcal{P}(X)$ such that the two following conditions hold: 
\begin{enumerate}
\item $p'' \in {\sf Gl}_{b}(p)$, 
\item $p' \in {\sf Sp}_b(p'')$. 
\end{enumerate}
If so, then $p'' = p \vee p'$. 
\end{theorem}

\subsection{Hasse diagram of the geodesic order}

Let us consider $(T,\leq)$ a finite set endowed with a partial order. The Hasse diagram of $(T,\leq)$ is the oriented graph whose vertices represent the elements of $T$: there exists an oriented edge between the vertex which represents $x$ to the one which represents $y$ if and only if $x$ is directly smaller than $y$, which means that $x < y$ and there does not exist any $z$ such that $x < z < y$.

\begin{theorem}
The Hasse diagram of $\left(\mathcal{P}_{X},\leq_b\right)$ is characterized by the following property: there exists an oriented edge from $p'$ to $p$ if and only if
\begin{itemize}
\item either $p'$ is directly finer-admissible than $p$, 
\item or $p'$ is directly coarser-admissible than $p$. 
\end{itemize}
\end{theorem}

\begin{proof}
Let $p$ and $p'$ in $\mathcal{P}_{X}$ such that $p'$ is directly smaller than $p$ for $\leq_{b}$. Using Theorem \ref{th:geocharact} and Lemmas \ref{lemme:carac1} and \ref{fin}, $p' \leq_b p \vee p'$ and $p\vee p' \leq_b p$. Since  $p'$ is directly smaller than $p$ for $\leq_{b}$, $p\vee p'$ is equal either to $p$ or $p'$. Using again  Lemma \ref{lemme:carac1}, either $p' \sqsupset_b p$ or $p' \dashv_b p$. If there was a partition $p'' \notin \{p,p'\}$ such that $p' \sqsupset_b p'' \sqsupset_b p$, by Lemma \ref{lemme:carac1} it would contradict the fact that $p'$ is directly smaller than $p$ for $\leq_b$. The same argument holds for $\dashv_{b}$. Thus $p'$ is either directly finer-admissible or directly coarser-admissible than $p$.
\end{proof}

\subsection{M\"obius function for the geodesic order}

It would be interesting to compute the Möbius function of $\leq_b$ which is roughly the inverse of the matrix $G_b$. 

\begin{definition}
\label{def:Mob}
Let $I$ be a finite set endowed with a partial order. Let $M$ be the matrix of the order defined as in Definition \ref{matriceorder}. The {\em Möbius function} is the function such that for any $a$ and $b$ in $I$, 
\begin{align*}
\mu(a,b) = (M^{-1})_{b,a}.
\end{align*}
\end{definition}

Let us consider $M$ the matrix used in the last definition. The matrix $M$ can be written as $\mathrm{Id}_I+N$ with $N$ a matrix which is strictly lower triangular and thus nilpotent. This allows us to compute the inverse of $M$: 
\begin{align}
M^{-1}= (\mathrm{Id}_I+N)^{-1} = \sum_{l=0}^{\infty} (-1)^{l} N^{l}, 
\end{align}
hence Rota-Hall's formula: 
\begin{align}
\label{eq:Rota}
M^{-1}\left(x,y\right) = \sum_{l=0}^{\infty} (-1)^{l}  \#C_{l}(x,y), 
\end{align}
where $C_{l}(x,y)$ is the set of sequences of length $l$ which are strictly decreasing between $x$ and $y$: 
\begin{align*}
C_{l}(x,y) = \{ (i_0,...,i_l), x = i_0 \neq ...\neq i_l = y,  i_0 \geq ...\geq i_l \}. 
\end{align*}
Let us remark that the matrix $M^{-1}$ is lower triangular since $N^{l}$ is lower triangular for any integer $l$. 

Our goal is to compute the Möbius function for $(\mathcal{P}(X), \leq_b)$: we need to compute the inverse of $G_b$. In order to do so, we need the following lemma.

\begin{lemme}\label{infimumversion}
Let $p$, $p'$ and $p''$ be three partitions in $\mathcal{P}(X)$. Let us suppose that $p' \dashv_b p''$ and $p'' \sqsupset_b p$, then $p'' = p \wedge p'$, where $p \wedge p'$ is the coarser partition which is finer than $p$ and than $p'$. 
\end{lemme}

\begin{proof}[Proof of Lemma \ref{infimumversion}]
Let us consider $p$, $p'$ and $p''$, three partitions in $\mathcal{P}_k$ which satisfy the hypotheses. Using the definitions of $\dashv_b$ and $\sqsupset_b$: 
\begin{align}
\label{eq:4}
{\sf nc}(p' \vee b) &= {\sf nc}(p'' \vee b),\\
\label{eq:5}
{\sf nc}(p'') - {\sf nc}(p'' \vee b) &= {\sf nc}(p) - {\sf nc}(p \vee b).
\end{align} 
Besides, using Lemma \ref{lemme:carac1}, we know that $p' \leq_b p''$ and $p'' \leq_b p$. Thus $p' \leq_b p$: 
\begin{align}
\label{eq:6}
{\sf nc}(p') - {\sf nc}(p' \vee b) + {\sf nc}(p \vee b) - {\sf nc}(p \vee p') = 0.
\end{align}
This allows us to write the following equalities: 
\begin{align*}
{\sf nc}(p'') = {\sf nc}(p'' \vee b) + {\sf nc}(p) - {\sf nc}(p \vee b) &=  {\sf nc}(p' \vee b) + {\sf nc}(p) - {\sf nc}(p \vee b)\\
&={\sf nc}(p) + {\sf nc}(p') - {\sf nc}(p\vee p'),
\end{align*}
where we applied successively the Equations (\ref{eq:5}), (\ref{eq:4}) and (\ref{eq:6}). Thus: 
\begin{align}
\label{eq:7}
{\sf nc}(p'') + {\sf nc}(p\vee p') - {\sf nc}(p) - {\sf nc}(p')  = 0.
\end{align}

Using the triangle inequality, we know that $d(p,p\wedge p')+ d(p \wedge p', p') - d(p,p') \geq 0$, which is equivalent to:
\begin{align}
\label{eq:8}
{\sf nc}(p \wedge p') + {\sf nc}( p \vee p') - {\sf nc}(p) - {\sf nc}(p') \geq 0. 
\end{align}
Since $p''$ is finer than $p'$ and than $p$, $p''$ is finer than $p \wedge p'$:  ${\sf nc}(p \wedge p') \leq {\sf nc}(p'')$. Using Equations (\ref{eq:7}) and (\ref{eq:8}), we get that ${\sf nc}(p'') = {\sf nc}(p \wedge p')$: $p'' = p \wedge p'$. 
\end{proof}

We can now compute the inverse of $G_b$.
\begin{theorem}
\label{inversematrice}
Let $p$ and $p'$ in $\mathcal{P}(X)$. We have: 
\begin{align*}
(G_b^{-1})_{p,p'} = \delta_{ p\wedge p' \sqsupset_b p }\ \delta_{ p' \dashv_b p\wedge p' }\ \mu_f(p\wedge p',p)\ \mu_f(p\wedge p', p'), 
\end{align*}
where for any partition $p_1$ and $p_2$ such that $p_1$ is finer than $p_2$: 
\begin{align*}
\mu_f(p_1,p_2) = (-1)^{{\sf nc}(p_1)- {\sf nc}(p_2)} \prod_{i=3}^{{\sf nc}(p_1)} ((i-1)!)^{r_i},
\end{align*}
where $r_i$ is the number of blocks of $p_2$ which contains exactly $i$ blocks of $p_1$. 
\end{theorem}

\begin{proof}
Using Theorem \ref{th:lienmatriciel}, we know that $G_b= C_bS_b$. Thus $G_b^{-1} = S_b^{-1}C_b^{-1}$. Let $p$ and $p'$ be two partitions in $\mathcal{P}(X)$: 
\begin{align*}
(G_b^{-1})_{p,p'} = \sum_{p'' \in \mathcal{P}_k} (S_b^{-1})_{p,p''} (C_b^{-1})_{p'',p'}. 
\end{align*}
Since $S_b$, respectively $C_b$, is the matrix of the order $\sqsupset_b$, respectively $\dashv_b$, $S^{-1}_b$, respectively $C^{-1}_b$, is lower triangular for $\sqsupset_b$, respectively $\dashv_b$: for any partitions $p_1$ and $p_2$ in $\mathcal{P}_k$: 
\begin{align*}
(S_b^{-1})_{p_1,p_2} &= \delta_{p_2 \sqsupset_b p_1}(S_b^{-1})_{p_1,p_2},\\
(C_b^{-1})_{p_1,p_2} &= \delta_{p_2 \dashv_b p_1} (C_b^{-1})_{p_1,p_2}. 
\end{align*}
Thus: 
\begin{align*}
(G_b^{-1})_{p,p'} = \sum_{p'' \in \mathcal{P}_k\ |\ p''\sqsupset_b p,\  p' \dashv_b p''} (S_b^{-1})_{p,p''} (C_b^{-1})_{p'',p'}. 
\end{align*}
Using Lemma \ref{infimumversion}, we get that: 
\begin{align*}
(G_b^{-1})_{p,p'}  = (S_b^{-1})_{p,p\wedge p'} (C_b^{-1})_{p \wedge p',p'}. 
\end{align*}
It remains to compute $(S_b^{-1})_{p, p\wedge p'}$ and $(C_b^{-1})_{p\wedge p', p'}$. For sake of clarity, untill the end of the proof, we will forget to specify the base partition $b$. 

Using Rota-Hall's formula, given by Equation (\ref{eq:Rota}), for any $p$ and $p'$ in $\mathcal{P}(X)$:
\begin{align*}
\left(S^{-1}\right)_{p, p \wedge p'} =\sum_{i=0}^{\infty} (-1)^{i} \sum_{(p_0,...,p_i) \in\mathcal{P}_k \mid p = p_0 \neq p_1 \neq ...\neq  p_i = p\wedge p'} \left[\prod_{l=0}^{i-1} \delta_{p_{l+1} \sqsupset p_{l}}\right]. 
\end{align*}
Yet, if $p \wedge p' \sqsupset p$, using Remark \ref{remarque:finer}, for any positive  integer $i$, for any $i+1$-tuple $(p_0,...,p_i)$: 
\begin{align*}
\prod_{l=0}^{i-1} \delta_{p_{l+1} \sqsupset p_{l}} = \prod_{l=0}^{i-1} \delta_{p_{l+1}\trianglelefteq p_{l}}, 
\end{align*} 
where $\trianglelefteq$ was defined at the beginning of Section \ref{sec:geometry}. Thus: 
\begin{align*}
(S^{-1})_{p, p \wedge p'} &= \delta_{p\wedge p' \sqsupset p }\ \sum_{i=0}^{\infty} (-1)^{i}  \sum_{(p_0,...,p_i) \in\mathcal{P}_k \mid p = p_0 \neq p_1 \neq ...\neq  p_i = p\wedge p'} \left[\prod_{l=0}^{i-1} \delta_{p_{l+1} \trianglelefteq p_{l}}\right]. 
\end{align*}
This implies that: 
\begin{align*}
\left(S^{-1}\right)_{p, p \wedge p'}  =  \delta_{p\wedge p' \sqsupset p } (F^{-1})_{p, p\wedge p'}, 
\end{align*}
where $F$ is the matrix such that for any $p_1, p_2 \in \mathcal{P}_k$, $F_{p_1,p_2} = \delta_{p_2 \trianglelefteq p_1}$. The inverse of this matrix is given by: 
\begin{align*}
(F^{-1})_{p_1,p_2} = \mu_f (p_2,p_1), 
\end{align*}
for any  $p_1, p_2 \in \mathcal{P}_k$ such that $p_2 \trianglelefteq p_1$ and where $\mu_{f}$ is the Möbius function for $\trianglelefteq$ and is given in the statement of Theorem \ref{inversematrice} (see Example $2.9$ in \cite{bjorner}). 

Similar arguments allow us to compute the inverse of $C_b$ and to obtain that: 
\begin{align*}
(C_b^{-1})_{p \wedge p', p'} = \delta_{p'\dashv p\wedge p' }\ \mu_{f}(p \wedge p', p'). 
\end{align*}
This allows us to obtain the desired formula for $G_b^{-1}$. 
\end{proof}

\begin{theorem}
The Möbius function for $(\mathcal{P}(X), \leq_b)$, denoted by $\mu_{\leq_b}$, is given by the fact that for any $p_1$ and $p_2$ in $\mathcal{P}(X)$: 
\begin{align*}
\mu_{\leq_b} (p_1,p_2)=\delta_{ p_1 \dashv_b p_1\wedge p_2 }\ \delta_{ p_1\wedge p_2 \sqsupset_b p_2 }\ \mu_f(p_1\wedge p_2,p_1)\ \mu_f(p_1\wedge p_2, p_2). 
\end{align*}
\end{theorem}

Let us remark that, as a by-product of the proof of Theorem \ref{inversematrice}, we computed the matrices $C_b^{-1}$ and $S_b^{-1}$, thus we know the Möbius functions for $\dashv_b$, $\sqsupset_b$ and $\leq_b$.

\section{The set $\mathcal{P}_k$ and the Kreweras complement}
\label{sec:sectionPk}
\subsection{Basic facts}
\subsubsection{Definitions}
Let $k$ be an integer, let us consider $2k$ elements which we denote by: $1, …, k$ and $1', …, k'$.
\begin{definition}
The set of partitions $\mathcal{P}_{k}$ is the set $\mathcal{P}(\{1,...,k,1',...,k'\})$.
\end{definition}
If $k=0$, then $\mathcal{P}_k = \{\emptyset\}$ and ${\sf nc}(\emptyset)= 0$. Let $p$ be an element of $\mathcal{P}_k$. When we represent graphically the partition $p$, we will consider two rows: $k$ vertices are in the top row, labeled by $1$ to $k$ from left to right and $k$ vertices are in the bottom row, labeled from $1'$ to $k'$ from left to right. An example is given in Figure \ref{fig:partition}. 

\begin{figure}[h!]
 \centering
  \includegraphics[width=150pt]{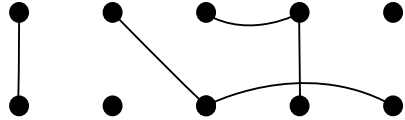}
 \caption{The partition $p=\{ \{1,1'\}, \{2'\}, \{2,3',5'\}, \{3,4,4'\}, \{5\}\}$.}
 \label{fig:partition}
\end{figure}

There exists a special partition, called the identity, in $\mathcal{P}_k$ given by: 
\begin{align*}
\mathrm{id}_{k} = \{ \{i,i'\}, i=1...k\}.
\end{align*}
The diagram of  $\mathrm{id}_{5}$  is drawn in Figure \ref{id5}.
\begin{figure}[h!]
 \centering
  \includegraphics[width=150pt]{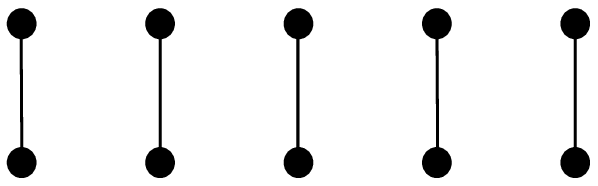}
 \caption{The partition $\mathrm{id}_{5}$.}
 \label{id5}
\end{figure}

From now on, we will always suppose that the base partition for the orders is $\mathrm{id}_{k}$. Besides, we will omit the index which was used to specify the chosen base partition. Thus, using Section \ref{sec:geometry}, we have three new orders on $\mathcal{P}_k$, $\leq$, $\dashv$ and $\sqsupset$, the notions of admissible splits and admissible gluings, three matrices of order $G$, $C$ and $S$, and we know the M\"obius function for each order.

\subsubsection{Irreducible partitions}
\label{sub:irredu}
We define the {\em cycles} of $p$ as the blocks of $p \vee \mathrm{id}_k$. A partition $p$ is {\sf irreducible} if ${\sf nc}(p \vee \mathrm{id}_k) =1$. If it is not the case, $p$ is {\em composed}. In particular, since ${\sf nc}(\emptyset \vee \mathrm{id}_0) = 0$, the empty partition is composed.  We will need a notion of weak irreducibility later: this is based on the notions of extraction. 

Let $J$ be a subset of $\{1,…,k\}\cup\{1',…,k'\}$. Let us denote by $J^{s}$ the symmetrization of $J$: $$J^{s} = J \cup \{ j \in \{1',…,k'\}, \exists i \in J \cap \{1,…,k\}, j=i' \} \cup \{ i \in \{1,…,k\}, i' \in J \}.$$ In order to get the {\sf extraction} of $p$ to $J$, denoted $p_{J}$, let us take the complete graph which represents $p$, let us erase all the vertices which are not in $J^{s}$ and all the edges which are not between two vertices in $J^{s}$ and at last let us label the remaining vertices from left to right. This is the graph of $p_{J}$.

\begin{definition}
\label{supportpart}
The {\em support} of $p$ is: $$ {\sf S}( p) = \{1,…,k\} \setminus \{i \in \{1,…,k\}, \{i, i'\} \in p\}.$$ 
The partition $p$ is {\em weakly-irreducible} if $p_{{\sf S}( p)}$ is either irreducible or equal to the empty partition. 
\end{definition}
In particular, the permutation $\mathrm{id}_k$ is weakly-irreducible. Let us define a notion of {\sf exclusive-irreducibility}. For any integer $k$, ${\sf 0}_k$ is the partition $\{1,...,k,1',...,k'\}$.  
\begin{definition}
\label{supportexclusive}
A partition $p$ is  {\em exclusive-irreducible} if there exists a cycle $c_0$ of $p$ such that any other cycle of $p$ is equal to a partition of the form ${\sf 0}_l$. If the cycle $c_0$ is unique, it is called the {\em exclusive-support} of $p$ and it is denoted by ${\sf Supp}^{c}(p)$. 
\end{definition}

\subsection{Special subsets of $\mathcal{P}_k$}
Since we have chosen a special set $X=\{1,...,k\}\cup\{1',...,k'\}$ and a special base partition, $\mathrm{id}_k$,  in order to define the order on $\mathcal{P}_k$, we can state further results on $\leq$ and ${\sf Sp}(p)$. 

\subsubsection{Definitions}
\label{sec:defA}
\begin{definition}
There exist some special subsets of $\mathcal{P}_k$: 
\begin{itemize}
\item $\mathcal{D}_k$, the set of partitions $p \in \mathcal{P}_{k}$ which are coarser than $\mathrm{id}_{k}$, 
\item $\mathcal{B}_{k}$, the set of Brauer partitions: these are the partitions $p \in \mathcal{P}_{k}$ such that for any block $s$ of $p$, $\#s = 2$. 
\item $\mathfrak{S}_{k}$, the set of permutations: these are the partitions $p \in \mathcal{P}_{k}$ such that for any block $s$ of $p$, $\#\left(s\cap \{1,...,k\}\right) =\#\left(s\cap \{1',...,k'\}\right) =1$. For any permutation $\sigma$, seen as a bijection from $\{1,...,k\}$ to itself, we can associate the partition $\sigma \in \mathfrak{S}_{k}$: 
\begin{align*}
\sigma = \big\{\{i, \sigma(i)'\} \mid  i \in \{1,...,k\}\big\}.  
\end{align*}
\end{itemize}
\end{definition}

\begin{remarque}
In this section, we focus on these special subsets, yet, in \cite{Gab2}, we will use also the two sets: 
\begin{itemize}
\item $\mathcal{H}_k$, the set of partitions $p$ such that for any block $s$ of $p$, $\#s \in 2\mathbb{N}$,
\item $\mathcal{B}s_k$, the set of partitions $p$ such that for any block $s$ of $p$, $\#s \leq 2$. 
\end{itemize}
\end{remarque}

In $\mathcal{B}_k$ and $\mathfrak{S}_k$ some elements are important: the transpositions and the Weyl contractions. Let $i$ and $j$ be two distinct integers in $\{1,...,k\}$. 
\begin{definition}
\label{contractionettranspo}
The transposition $(i,j)$ in $\mathfrak{S}_k$ is: 
\begin{align*}
(i,j) =  \{\{i',j\},\{i,j'\}\} \cup \{\{l,l'\}, l \notin \{i,j\} \}. 
\end{align*}
The Weyl contraction $[i,j]$ in $\mathcal{B}_k$ is: 
\begin{align*}
[i,j] = \{\{i,j\},\{i',j'\}\} \cup \{\{l,l'\}, l \notin \{i,j\} \}.
\end{align*}
\end{definition}

\begin{figure}[h!]
 \centering
  \includegraphics[width=130pt]{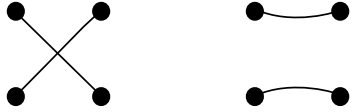}
 \caption{The transposition $(1,2)$ and the Weyl contraction $[1,2]$.}
 \label{Weyltransp}
\end{figure}

Let $(1, …, k)$ be the $k$-cycle in $\mathfrak{S}_k$ which sends $i$ on $i+1$ for any $i \in \{1,...,k-1\}$ and which sends $k$ on $1$. It is natural to consider the restriction of the geodesic order $\leq$ to the two sets $[\mathrm{id}_k, (1,...,k)] \cap \mathfrak{S}_k$ and $\mathcal{D}_k$: these can be identified as some well known ordered sets, repectively the non crossing partitions ${\sf NC}_k$ and the partitions ${\sf P}_k$ of $k$ elements. 

Using Theorem \ref{def:dist}, we see that the restriction of the distance $d$ on $\mathfrak{S}_k$ gives the usual Cayley distance on $\mathfrak{S}_k$ (Lemma $6.26$ of \cite{Levymaster}): the order is thus the usual geodesic order. The reader can understand why we used the name of Cayley graph in Definition \ref{Cayley}: if we forget about the partitions not in $\mathfrak{S}_k$ and replace the paths of length two which join two elements of $\mathfrak{S}_k$ by edges, we get the Cayley graph of $\mathfrak{S}_k$. P. Biane (\cite{Biane1}, \cite{Biane2}) showed that non-crossing partitions can be described using the geodesic condition in the Cayley graph of $\mathfrak{S}_k$: they are the elements in the geodesics between the identity and the $k$-cycle $(1,...,k)$. We recall that the order $\trianglelefteq$ was defined at the beginning of Section \ref{sec:geometry}. 

\begin{theorem}
\label{th:noncrossing}
The ordered set $([\mathrm{id}_k, (1,...,k)] \cap \mathfrak{S}_k, \leq)$ is isomorphic to $({\sf NC}_k, \trianglelefteq)$. 
\end{theorem}

 Let us consider the restriction of $\leq$ to $\mathcal{D}_k$. 

\begin{theorem}
\label{th:ident}
The ordered set $(\mathcal{D}_k, \leq )$ is isomorphic to $({\sf P}_k, \trianglelefteq )$. 
\end{theorem}

\begin{proof}
Let $p$ be a partition in $\mathcal{D}_k$. Let us consider $\phi(p)$, the partition in ${\sf P}_k$ which is obtained by considering the restriction of $p$ to $\{1,...,k\}$: if $p^1$, ..., $p^r$ are the blocks of $p$, 
\begin{align*}
\phi(p) = \{ p^i \cap \{1,...,k\} | i \in \{1,...,r\}\}. 
\end{align*}
It is straightforward to see that the application $\phi: \mathcal{D}_k \to {\sf P}_k$ is a bijection. Let $p$ and $p'$ two partitions in  $\mathcal{D}_k$. Let us prove that $p' \leq p$ if and only if $\phi(p') \trianglelefteq \phi(p)$. Using Theorem \ref{th:geocharact}, the condition $p' \leq p$ is equivalent to the fact that $p' \vee p \in {\sf Gl}(p)$ and $p' \in {\sf Sp}(p \vee p')$. But for any $p$ and $p'$ in $\mathcal{D}_k$, $ {\sf Gl}(p) = \{p\}$ and $p'$ is always in ${\sf Sp}(p \vee p')$. Thus, $p' \leq p$ is equivalent to $p' \vee p = p$ which is equivalent to $p' \trianglelefteq p$. It is easy to see that this condition is equivalent to $\phi(p') \trianglelefteq \phi(p)$. 
\end{proof}

Thus, both partitions and non-crossing partitions of $k$ elements can be seen as subset of an unique ordered set $(\mathcal{P}_k, \leq)$. 

\subsubsection{No Brauer element is smaller than a permutation}
In the following lemma, we show that the geodesics in the Cayley graph of $\mathcal{P}_k$ between two permutations either stay in the set of permutations or intersect $\mathcal{P}_k\setminus \mathcal{B}_k$. 

\begin{lemme}
\label{equageo}
Let $\sigma \in \mathfrak{S}_k$, then $[\mathrm{id}_k, \sigma] \cap \mathcal{B}_k = [\mathrm{id}_k, \sigma] \cap {\mathfrak{S}_k}$.
\end{lemme}

\begin{proof}
We do a proof by contradiction. Let $S \subset \mathfrak{S}_k$ be the set of permutations such that $[\mathrm{id}_k, \sigma]\cap{\mathcal{B}_k} \ne [\mathrm{id}_k, \sigma]\cap{\mathfrak{S}_k}.$ Let $\sigma \in S$ be a permutation such that $d(\mathrm{id}_k, \sigma) = \min\limits_{\sigma' \in S} d(\mathrm{id}_k,\sigma')$. Let us consider $b$ an element of $\mathcal{B}_k\setminus \mathfrak{S}_k$ such that $b \in [\mathrm{id}_k,\sigma]\cap{\mathcal{B}_k}$. There exists a geodesic in $\mathcal{B}_k$ which goes through $b$ and goes from $\mathrm{id}_k$ to $\sigma$. Let $b' \in \mathcal{B}_k$ be the unique element on this geodesic such that $d(\mathrm{id}_k, b')=1$. Let us remark that $b \in [b', \sigma]\cap{\mathcal{B}_k}$: this implies that $b'$ can not be a permutation. Indeed, if $b'$ was a permutation, then $[b', \sigma]\cap{\mathcal{B}_k} \ne [b',\sigma]\cap{\mathfrak{S}_k}$ and thus, $[\mathrm{id}_k, b'^{-1}\sigma]\cap{\mathcal{B}_k}\ne [\mathrm{id}_k, b'^{-1}\sigma]\cap {\mathfrak{S}_k} $. Yet $d(\mathrm{id}_k, b'^{-1}\sigma) = d(b',\sigma) = d(\mathrm{id}_k, \sigma)-1$. This would contradict the fact that $d(\mathrm{id}_k, \sigma) = \min_{\sigma' \in S} d(\mathrm{id}_k,\sigma')$. Thus $b'$ must be an element of $\mathcal{B}_k\setminus \mathfrak{S}_k$. Since $d(\mathrm{id}_k,b') = 1$, there exist $i$ and $j$ in $\{1,…,k\}$ such that $b'$ is equal to the Weyl contraction $[i,j]$ in $\mathcal{B}_k$. Thus there exist $i$ and $j$ in $\{1,…,k\}$ such that $[i,j] \in [\mathrm{id}_k,\sigma]$. Using Theorem \ref{th:geocharact}, this means that $[i,j] \vee \sigma \in {\sf Gl}(\sigma)$: $i$ and $j$ must be in the same cycle of $\sigma$. We can suppose that $\sigma$ is a cycle of size $k$. It is not difficult to see graphically that $[i,j] \vee \sigma = {\sf 0}_{k}$ where we recall that ${\sf 0}_k = \{1,...,k,1',...,k'\}$. By Theorem \ref{th:geocharact}, $[i,j] \in {\sf Sp}({\sf 0}_k)$: this is not possible since ${\sf nc}({\sf 0}_{k}) - {\sf nc}({\sf 0}_k \vee \mathrm{id}_k) = 0$ and ${\sf nc}([i,j]) - {\sf nc}([i,j] \vee \mathrm{id}_k) = 1$.
\end{proof}

\begin{remarque}
The same result holds if one replaces $\mathcal{B}_k$ by $\mathcal{B}s_k$. 
\end{remarque}

\subsubsection{Admissible splittings}
We will consider the interaction between Brauer elements and the notion of admissible splitting. 
\begin{lemme}
\label{pcoarsersigma}
Let $p \in \mathcal{P}_k$, the set ${\sf Sp}( p ) \cap \mathcal{B}_k$ is either empty or has exactly one element. In particular, for any $p \in \mathcal{B}_k$, ${\sf Sp}(p) = \{p\}$. 
\end{lemme}

In order to prove this lemma, an important remark is to see that no partition $p \in \mathcal{B}_k$ has a pivotal block. Indeed, let us suppose that a partition $p \in \mathcal{B}_k$ has a pivotal block that we will denote by $c$. We can always suppose that $p$ is irreducible. Since we can shuffle the columns of $p$ and take the transpose of $p$, we can always suppose that $c$ is of the form $\{i, (i + 1)'\}$ or $\{i, i + 1\}$. We can also suppose that when one cuts the block $c$, the new partition we get has the form $p_1 \otimes p_2$ where the notion of tensor product is defined in the next section and where $p_1 \in  \mathcal{P}_i$. The partition $p_1$ must be composed of blocks of size two except one block which is equal to $\{i\}$. This is not possible since $p_1$ must be a partition of $2i$ elements. 

\begin{proof}
Let us consider a partition $p \in \mathcal{P}_k$ and let $b \in \mathcal{B}_k$ such that $ b \in {\sf Sp}(p)$. We can suppose that $p$ is irreducible. Let us denote by ${\sf C}(b)$ the set of cycles of $b$.  By reversing the orientation of the paths of admissible splittings from $p$ to $b$, we get a path of gluings which goes from $b$ to $p$ and we see that there exists:
\begin{enumerate}
\item a covering tree of the graph $({\sf C}(b), {\sf C}(b)\times {\sf C}(b))$ which set of edges is denoted by $E$,
\item for any edge $(c,c') \in E$, a couple $(i_{c}, i_{c'})$ such that $i_{c} \in c$ and $i_{c'} \in c'$,
\end{enumerate} 
such that $p = b \vee (\{ \{i_c,i_{c'}\} | (c,c') \in E\} \cup \{\{ i\} | i \notin \cup_{(c,c') \in E} \{i_c, i_{c'}\} \})$. Using this equality, we see that the only admissible splittings that we can do is to cut the gluings between blocks of $b$: the only other possibility would be to cut a block of $b$ but $b$ does not have pivotal block, this splitting is not admissible. Thus, ${\sf Sp}(p) = \{b\}$. 
\end{proof}

This last lemma leads us to the following definition. 
\begin{definition}
\label{setbbarre}
For any positive integer $k$, we define: 
\begin{align*}
\overline{\mathfrak{S}_k} &= \{p \in \mathcal{P}_k \mid \#\left({\sf Sp}(p ) \cap \mathfrak{S}_k\right) = 1  \}, \\
\overline{\mathcal{B}_k} &=\{p \in \mathcal{P}_k \mid \#\left({\sf Sp}(p ) \cap \mathcal{B}_k\right) = 1  \}. 
\end{align*}
For any $p \in \overline{\mathcal{B}}_{k}$, we denote by ${\sf Mb}(p)$ the unique element in ${\sf Sp}(p ) \cap \mathcal{B}_k$. 
\end{definition} 

The Figure \ref{exemple2} gives an example of partition $p$ in $\overline{\mathfrak{S}_k}$. 
\begin{figure}[h!]
 \centering
  \includegraphics[width=180pt]{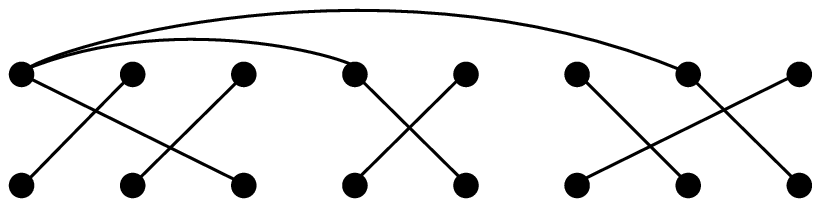}
 \caption{A partition $p$ such that $ {\sf Mb}(p)= (1,2,3)(4,5)(6,7,8) $.}
 \label{exemple2}
\end{figure}

\subsection{More structures on $\mathcal{P}_k$}
Using the fact that $\{1,...,k\} \cup \{1',...,k'\}$ is a disjoint union of two sets of equal cardinal, one can define operations that one could not on the general set $\mathcal{P}(X)$. The reference article for this section and the partition algebra is the article \cite{Halv} of T.Halverson and A.Ram. 

Let $k$ and $l$ be two non negative integers, and let $p \in \mathcal{P}_k$ and $p' \in \mathcal{P}_{l}$. Let us give some notions that one can define using a graphical construction: in each construction, the choice of the diagrams which represent the partitions does not matter.

\begin{description}
\item[Tensor product\!\!] Let us consider two diagrams: one associated with $p$, another with $p'$. Let $p\otimes p'$ be the partition in $\mathcal{P}_{k+l}$ associated with the diagram where one has put the diagram associated with $p$ on the left of the diagram associated with $p'$. 

\item[Transposition\!\!] The transposition of $p$, denoted by $\!\text{ }^{t}p$, is the partition obtained by permuting the role of $\{1,…,k\}$ and $\{1',…,k'\}$. For example if $k=3$, let $p=\big\{\{1,1',3'\}, \{2,3\}, \{2'\}\big\}$, then $\!\!\text{ }^{t}p =\big\{ \{1',1,3\}, \{2',3'\}, \{2\}\big\}$. For every diagram associated with $p$, the diagram obtained by flipping it according to a horizontal axis is a diagram associated with $\!\text{ }^{t}p$. 

\item[Multiplication $\circ$\!\!] Let us suppose that $k=l$. Let us put one diagram representing $p'$ above one diagram representing $p$. Let us identify the lower vertices of $p'$ with the upper vertices of $p$. We obtain a graph with vertices on three levels, then erase the vertices in the middle row, keeping the edges obtained by concatenation of edges passing through the deleted vertices. Any connected component entirely included in the middle row is then removed. Let us denote by $\kappa(p,p')$ the number of such connected components. We obtain an other diagram associated with a partition denoted by $p \circ p'$. An example is given in Figure \ref{fig:produit}.
\begin{figure}[h!]
 \centering
  \includegraphics[width=320pt]{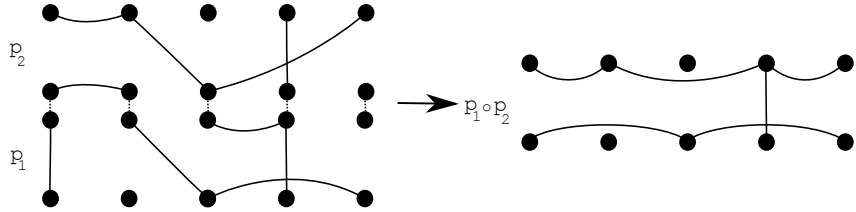}
 \caption{Partition $p_1\circ p_2$.}
 \label{fig:produit}
\end{figure}
\item[Number of erased loops\!\!] $\kappa(p,p')$ was defined when we defined the multiplication $p \circ p'$. It will be an important number in order to define other multiplications on partitions. 
\end{description}

We can extend these operations on $\mathbb{C}[\mathcal{P}_k]$ by linearity or bi-linearity. Before we define a modification of $\circ$, let us make some remarks on this product. The sets $\mathcal{B}_k$ and $\mathfrak{S}_k$ are stable by this concatenation operation. Recall that any permutation $\sigma \in \mathfrak{S}_k$ can be seen as a bijection from $\{1,…,k\}$ to itself. For any permutations $\sigma_1$ and $\sigma_2$, the bijection associated with $\sigma_1\circ \sigma_2$ is the composition of the two bijections associated with $\sigma_{1}$ and $\sigma_{2}$. Besides, the sub-algebra $\mathbb{C}[\mathfrak{S}_k]$ is not only stable for the $\circ$ operation: it also satisfies the following property which can be easily proved.
\begin{lemme}
\label{stablemieux}
Let $p$ and $p'$ be in $\mathcal{P}_k$, if $p \circ p' \in \mathfrak{S}_k$ then $p$ and $p'$ are in $\mathfrak{S}_k$. 
\end{lemme}

Besides, for any partition $\sigma \in \mathfrak{S}_k$ and any $p \in \mathcal{P}_k$, $\kappa(\sigma, p) = \kappa(p, \sigma)=0$. Let us remark that $\mathrm{id}_k$ is a right and left neutral element for $\circ$. At last, as a consequence of Lemma \ref{stablemieux}, since $\mathrm{id}_k \in \mathfrak{S}_k$, the only invertible elements of $\mathcal{P}_k$ are the permutations. The inverse of a permutation $\sigma$ is $\sigma^{-1}= \text{ }^{t}\sigma$. We can now recall the definition of the partition algebra $\mathbb{C}\left[\mathcal{P}_k(N)\right]$. From now on, $N$ is a positive integer. 

\begin{definition}
\label{multiplication}
The partition algebra $\mathbb{C}\left[\mathcal{P}_k(N)\right]$ is the associative algebra over $\mathbb{C}$ with basis $\mathcal{P}_k$ endowed with the multiplication defined by: 
\begin{align*}
\forall\ p_1,\ p_2 \in \mathcal{P}_k,\ \ p_1p_2 = N^{\kappa(p_1,p_2)} (p_1 \circ p_2). 
\end{align*}
\end{definition}

\subsection{Partitions and representation}
\label{sec:rep}
In this section, we recall a natural action of the partition algebra on $\left(\mathbb{C}^{N}\right)^{\otimes k}$ (for more explanations, \cite{Halv} of T.Halverson and A.Ram). This action will be useful in order to translate combinatorial properties into linear algebraic properties. 

\begin{definition}
\label{delta}
For any $p\in \mathcal{P}_{k}$ and any $k$-uples $(i_1, …, i_k)$ and $(i_{1'}, …, i_{k'})$ of elements of $\{ 1, …, N\}$, we set: 
\begin{align*}
p_{i_{1'}, …, i_{k'}}^{i_1, …, i_k} =
\left\{
    \begin{array}{ll}
       1, & \mbox{if for any two elements } r \mbox{ and } s \in \{1,…, k\}\cup\{1',…,k'\} \mbox{ which} \\& \mbox{are in the same block of } p,\mbox{one has } i_r = i_s,\\
        0, & \mbox{otherwise.}
    \end{array}
\right.
\end{align*}
\end{definition}

We can now define the action of the partition algebra $\mathbb{C}[\mathcal{P}_{k}(N)]$ on $\left(\mathbb{C}^{N}\right)^{\otimes k}$. Let $(e_1, …, e_N)$ be the canonical basis of $\mathbb{C}^{N}$. 

\begin{definition}\label{rep}
The action of the partition algebra $\mathbb{C}[\mathcal{P}_{k}(N)]$ on $\left(\mathbb{C}^{N}\right)^{\otimes k}$ is defined by the fact that for any $p \in \mathcal{P}_k$, for any $(i_1, …, i_k) \in \{1,…,N\}^{k}$:
\begin{align*}
p.(e_{i_1}\otimes…\otimes e_{i_k}) = \sum\limits_{(i_{1'}, …, i_{k'}) \in \{1, …, N\}^{k}} p_{i_{1'}, …, i_{k'}}^{i_1, …, i_k}\ \ e_{i_{1'}}\otimes…\otimes e_{i_{k'}}.
\end{align*}
\end{definition}

This action defines a representation of the partition algebra $\mathbb{C}[\mathcal{P}_{k}(N)]$ on $\left(\mathbb{C}^{N}\right)^{\otimes k}$ which we denote by $\rho_{N}$: 
\begin{align*}
\rho_{N}: \mathbb{C}[\mathcal{P}_{k}(N)] \mapsto {\sf End}\left(\left(\mathbb{C}^{N}\right)^{\otimes k}\right). 
\end{align*} 
Let us define $E_{i}^{j}$ be the matrix which sends $e_j$ on $e_i$ and any other element of the canonical basis on $0$. Let $p$ be a partition in $\mathcal{P}_k$. We can write the matrix of $\rho_{N}(p )$ in the basis $(e_{i_1}\otimes …\otimes e_{i_k})_{(i_l)_{l=1}^{k} \in \{1,…,N\}^{k}}$: 
\begin{align}
\label{forme}
\rho_{N}(p ) = \sum_{ (i_1, …, i_k, i_{1'}, …, i_{k'}) / p_{i_{1'}, …, i_{k'}}^{i_1, …, i_k}=1 } E_{i_{1'}}^{i_{1}} \otimes … \otimes E_{i_{k'}}^{i_{k}}. 
\end{align}
For example, if $p$ is the transposition $(1,2)$, then: 
\begin{align*}
\rho_{N}(( 1,2)) = \sum_{a,b=1}^{N} E_{a}^{b}\otimes E_{b}^{a}.
\end{align*} 
We think that this presentation allows us to understand, in an easier way, the representation $\rho_{N}$. In Figure \ref{exempleretrou}, we illustrate how to find the partition which representation is given by a sum of the form (\ref{forme}). The partition $p_1$ used in Figure \ref{exempleretrou} is the partition drawn in Figure \ref{fig:partition}.

\begin{figure}[h!]
 \centering
  \includegraphics[width=160pt]{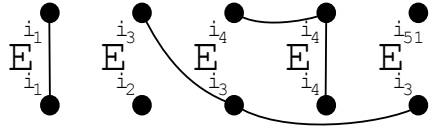}
 \caption{$\sum_{i_1, i_2, i_3, i_4, i_5} E_{i_1}^{i_1} \otimes E_{i_2}^{i_3} \otimes E_{i_3}^{i_4} \otimes E_{i_4}^{i_4} \otimes E_{i_3}^{i_5} = \rho_{N}(p)$.}
 \label{exempleretrou}
\end{figure}

Let us remark that the natural action of $\mathbb{C}[\mathcal{P}_k(N)]$ on $\left(\mathbb{C}^{N}\right)^{\otimes k}$ behaves well under the operation of tensor product: 
\begin{align*}
\rho_{N} (p\otimes p') = \rho_{N}(p )\otimes \rho_{N}(p'). 
\end{align*}

Let us suppose that $N \geq 2k$. Using Theorem $3.6$ in \cite{Halv}, the application $\rho_{N}$ is injective. Actually, if one considers only its restriction to the symmetric algebra or the Brauer algebra, it is enough to ask for $N\geq k$. For $N=k-1$ this result does not hold, this is a consequence of the Mandelstam's identity which asserts that: 
\begin{align*}
\sum_{\sigma \in \mathfrak{S}_{k}} {\epsilon(\sigma)} \rho_{k-1}(\sigma) = 0, 
\end{align*}
where $\epsilon(\sigma)$ is the signature of $\sigma$. In the following, we will often need that $\rho_N$ is injective, yet, for a sake of clarity, since we are concerned mainly with asymptotics when $N$ goes to infinity, we will ommit to specify each time that $N$ must be greater than $2k$ or $k$.

\subsection{The trace on $\mathcal{P}_k$}
Let $N$ be a positive integer. Depending on the context, we will consider a partition either as an element of $\mathcal{P}_k$ or as an element of ${\sf End}\left(\left(\mathbb{C}^{N}\right)^{\otimes k}\right)$ via the action defined in Definition \ref{rep}. We remind the reader that $(e_1, …, e_N)$ is the canonical base of $\mathbb{C}^{N}$. The family $\{ e_{i_1}\otimes…\otimes e_{i_k}, (i_1, …, i_k) \in \{1, …,N\}^{k}\}$ is a basis of $\left(\mathbb{C}^{N}\right)^{\otimes k}$: let ${\sf Tr}^{k}$ be the trace with respect to this canonical basis. We do not renormalize it, thus ${\sf Tr}^{k}\left(Id_{\left(\mathbb{C}^{N}\right)^{\otimes k}}\right) = N^{k}.$ We can define the trace of a partition. 

\begin{definition}
\label{trace}
Let $p$ be a partition in $\mathcal{P}_k$. We define: 
\begin{align*}
{\sf Tr}_N( p) = {\sf Tr}^{k}\left( \rho_{N}( p)\right). 
\end{align*}
For any integer $N$, we extend ${\sf Tr}_N$ by linearity to $\mathbb{C}[\mathcal{P}_k(N)]$. 
\end{definition}
When $N$ is explicit, we will only denote ${\sf Tr}_{N}$ by ${\sf Tr}$. The trace allows us to go from combinatorics arguments to linear algebra arguments since we have the following lemma. 

\begin{lemme}
Let $p$ be a partition in $\mathcal{P}_k$: 
\begin{align}
\label{lientracenc}
{\sf Tr}_N( p) = N^{{\sf nc}(p \vee \mathrm{id}_k)}. 
\end{align}
\end{lemme}

A generalisation of Equation (\ref{lientracenc}) is that for any partitions $p$ and $p'$, seen as elements of $\mathbb{C}[\mathcal{P}_k(N)]$, 
\begin{align}
\label{eq:lientrace2}
{\sf Tr}_N( p\ ^{t}\!p') = N^{\sf nc(p \vee p')}.
\end{align}
This equation will be used quite intensively. It is a consequence of Equation (\ref{lientracenc}) and the combinatorial equality: 
\begin{align}
\label{eq:lientrace3}
{\sf nc}((p \circ\ \!^{t}\!p') \vee \mathrm{id}_k ) + \kappa(p,\ \!^{t}\!p') = {\sf nc}(p \vee p'),
\end{align}
which can be understood by flipping the diagram of $\!\!\text{ }^{t}p$ over the one of $p'$: the flip transposes $\!\!\text{ }^{t} p$ thus we get the two diagrams of $p$ and $p'$ one over the other. By definition, the diagram constructed by putting a diagram representing $p'$ over one representing $p$ is associated with $p \vee p'$. But any block of this diagram either comes from a cycle of $p \circ\ \!^{t}\!p'$ or from a loop that we erased while doing the product $p \circ\ \!\!^{t}\!p'$.

\subsection{The exclusive basis of $\mathbb{C}[\mathcal{P}_k]$}
\label{exclusivebasis}
The basis used to define the partition algebra $\mathbb{C}[\mathcal{P}_k(N)]$ is quite natural, yet, it is not always very easy to work with. Indeed, if we look at the representation $\rho_{N}^{\mathcal{P}_k}$ of a partition, we see that the condition we used to define the delta function is not exclusive. It means that we did not use the following exclusive delta function: 
\begin{align*}
(p_{i_{1'}, …, i_{k'}}^{i_1, …, i_k})^{ex} =\left\{
    \begin{array}{ll}
       1, & \mbox{if for any two elements } r \mbox{ and } s \in \{1,…, k\}\cup\{1',…,k'\}, \\& i_r = i_s \mbox{ if and only if } r \mbox{ and } s \mbox{ are in the same block of } p,\\
        0, & \mbox{otherwise.}
    \end{array}
\right.
\end{align*}

By changing, in Definition \ref{rep}, the delta function defined in Definition \ref{delta} by this new exclusive delta function, we define a new function: $$\tilde{\rho}_{N}: \mathbb{C}[\mathcal{P}_k(N)] \to {\sf End}\left(\left(\mathbb{C}^{N}\right)^{\otimes k}\right).$$

Does it exist, for any partition $p\in \mathcal{P}_k$ an element $p^{c} \in \mathbb{C}[\mathcal{P}_k]$ such that for any integer $N$, $\rho_{N}^{\mathcal{P}_k}( p^{c}) = \tilde{\rho}_{N}^{\mathcal{P}_{k}}( p)$ ? The answer is given by the following definition, as explained by Equation $(2.3)$ of \cite{Halv}. 

\begin{definition}
\label{cop}
The exclusive partition basis, denoted by $(p^{c})_{p \in \mathcal{P}_k}$, is the unique family of elements in $\mathbb{C}[\mathcal{P}_k]$ defined by the relation: 
\begin{align*}
p = \sum_{p \trianglelefteq p'} p'^{c}. 
\end{align*} 
\end{definition}
The notion of being coarser defines a partial order on $\mathcal{P}_k$: the relation can be inverted. The family $(p^{c})_{p \in \mathcal{P}_k}$ is well defined and it is a basis of the partition algebra $\mathbb{C}[\mathcal{P}_{k}]$. It satisfies that  for any partition $p \in \mathcal{P}_k$, 
$$\rho_{N}( p^{c}) = \tilde{\rho}_{N}( p).$$

\subsection{Geodesics and tensor product}

The following property, known for $\mathfrak{S}_k$ and $\mathcal{B}_k$, is still true for $\mathcal{P}_k$: a geodesic between $\mathrm{id}_k$ and $p_1\otimes p_2$ must be the tensor product of a geodesic between $\mathrm{id}_k$ and $p_1$ and a geodesic between $\mathrm{id}_k$ and $p_2$. 

\begin{proposition}
\label{multiplicativgeo}
Let $k$ and $l$ be two positive integers. Let $p_1 \in \mathcal{P}_k$ and $ p_2\in \mathcal{P}_l$. For any $p' \in \mathcal{P}_{k+l}$ such that $p' \leq p_1\otimes p_2$, there exist $p'_1 \leq p_1$ and $p'_2 \leq  p_2$ such that $p' = p'_1 \otimes p'_2$. 
\end{proposition}

\begin{proof}
Let us suppose that $p' \leq p_1\otimes p_2$. Using Theorem \ref{th:geocharact}, $p' \vee (p_1 \otimes p_2)$ is an admissible gluing of $p_1 \otimes p_2$. This proves that there exist $p'_1 \in \mathcal{P}_k$ and $p'_2 \in \mathcal{P}_l$ such that $p' = p'_1 \otimes p'_2$. Since: 
\begin{align*}
{\sf df}(p'_1 \otimes p'_2, p_1 \otimes p_2) = {\sf df}(p'_1, p_1) + {\sf df}(p'_2, p_2), 
\end{align*} 
we get that $p'_1 \leq p_1$ and $p'_2 \leq  p_2$. 
\end{proof}

We have seen the consequences that $p' \leq p_1 \otimes p_2$. What are the consequences that $p_1\otimes p_2 \leq p'$ ? We will answer this question only when $p_1 \otimes p_2$ is finer than $p'$. This case is important in the theory of random matrices which are invariant in law by conjugation by the symmetric group (\cite{Gab2}, \cite{Gab3}). In order to answer, we need to introducte the notions of left- and right-parts of a partition $p$. Let $k_1$ and $k_2$ be two positive integers and recall the notion of extraction that we defined before Definition \ref{supportpart}.

\begin{definition}
Let $p \in \mathcal{P}_{k_1+k_2}$, we denote by $p_{k_1}^{l}$ the extraction of $p$ to $\{1,...,k_1\}$ and $p_{k_1}^{r}$ the extraction of $p$ to $\{k_1+1,...,k_1+k_2\}$. The left-part of $p$, namely $p_{k_1}^{l}$, is in $\mathcal{P}_{k_1}$ and the right-part of $p$, namely $p_{k_2}^{r}$, is in $\mathcal{P}_{k_2}$.
 \end{definition}

\begin{proposition}
Let $k_1$ and $k_2$ be two positive integers and let $k = k_1+k_2$. Let $p$ be an element of $\mathcal{P}_{k}$. Let $p_1$ and $p_2$ be respectively in $\mathcal{P}_{k_1}$ and $\mathcal{P}_{k_2}$ such that $p_1 \otimes p_2$ is finer than $p$. Then: 
\begin{align}
\label{eq:deftensor}
{\sf df}( p_1 \otimes p_2, p) = {\sf df}( p_1, p_{k_1}^{l}) +{\sf df}( p_2, p_{k_1}^{r}) +{\sf df}( p_{k_1}^{l} \otimes p_{k_1}^{r} , p). 
\end{align}
In particular, we have equivalence between: 
 \begin{enumerate}
 \item $p_1\otimes p_2 \sqsupset p$, 
 \item $p_{1} \sqsupset p_{k_1}^{l}$, $p_{2} \sqsupset p^{r}_{k_1}$ and $p^{l}_{k_1} \otimes p^{r}_{k_1} \sqsupset p$. 
 \end{enumerate}
\end{proposition}

\begin{proof}
The Equation (\ref{eq:deftensor}) can be proved by a simple calculation, using the fact that $p_1\otimes p_2$ is finer than $p$ and the fact that for any partition $u$ and $v$, respectively in $\mathcal{P}_{k_1}$ and $\mathcal{P}_{k_2}$, 
\begin{align*}
 {\sf nc}( u\otimes v) &=  {\sf nc}( u) +  {\sf nc}( v), \\
 {\sf nc}( (u \otimes v )\vee \mathrm{id}_k) &= {\sf nc}(u \vee \mathrm{id}_{k_1}) + {\sf nc}( v \vee \mathrm{id}_{k_2}).
\end{align*}
The equivalences are consequences of  Equation (\ref{eq:deftensor}) and Lemma \ref{lemme:carac1}.
\end{proof}

\subsection{The Kreweras complement for $\mathcal{P}_k$}
\label{sec:Kreweras}
Let us consider $\sigma$ and $\sigma'$ two permutations in $\mathfrak{S}_k$. Since $d$ is a distance: 
\begin{align*}
d(\mathrm{id}_k, \sigma) \leq d(\mathrm{id}_k, \sigma') + d(\sigma', \sigma). 
\end{align*}
Any permutation has an inverse and the restriction of $d$ to $\mathfrak{S}_k$ is invariant by left or right multiplication by a permutation. The last inequality is then equivalent to: 
\begin{align}
\label{ineq:trian}
d(\mathrm{id}_k, \sigma) \leq d(\mathrm{id}_k, \sigma') + d(\mathrm{id}_k, \sigma'^{-1}\sigma).
\end{align}
Thus $\sigma'\leq \sigma$ if and only if there exists $\tilde{\sigma}$ such that $\sigma= \sigma' \tilde{\sigma}$ and $d(\mathrm{id}_k, \sigma) = d(\mathrm{id}_k, \sigma') + d(\mathrm{id}_k, \tilde{\sigma}).$ We recall that $\tilde{\sigma}$ is then unique and it is called the {\em Kreweras complement} of $\sigma'$ in $\sigma$, denoted in this article ${\sf K}_{\sigma}(\sigma')$ (\cite{kreweras}, \cite{simionullman}). In general, any partition does not have any inverse. It is natural to wonder if an inequality of the form (\ref{ineq:trian}) holds and if so, it is natural to wonder if we can use it to define a Kreweras complement. The answer to the first question is given by the following theorem. 

\subsubsection{A new inequality}
Let $p$ and $p'$ be two partitions in $\mathcal{P}_k$,
\begin{theorem}
We have 
\begin{align}
\label{autreequation}
d(\mathrm{id}_k, p\circ p') \leq d(\mathrm{id}_k, p) + d(\mathrm{id}_k, p') - \frac{k+{\sf nc}(p\circ p') - {\sf nc}( p) - {\sf nc}(p')}{2}-\kappa(p,p'). 
\end{align} 
\end{theorem}

As we did for the triangular inequality, we define a new defect.

\begin{definition}
\label{defect}
The $\prec$-defect $\eta(p,p')$ is: 
\begin{align*}
\eta(p,p') = d(\mathrm{id}_k,p) + d(\mathrm{id}_k,p') - d(\mathrm{id}_k, p\circ p') -\frac{k+{\sf nc}(p\circ p') - {\sf nc}( p)- {\sf nc}(p')}{2} - \kappa(p,p').
\end{align*}
\end{definition}

Using a simple calculation, we can give an other form to the $\prec$-defect. 
\begin{lemme}
For any $p$ and $p'$ in $\mathcal{P}_k$, $\eta(p,p')$ is equal to: 
\begin{align}
\label{definitiondefect2}
 {\sf nc}( p)  - {\sf nc}(p \vee \mathrm{id}_k)  + {\sf nc}(p') - {\sf nc}(p' \vee  \mathrm{id}_k) - {\sf nc}( p \circ p') + {\sf nc}( p \circ p' \vee  \mathrm{id}_k ) - \kappa(p,p'). 
\end{align}
\end{lemme}

For now, we do not know if $\eta(p,p') \geq 0$ but we can express $\eta(p,p')$ using the defect for $\leq$. 

\begin{theorem}
\label{th:dautresvaleurs}
Let $p_0$, $p_1$ and $p_2$ be three partitions in $\mathcal{P}_k$, the following quantities are equal: 
\begin{enumerate}
\item ${\sf df}(p_1 \circ p_2, p_0) + \eta(p_1, p_2), $
\item$ {\sf df}(p_1,p_0) + {\sf df}(p_2, {}^{t}p_1 \circ p_0),$ 
\item ${\sf df}(p_1, p_0\circ {}^{t} p_2) + {\sf df}(p_2, p_0), $
\item ${\sf df}(p_1 \otimes p_2, (p_0 \otimes \mathrm{id}_k) \tau),$ 
\end{enumerate}
where $\tau$ is the permutation in $\mathfrak{S}_{2k}$ equal to $(1,k+1)(2,k+2)...(k,2k)$. In particular, taking $p_0 = p_1 \otimes p_2$: 
\begin{align*}
\eta(p_1,p_2) &= {\sf df}(p_1,p_1\circ p_2) + {\sf df}(p_2, { }^{t}p_1 \circ p_1 \circ p_2) 
\\&= {\sf df}(p_1, p_1\circ p_2 \circ { }^{t} p_2) + {\sf df}(p_2, p_1 \circ p_2) 
\\&= {\sf df}(p_1 \otimes p_2, ((p_1 \circ p_2) \otimes  \mathrm{id}_k) \tau).
\end{align*}
\end{theorem}

\begin{proof}
The proof is done by doing calculations and using intensively the Equation (\ref{eq:lientrace2}). From now on, we will take the following convention: the $\circ$ and $\otimes$ operations are done first before any $\vee$ operation, thus $p_1 \circ p_2 \vee p_3 \otimes p_4$ stands for $(p_1 \circ p_2) \vee (p_3 \otimes p_4).$

Let us expand ${\sf df}(p_1 \circ p_2, p_0) + \eta(p_1, p_2)$. Using Equations (\ref{definitiondefect}) and (\ref{definitiondefect2}), ${\sf df}(p_1 \circ p_2, p_0) + \eta(p_1, p_2)$ is equal to: 
\begin{align*}
\!\!\!\!\!\!\!\!\!\!\!\!\!\!\!\!\!\!\!\!\!\!\!\!\!\!\!\!\!\!\!\!\!\!\!\! {\sf nc}(p_1) - {\sf nc}(p_1 \vee \mathrm{id}_k) +  {\sf nc}(p_2) &- {\sf nc}(p_2 \vee \mathrm{id}_k)+ {\sf nc}(p_0 \vee \mathrm{id}_k) \\
& - {\sf nc}(p_1 \circ p_2 \vee p_0) - \kappa (p_1, p_2).
\end{align*}
 
In a similar way $ {\sf df}(p_1,p_0) + {\sf df}(p_2, {}^{t}p_1 \circ p_0)$ is equal to: 
\begin{align*}
{\sf nc}(p_1) - {\sf nc}(p_1 \vee \mathrm{id}_k) &+  {\sf nc}(p_2)- {\sf nc}(p_2 \vee \mathrm{id}_k)+ {\sf nc}(p_0 \vee \mathrm{id}_k) \\
&-{\sf nc}(p_1 \vee p_0 ) -{\sf nc}(p_2 \vee {}^{t}p_1 \circ p_0 ) +{\sf nc}({}^{t}p_1 \circ p_0 \vee  \mathrm{id}_k ). 
\end{align*}

For ${\sf df}(p_1, p_0\circ {}^{t} p_2) + {\sf df}(p_2, p_0),$ we get: 
\begin{align*}
{\sf nc}(p_1) - {\sf nc}(p_1 \vee \mathrm{id}_k) &+  {\sf nc}(p_2)- {\sf nc}(p_2 \vee \mathrm{id}_k)+ {\sf nc}(p_0 \vee \mathrm{id}_k) \\
&- {\sf nc}(p_1 \vee p_0 \circ { }^{t}p_2) + {\sf nc}(p_0 \circ { }^{t}p_2 \vee  \mathrm{id}_k) - {\sf nc}(p_2 \vee p_0).
\end{align*}

For  ${\sf df}(p_1 \otimes p_2, (p_0 \otimes \mathrm{id}_k) \tau),$  we get: 
\begin{align*}
{\sf nc}(p_1\otimes p_2) - {\sf nc}(p_1\otimes p_2 \vee \mathrm{id}_{2k} ) - {\sf nc}(p_1\otimes p_2 \vee (p_0 \otimes \mathrm{id}_k) \tau ) + {\sf nc}((p_0 \otimes \mathrm{id}_k) \tau \vee  \mathrm{id}_{2k}), 
\end{align*}
but ${\sf nc}(p_1 \otimes p_2) = {\sf nc}(p_1) +  {\sf nc}(p_2)$, ${\sf nc}(p_1\otimes p_2 \vee \mathrm{id}_{2k} ) = {\sf nc}(p_1 \vee \mathrm{id}_{k} ) +  {\sf nc}(p_2 \vee \mathrm{id}_{k} )$ and $ {\sf nc}((p_0 \otimes \mathrm{id}_k) \tau \vee  \mathrm{id}_{2k}) = {\sf nc}(p_0 \vee  \mathrm{id}_k)$. Thus, ${\sf df}(p_1 \otimes p_2, (p_0 \otimes \mathrm{id}_k) \tau),$ is equal to: 

\begin{align*}
{\sf nc}(p_1) - {\sf nc}(p_1 \vee \mathrm{id}_k) +  {\sf nc}(p_2)- {\sf nc}(p_2 \vee \mathrm{id}_k)&+ {\sf nc}(p_0 \vee \mathrm{id}_k) \\
&-{\sf nc}(p_1 \otimes p_2 \vee (p_0 \otimes \mathrm{id}_k) \tau). 
\end{align*}

The first lines are all equal, thus it remains to prove that the following numbers are equal: 
\begin{enumerate}
\item $ - {\sf nc}(p_1 \circ p_2 \vee p_0) - \kappa (p_1, p_2)$,
\item$-{\sf nc}(p_1 \vee p_0 ) -{\sf nc}(p_2 \vee {}^{t}p_1 \circ p_0 ) +{\sf nc}({}^{t}p_1 \circ p_0 \vee  \mathrm{id}_k ) $,
\item$- {\sf nc}(p_1 \vee p_0 \circ { }^{t}p_2) + {\sf nc}(p_0 \circ { }^{t}p_2 \vee  \mathrm{id}_k) - {\sf nc}(p_2 \vee p_0)$,
\item $-{\sf nc}(p_1 \otimes p_2 \vee (p_0 \otimes \mathrm{id}_k) \tau).$
\end{enumerate}

Let us prove that the first one and the second one are equal. The idea is to use Equation (\ref{eq:lientrace2}) and the fact that we need to prove: 
\begin{align*}
N^{ {\sf nc}(p_1 \vee p_0 )+{\sf nc}(p_2 \vee {}^{t}p_1 \circ p_0 ) } = N^{ {\sf nc}(p_1 \circ p_2 \vee p_0) + \kappa (p_1, p_2) +{\sf nc}({}^{t}p_1 \circ p_0 \vee  \mathrm{id}_k )}.
\end{align*}

We have: 
\begin{align*}
N^{{\sf nc}(p_1 \vee p_0 )+{\sf nc}(p_2 \vee {}^{t}p_1 \circ p_0 ) } &= {\sf Tr}(p_1 \ { }^{t}p_0) {\sf Tr}(p_2\ ({ }^{t}p_0 \circ p_1))\\
&= {\sf Tr}(p_1 \ { }^{t}p_0) {\sf Tr}(({ }^{t}p_0 \circ p_1) \circ p_2) N^{\kappa(p_1,p_2)}\\
&= {\sf Tr}({ }^{t}p_0\ p_1 ) {\sf Tr}({ }^{t}p_0 \circ (p_1 \circ p_2)) N^{\kappa(p_1,p_2)}\\
&={\sf Tr}({ }^{t}p_0 \circ  p_1 ) N^{ \kappa({ }^{t}p_0, p_1)} {\sf Tr}({ }^{t}p_0 \circ (p_1 \circ p_2))N^{\kappa(p_1,p_2)} \\
&={\sf Tr}({ }^{t}p_0 \circ  p_1 )  {\sf Tr}({ }^{t}p_0  (p_1 \circ p_2))N^{\kappa(p_1,p_2)}\\
&= N^{{\sf nc}({}^{t}p_1 \circ p_0 \vee  \mathrm{id}_k )} N^{ {\sf nc}(p_1 \circ p_2 \vee p_0)} N^{\kappa(p_1,p_2)}, 
\end{align*}
which allows us to conclude. In order to prove that the third element is equal to the first one, the proof is similar. Let us show that the first and the forth are equal: 
\begin{align*}
N^{{\sf nc}(p_1 \otimes p_2 \vee (p_0 \otimes \mathrm{id}_k) \tau)} &= {\sf Tr}((p_1 \otimes p_2)\ { }^{t} ((p_0 \otimes \mathrm{id}_k)\tau))\\&= {\sf Tr}(p_1p_2\text{ }^{t}p_0) \\
&=N^{\kappa(p_1,p_2)} {\sf Tr}((p_1\circ p_2)\text{ }^{t}p_0)\\
&=N^{\kappa(p_1,p_2)+ {\sf nc}(p_1 \circ p_2 \vee p_0)}. 
\end{align*}
This concludes the proof. 
\end{proof}

This theorem allows us to prove Inequality (\ref{autreequation}): since the defect ${\sf df}$ is always non negative, $\eta$ which is a sum of two defects for $\leq$ is also non negative. In the first place, Inequality (\ref{autreequation}) was obtained as a consequence of the triangle inequality for $d$ and an inequality between $d(p,p\circ p')$ and $d(\mathrm{id}_k, p')$ given by the following proposition. This inequality generalize the invariance of the restiction of $d$ to $\mathfrak{S}_k$ by left or right multiplication by a permutation. 
\begin{proposition}
\label{presqueinv}
Let $p$ and $p'$ in $\mathcal{P}_k$, we have the following inequality: 
\begin{align*}
d(p, p\circ p') \leq d(\mathrm{id}_k, p') - \frac{k+{\sf nc}(p\circ p') - {\sf nc}( p) - {\sf nc}(p')}{2}-\kappa(p,p').
\end{align*}
\end{proposition}

\begin{proof}
It is a consequence of the triangle inequality: 
\begin{align}
\label{equ3}
d\left(p \otimes \mathrm{id}_k, \left((p \circ p') \otimes \mathrm{id}_k\right) \tau\right)\leq d(p \otimes \mathrm{id}_k, p \otimes p') + d\left(p \otimes p', \left((p \circ p') \otimes \mathrm{id}_k\right) \tau\right).
\end{align}
and computations similar to what we did in order to prove Theorem \ref{th:dautresvaleurs}. 
\end{proof}

\subsubsection{The Kreweras complement}

\begin{definition}
\label{geo2}
Let $p$ and $p'$ be two elements of $\mathcal{P}_{k}$. We will say that $p'$ is an {\em admissible prefixe} of $p$ if and only if: 
\begin{enumerate}
\item there exists $p''$ such that $p = p' \circ p''$, 
\item we have the equality: 
\begin{align*}
d(\mathrm{id}_k,p) = d(\mathrm{id}_k,p') + d(\mathrm{id}_k,p'') - \frac{k+{\sf nc}( p)-{\sf nc}( p') - {\sf nc}(p'')}{2} - \kappa(p',p''). 
\end{align*}
\end{enumerate}
It $p'$ is an admissible prefixe of $p$, we write $p' \prec p$ and the set of $p''$ which satisfy $1.$ and $2.$ is called the Kreweras complement of $p'$ in $p$. We denote it ${\sf K}_{p}( p')$. 
\end{definition}

If $p' \prec p$ does not hold, the Kreweras complement of $p'$ in $p$ is set to be equal to $\emptyset$. Now, let us suppose that $p' \prec p $. The set ${\sf K}_{p}( p')$ is not empty but in general it is not reduced to a unique partition. For example, one can show that if $p' = \{\{1,2,1',2'\}\}$ and $p = \{\{1',2'\},\{1\}, \{2\}\}$ then: 
\begin{align*}
{\sf K}_{p}( p') = \left\{\big\{\{1\},\{2\},\{1'\},\{2'\}\big\}, \big\{\{1\},\{2\},\{1',2'\}\big\}\right\}. 
\end{align*}

Let us remark that for any $\sigma \in \mathfrak{S}_k$, $\{\sigma' \in \mathfrak{S}_k, \sigma' \prec \sigma\} = [\mathrm{id}_k,\sigma]\cap{\mathfrak{S}_k}$. This is due to the fact that $\kappa(\sigma,\sigma') = 0$ for any couple of permutations, the fact that ${\sf nc}$ is constant on the set of permutations and the fact that any permutation is invertible. Using similar arguments and Lemma \ref{stablemieux}, one can have the better result. 
\begin{lemme}
\label{geoamperm}
Let $\sigma \in \mathfrak{S}_k$: 
\begin{align*}
\{p \in \mathcal{P}_k, p \prec \sigma\} = [\mathrm{id}_k, \sigma]\cap{\mathfrak{S}_k}. 
\end{align*}
Besides, for any $p \in \mathcal{P}_k$, $\{\sigma\in \mathfrak{S}_k, \sigma \prec p\} = [\mathrm{id}_k, p] \cap \mathfrak{S}_k$. 
\end{lemme}

We have seen, in Theorem \ref{th:noncrossing}, that the poset of non-crossing partitions over $\{1, …, k\}$ is isomorphic to $(\left[\mathrm{id}_k, (1,…,k)\right]\cap \mathfrak{S}_k, \leq)$. From now on, we will consider any non-crossing partition over $\{1, …, k\}$ as an element of $\left[\mathrm{id}_k, (1,..,k)\right]\cap \mathfrak{S}_k$. Our notion generalizes the usual notion of Kreweras complement since for any $\sigma \in [\mathrm{id},(1,...,k)]\cap{\mathfrak{S}_k}$, ${\sf K}_{(1,...,k)}( \sigma)$ is the Kreweras complement of the non-crossing partition corresponding to $\sigma$. In the general case, as a direct consequence of the definitions and Theorem \ref{th:dautresvaleurs}, we get the following theorem. 

\begin{theorem}
\label{th:caractK}
Let $p_0$, $p_1$ and $p_2$ be three partitions in $\mathcal{P}_k$. The following assertions are equivalent: 
\begin{enumerate}
\item $p_1\circ p_2 \leq p_0$ and $p_2 \in {\sf K}_{p_1\circ p_2}(p_1)$, 
\item $p_1 \leq p_0 $ and $p_2 \leq ^{t}\!p_1 \circ p_0$, 
\item $p_1 \leq p_0 \circ ^{t}\! p_2$ and $p_2 \leq p_0$,
\item $p_1 \otimes p_2 \leq (p_0 \otimes \mathrm{id}_k) \tau$.
\end{enumerate}
Thus, by specifying $p_0= p_1\circ p_2$, we have the equivalences: 
\begin{enumerate}
\item $p_2 \in {\sf K}_{p_1\circ p_2}(p_1)$, 
\item $p_1 \leq p_1 \circ p_2$ and $p_2 \leq ^{t}\!p_1 \circ p_1 \circ p_2$, 
\item $p_1 \leq p_1 \circ p_2 \circ ^{t}\! p_2$ and $p_2 \leq p_1 \circ p_2$. 
\item $p_1\otimes p_2 \leq (p_1\circ p_2 \otimes \mathrm{id}_k) \tau$.
\end{enumerate}
\end{theorem}

\subsubsection{The order $\prec$}

In the previous section, we have defined a relation $\prec$ that we are going to study a little in this section. As explained in the beginning of Section \ref{sec:Kreweras}, $\prec$ and $\leq$ are equivalent when we restrict them to $\mathfrak{S}_k$. In general, we have the following result. 

\begin{proposition}
\label{dansgeo}
For any partitions $p$ and $p'$ in  $\mathcal{P}_k$,  $p' \prec p$ implies that $p' \leq p$.
\end{proposition}

\begin{proof}
Indeed, if $p' \prec p$, there exists $p''$ such that $p=p' \circ p''$ and $p'' \in {\sf K}_{p' \circ p''}(p')$. The result is then a consequence of Theorem \ref{th:caractK}.
\end{proof}

This will allow us to prove the following result. 

\begin{theorem}
The relation $\prec$ defines a new order on $\mathcal{P}_k$. 
\end{theorem}

\begin{proof}
The fact that $\prec$ is reflexive is due to the fact that for any $p$, $p$ can be decomposed as $p \circ \mathrm{id}_k$: using the fact that $\kappa(p,\mathrm{id}_k)=0$, we see that $p'$ is a admissible prefixe of $p$. The antisymmetry is a consequence of Proposition \ref{dansgeo} and the antisymmetry of $\leq$. It remains to prove the transitivity which is proved in the upcoming Proposition \ref{prop:transprec}. At last, we can see that $\prec$ is not equal to $\leq$ since it is easy to find some $p$ and $p'$ in $\mathcal{P}_k$ such that $p' \leq p$ but $p$ can not be decomposed as $p \circ p''$. 
\end{proof}

Let us state a consequence of Proposition \ref{dansgeo}: the factorization property for $\prec$. 
\begin{lemme}
\label{factocanard}
Let $k$ and $l$ be two positive integers. Let $p_1 \in \mathcal{P}_k$ and $ p_2\in \mathcal{P}_l$. For any $p' \in \mathcal{P}_{k+l}$ such that $p' \prec p_1\otimes p_2$, there exist $p'_1 \prec p_1$ and $p'_2 \prec p_2$ such that $p' = p'_1 \otimes p'_2$. 
\end{lemme}

\begin{proof}
It is a consequence of  Proposition \ref{dansgeo} and the factorization property for the geodesics stated in Proposition \ref{multiplicativgeo}. 
\end{proof}
 
\section{Structures on $\left(\bigoplus_{k=0}^{\infty} \mathbb{C}[\mathcal{P}_k]\right)^{*}$}
\label{sec:structure}

Using the definitions of $\leq$, $\prec$ and the Kreweras complement, we can define some new structure on $\left(\bigoplus_{k}^{\infty} \mathbb{C}[\mathcal{P}_k]\right)^{*}$ which have a similar flavor than the structures constructed in \cite{mastnaknica} for linear forms on permutations. Let us remark that we have the canonical isomorphisms: 
\begin{align*}
\left(\bigoplus_{k=0}^{\infty} \mathbb{C}[\mathcal{P}_k]\right)^{*} \simeq \prod_{k=0}^{\infty}\mathbb{C}[\mathcal{P}_k]\simeq \bigcup\limits_{k =0}^{\infty} \mathbb{C}^{\mathcal{P}_k}. 
\end{align*}
Thus, the structures we are going to define on $\left(\bigoplus_{k=0}^{\infty} \mathbb{C}[\mathcal{P}_k]\right)^{*}$ can also be seen on the two other spaces. 

\begin{remarque}
\label{rq:restrictionak}
It is important to notice, when reading this section, that most of the definitions or notions and results (in fact all except the one about characters, infinitesimal characters and $\boxplus$-convolutions) can be restricted to the space of linear forms $(\mathbb{C}[\mathcal{P}_k])^{*}$: this will be useful in Section \ref{sec:obsconv}. 
\end{remarque}

\subsection{Some transformations} 

\begin{definition}
\label{def:moment}
The $\mathcal{M}$-transform is the application $\mathcal{M}: \left(\bigoplus_{k=0}^{\infty} \mathbb{C}[\mathcal{P}_k]\right)^{*} \to \left(\bigoplus_{k=0}^{\infty} \mathbb{C}[\mathcal{P}_k]\right)^{*}$ such that for any $\phi \in \left(\bigoplus_{k=0}^{\infty} \mathbb{C}[\mathcal{P}_k]\right)^{*}$, any $k$ and any $p \in \mathcal{P}_k$: 
\begin{align}
\label{eq:mtransfo}
(\mathcal{M}(\phi))(p) = \sum_{p' \leq p} \phi(p').
\end{align}
\end{definition}

Let us consider a positive integer $k$. The application $\mathcal{M}$ can be restricted to the space $(\mathbb{C}[\mathcal{P}_k])^{*}$ seen as a subspace of $\left(\bigoplus_{k=0}^{\infty} \mathbb{C}[\mathcal{P}_k]\right)^{*}$. But the space $(\mathbb{C}[\mathcal{P}_k])^{*}$ is isomorphic to $\mathbb{C}[\mathcal{P}_k]$ or $\mathbb{C}^{\mathcal{P}_k}$. The $\mathcal{M}$-transform induces an application on $\mathbb{C}^{\mathcal{P}_k}$ which is only the multiplication by the matrix $G$ of the geodesic order $\leq$. Since $G = CS$ (Theorem \ref{th:lienmatriciel}), it is natural to consider the $\mathcal{M}^{\to c}$-transform (resp. the $\mathcal{M}^{c\to}$-transform) which is defined in the same way as $\mathcal{M}$ but using the order $\sqsupset$ (resp. $\dashv$) in Equation (\ref{eq:mtransfo}). The equation $G=CS$ then implies the following equality: 
\begin{align}
\label{eq:decompo}
\mathcal{M} = \mathcal{M}^{c \to}\mathcal{M}^{\to c}. 
\end{align}

Since $G$ is invertible, the $\mathcal{M}$-transform is a bijection: we can consider its inverse.

\begin{definition}
The $\mathcal{R}$-transform is the inverse of the $\mathcal{M}$-transform: $\mathcal{R} = \mathcal{M}^{-1}.$
\end{definition}

\subsection{Convolutions}
In this section, we will only consider the subspace of linear forms which are invariant by the action of the permutations, namely $\left(\left(\bigoplus_{k=0}^{\infty} \mathbb{C}[\mathcal{P}_k]\right)^{*}\right)^{\mathfrak{S}}$, which is equal to: 
\begin{align*}
 \left\{ \phi \in\left(\bigoplus_{k=0}^{\infty} \mathbb{C}[\mathcal{P}_k]\right)^{*} |\ \forall k \in \mathbb{N}, \forall \sigma \in \mathfrak{S}_k, \forall p \in \mathcal{P}_{k}, \phi(\sigma \circ p \circ \sigma^{-1}) = \phi(p) \right\},
\end{align*}
or 
\begin{align*}
 \left( \bigoplus_{k=0}^{\infty} \mathbb{C}[\mathcal{P}_k/ \mathfrak{S}_k]\right)^{*}, 
\end{align*}
where $\mathcal{P}_k/\mathfrak{S}_k = \{ \{ \sigma\circ p \circ \sigma^{-1}, \sigma \in \mathfrak{S}_k\}, p \in \mathcal{P}_k\}$.  We do so since it allows us to clarify the explanations, but the same results hold for the general space $\left(\bigoplus_{k=0}^{\infty} \mathbb{C}[\mathcal{P}_k]\right)^{*}$ by slightly changing the definitions. From now on, $p$ is a partition and ${\bold p}$ or $[p]$ are the orbit of $p$ in $\mathcal{P}_k/ \mathfrak{S}_k$. 

Let us remark that $\bigoplus_{k=0}^{\infty}\mathbb{C}[\mathcal{P}_k/\mathfrak{S}_k]$ can be endowed with the multiplication given by $([p],[p']) \to [p \otimes p']$. This multiplication, that we denote also by $\otimes$, does not depend of the choice or the representant of the orbits and its neutral element is given by $\{\emptyset\}$. Besides, the $\mathcal{M}$, $\mathcal{R}$, $\mathcal{M}^{\to c}$ and $\mathcal{M}^{c \to }$ are invariant by the action of the symmetric group: they are well defined on $\left( \bigoplus_{k=0}^{\infty} \mathbb{C}[\mathcal{P}_k/ \mathfrak{S}_k]\right)^{*}$. 

\subsubsection{Additive structure}
\begin{definition}
Let $\bold{p}$ be in $\mathcal{P}_k/ \mathfrak{S}_k$. We can suppose that $p$ has the form $p_1 \otimes ...\otimes p_r$ with $p_i$ an irreducible partition for any $i \in \{1,...,r\}$. We define: 
\begin{align*}
\Delta_{\boxplus}({\bold p}) = \sum_{I \subset \{1,...,r\}} [(\otimes_{i\in I} p_i)] \otimes [\otimes_{i\notin I} p_i] \in \left(\bigoplus_{k=0}^{\infty}\mathbb{C}[\mathcal{P}_k / \mathfrak{S}_k]\right) \otimes  \left(\bigoplus_{k=0}^{\infty} \mathbb{C}[\mathcal{P}_k / \mathfrak{S}_k]\right).
\end{align*}
We extend by linearity $\Delta_{\boxplus}$ to $ \bigoplus_{k=0}^{\infty}\mathbb{C}[\mathcal{P}_k / \mathfrak{S}_k]$: 
\begin{align*}
\Delta_{\boxplus} : \bigoplus_{k=0}^{\infty}\mathbb{C}[\mathcal{P}_k / \mathfrak{S}_k] \to \left(\bigoplus_{k=0}^{\infty}\mathbb{C}[\mathcal{P}_k / \mathfrak{S}_k]\right) \otimes  \left(\bigoplus_{k=0}^{\infty} \mathbb{C}[\mathcal{P}_k / \mathfrak{S}_k]\right).
\end{align*}
\end{definition}

\begin{proposition}
\label{prop:hoft}
Let $\epsilon_{\boxplus}$ be the linear form on $\bigoplus_{k=0}^{\infty}\mathbb{C}[\mathcal{P}_k / \mathfrak{S}_k]$ wich sends ${\bold p}$ on $\delta_{{\bold p} = \{\emptyset\}}$: $\left(\bigoplus_{k=0}^{\infty}\mathbb{C}[\mathcal{P}_k / \mathfrak{S}_k], \otimes, \emptyset, \Delta_{\boxplus}, \epsilon_{\boxplus}\right)$ is a graded connected Hopf algebra. 
\end{proposition}

\begin{proof}
It is easy to see that $\left(\bigoplus_{k=0}^{\infty}\mathbb{C}[\mathcal{P}_k / \mathfrak{S}_k], \otimes, \{\emptyset\}, \Delta_{\boxplus}, \epsilon_{\boxplus}\right)$ is an associative and co-associative bi-algebra. Besides, if $p$ is a partition in $\mathcal{P}_k$, we define ${\sf dg}({\bold p}) = k$: $\bigoplus_{k=0}^{\infty}\mathbb{C}[\mathcal{P}_k / \mathfrak{S}_k]$ becomes a graded bi-algebra which is connected. Thus it is a graded connected Hopf algebra. 
\end{proof}

\begin{definition}
The convolution of $\phi_1$ and $\phi_2$ in $ \left( \bigoplus_{k=0}^{\infty} \mathbb{C}[\mathcal{P}_k/ \mathfrak{S}_k]\right)^{*}$ is given by: 
\begin{align*}
\phi_1 \boxplus \phi_2 = (\phi_1 \otimes \phi_2)  \Delta_{\boxplus}.
\end{align*}
\end{definition}

Let us remark that $\epsilon_{\boxplus}$ is the neutral element for $\boxplus$. 

\subsubsection{Multiplicative structure}
\begin{definition}
Let ${\bold p}$ be in $\mathcal{P}_k/ \mathfrak{S}_k$. We define: 
\begin{align*}
\Delta_{\boxtimes}({\bold p}) = \sum_{p_1\circ p_2 = p, p_2 \in {\sf K}_{p_1}(p_1 \circ p_2)} [p_1] \otimes [p_2] \in \left(\bigoplus_{k=0}^{\infty}\mathbb{C}[\mathcal{P}_k / \mathfrak{S}_k]\right) \otimes  \left(\bigoplus_{k=0}^{\infty} \mathbb{C}[\mathcal{P}_k / \mathfrak{S}_k]\right).
\end{align*}
We extend by linearity $\Delta_{\boxtimes}$ to $\bigoplus_{k=0}^{\infty}\mathbb{C}[\mathcal{P}_k / \mathfrak{S}_k]$: 
\begin{align*}
\Delta_{\boxtimes} :  \mathbb{C}[\mathcal{P}_k / \mathfrak{S}_k] \to \left(\bigoplus_{k=0}^{\infty}\mathbb{C}[\mathcal{P}_k / \mathfrak{S}_k]\right) \otimes  \left(\bigoplus_{k=0}^{\infty} \mathbb{C}[\mathcal{P}_k / \mathfrak{S}_k]\right).
\end{align*}
\end{definition}

In the following proposition, we use the notation that $\mathrm{id}_0 = \emptyset$

\begin{proposition}
Let $\epsilon_{\boxtimes}$ be the linear form on $\bigoplus_{k=0}^{\infty}\mathbb{C}[\mathcal{P}_k / \mathfrak{S}_k]$ wich sends, for any $k \in \mathbb{N}$, $\mathrm{id}_k$ on $1$, and any other partition $p$ on $0$: $\left(\bigoplus_{k=0}^{\infty}\mathbb{C}[\mathcal{P}_k / \mathfrak{S}_k], \otimes, \emptyset, \Delta_{\boxtimes}, \epsilon_{\boxtimes}\right)$ is an (associative, co-associative) bi-algebra. 
\end{proposition}

\begin{proof}
The two main problems are to prove the fact that $\Delta_{\boxtimes}$ is a morphism and that it is co-associative. The other facts can be verified easily. 

Let $p_1$ and $p_2$ be two partitions. We need to show that: 
\begin{align*}
\Delta_{\boxtimes}([p_1] \otimes [p_2]) = \Delta_{\boxtimes}([p_1]) \otimes \Delta_{\boxtimes}([p_2]).
\end{align*}
This is a direct consequence of Lemma \ref{factocanard}. 

For the co-associativity, we need to prove that: 
\begin{align}
\label{eq:coasso}
(\Delta_{\boxtimes}\otimes \mathrm{id}) \Delta_{\boxtimes} = ( \mathrm{id} \otimes \Delta_{\boxtimes}) \Delta_{\boxtimes}.
\end{align}
Using the definitions, if $p \in \mathcal{P}_k$: 
\begin{align*}
(\Delta_{\boxtimes}\otimes \mathrm{id}) \Delta_{\boxtimes} ([p])= \sum_{p_1,p_2,p_3 \in \mathcal{P}_k| p_1\circ p_2 \circ p_3 = p, p_3 \in {\sf K}_{p_1 \circ p_2\circ p_3}(p_1\circ p_2), p_2 \in {\sf K}_{p_1 \circ p_2}(p_1)} [p_1] \otimes [p_2] \otimes [p_3].
\end{align*}
Using the upcoming Equation (\ref{eq:deltaassociatif}), 
\begin{align*}
(\Delta_{\boxtimes}\otimes \mathrm{id}) \Delta_{\boxtimes} ([p]) = \sum_{p_1,p_2,p_3 \in \mathcal{P}_k| p_1\circ p_2 \circ p_3 = p, p_{2} \circ p_3 \in {\sf K}_{p_1 \otimes p_2 \otimes p_3}(p_1), p_{3} \in {\sf K}_{p_2 \circ p_3}(p_2)} [p_1] \otimes [p_2] \otimes [p_3], 
\end{align*}
hence Equation (\ref{eq:coasso}).
\end{proof}

It is easy to see that $\left(\bigoplus_{k=0}^{\infty}\mathbb{C}[\mathcal{P}_k / \mathfrak{S}_k], \otimes, \emptyset, \Delta_{\boxtimes}, \epsilon_{\boxtimes}\right)$ is not a Hopf algebra: the antipode does not exist.

\begin{definition}
The convolution of $\phi_1$ and $\phi_2$ in $ \left( \bigoplus_{k=0}^{\infty} \mathbb{C}[\mathcal{P}_k/ \mathfrak{S}_k]\right)^{*}$ is given by: 
\begin{align*}
\phi_1 \boxtimes \phi_2 = (\phi_1 \otimes \phi_2)  \Delta_{\boxtimes}.
\end{align*}
\end{definition}

Let us remark that $\epsilon_{\boxtimes}$ is the neutral element for $\boxtimes$. We will need to understand the action of the $\mathcal{M}$-transform on the $\boxtimes$-convolution. This is given by the following proposition.

\begin{proposition}
Let $\phi_1$ and $\phi_2$ in $ \left( \bigoplus_{k=0}^{\infty} \mathbb{C}[\mathcal{P}_k/ \mathfrak{S}_k]\right)^{*}$: 
\begin{align}
\label{eq:metconvol}
\mathcal{M}(\phi_1 \boxtimes \phi_2 ) = \mathcal{M}(\phi_1) \boxtimes^{m}_{g}  \phi_2 = \phi_1 \boxtimes^{m}_{d} \mathcal{M} (\phi_2), 
\end{align}
where $ \boxtimes^{m}_{g}$ and $ \boxtimes^{m}_{d}$ are the convolutions associated with $\Delta_{\boxtimes^{m}_{g}}$ and $\Delta_{\boxtimes^{m}_{d}}$ given by:
\begin{align*}
\Delta_{\boxtimes^{m}_{g}} (p) &= \sum_{p_1 \leq p} [p \circ \ \!\!^{t}p_1] \otimes [p_1], \ \ \ \ \ \ \ \ 
\Delta_{\boxtimes^{m}_{d}} (p) =  \sum_{p_1 \leq p} [p_1] \otimes [^{t}p_1 \circ p].
\end{align*}
\end{proposition}

\begin{proof}
Actually, it is easier to prove the same equations in the setting of $ \left( \bigoplus_{k=0}^{\infty} \mathbb{C}[\mathcal{P}_k]\right)^{*}$, where $\boxtimes$, $\boxtimes_{g}^{m}$ and $\boxtimes_{d}^{m}$ are defined similarly except that we do not use the brakets $[\ ]$. The special case of $ \left( \bigoplus_{k=0}^{\infty} \mathbb{C}[\mathcal{P}_k/ \mathfrak{S}_k]\right)^{*}$ is then a consequence. 

Let $k$ be an integer. For any $p \in \mathcal{P}_k$, we denote by $p^{*}$ the linear form which gives the coordinate on $p$. By bi-linearity, we only need to prove Equation (\ref{eq:metconvol}) for $\phi_1 = p_1^{*}$ and $\phi_2 = p_2^{*}$ where $p_1$ and $p_2$ are two partitions in $\mathcal{P}_k$. In this case,  Equation (\ref{eq:metconvol}) boils down to the fact that for any $p_0 \in \mathcal{P}_k$:
\begin{align*}
\delta_{ p_1\circ p_2 \leq p_0 } \delta_{p_2 \in {\sf K}_{p_1 \circ p_2}(p_1)} = \delta_{p_1 \leq p_0} \delta_{p_2 \leq \text{ }^{t}p_1 \circ p_0 }
\end{align*}
and 
\begin{align*}
\delta_{ p_1\circ p_2 \leq p_0 } \delta_{p_2 \in {\sf K}_{p_1 \circ p_2}(p_1)} = \delta_{p_1 \leq p_0 \circ \text{ }^{t}p_2} \delta_{p_2 \leq p_0}. 
\end{align*}
These Equalities are consequences of Theorem \ref{th:caractK}. 
\end{proof}

\subsection{Characters, infinitesimal characters and moment maps}
\label{sec:carac}
\subsubsection{Character and infinitesimal characters}

\begin{definition}
A character of $\bigoplus_{k=0}^{\infty}\mathbb{C}[\mathcal{P}_k/\mathfrak{S}_k]$ is a element of  $\left(\bigoplus_{k=0}^{\infty}\mathbb{C}[\mathcal{P}_k/\mathfrak{S}_k]\right)^{*}$ such that:
\begin{enumerate}
\item $\phi(\emptyset) =1$,
\item for any ${\bold p_1}$ and ${ \bold p_2}$ in $\mathcal{P}_k/\mathfrak{S}_k$, $\phi({\bold p_1} \otimes {\bold p_2}) =\phi({\bold p_1}) \phi({\bold p_2})$.
\end{enumerate}
The set of characters is denoted by $\mathcal{X}[\mathcal{P}]$. 
\end{definition}

\begin{notation}
In the following, the results hold either for $\boxplus$ or $\boxtimes$. This leads us to use the following notation: by $\boxdot$, we denote either $\boxplus$ or $\boxtimes$. 
\end{notation}

\begin{proposition}
The set of characters is stable by $\boxdot$. 
\end{proposition}
\begin{proof}
Let $\phi_1$ and $\phi_2$ be two characters. Let $k_1$ and $k_2$ be two integers, let us consider two elements $p_1$ and $ p_2$, respectively in $\mathcal{P}_{k_1}$ and $\mathcal{P}_{k_2}$. Let us remark that $\otimes$ is used for two different notations: as an inner law for $\bigoplus_{k=0}^{\infty}\mathbb{C}[\mathcal{P}_{k} / \mathfrak{S}_k]$ and as the tensor product of two elements of $\bigoplus_{k=0}^{\infty}\mathbb{C}[\mathcal{P}_{k} / \mathfrak{S}_k]$ seen as an element of $\left(\bigoplus_{k=0}^{\infty}\mathbb{C}[\mathcal{P}_{k} / \mathfrak{S}_k]\right) \otimes \left(\bigoplus_{k=0}^{\infty}\mathbb{C}[\mathcal{P}_{k} / \mathfrak{S}_k]\right)$. For sake of clarity, the tensor product in the latter sense will be denoted by $\bigotimes$.

Using the definitions and the fact that $\Delta_{\boxdot}$ is a morphism for $\otimes$: 
\begin{align*}
\phi_1 \boxdot \phi_2 ({\bold p_1} \otimes {\bold p_2}) = (\phi_1 \bigotimes \phi_2)  \Delta_{\boxdot} ({\bold p_1} \otimes {\bold p_2}) &= (\phi_1 \bigotimes \phi_2) \left[ \Delta_{\boxdot} ({\bold p_1}) \otimes \Delta_{\boxdot}({\bold p_2}) \right].
\end{align*}

Let us use the Sweedler notations: $\Delta_{\boxdot} ({\bold p_1}) = \sum p_1^{(1)} \bigotimes p_1^{(2)}$ and 
$\Delta_{\boxdot} ({\bold p_2}) = \sum p_2^{(1)} \bigotimes p_2^{(2)}$. Then: 

\begin{align*}
(\phi_1 \bigotimes \phi_2) \left[ \Delta_{\boxdot} ({\bold p_1}) \otimes \Delta_{\boxdot}({\bold p_2}) \right] &= (\phi_1 \bigotimes \phi_2 ) \left[ \sum (p_1^{(1)} \otimes p_2^{(1)}) \bigotimes (p_1^{(2)} \otimes p_2^{(2)}) \right] \\
&= \sum \phi_1  (p_1^{(1)} \otimes p_2^{(1)}) \phi_2 (p_1^{(2)} \otimes p_2^{(2)}) \\
&=  \sum \phi_1  (p_1^{(1)}) \phi_1( p_2^{(1)}) \phi_2 (p_1^{(2)} ) \phi_2( p_2^{(2)}) \\
&=( \sum  \phi_1  (p_1^{(1)}) \phi_2 (p_1^{(2)} ))( \sum  \phi_1( p_2^{(1)}) \phi_2( p_2^{(2)}))\\
&= (\phi_1 \boxdot \phi_2 ({\bold p_1})) ( \phi_1 \boxdot \phi_2 ({\bold p_2})),
\end{align*}
where, in the third equality, we used the multiplicativity property of $\phi_1$ and $\phi_2$. 
\end{proof}

After this proposition, it is natural to study semi-groups of elements in $\mathcal{X}[\mathcal{P}]$ and to study the set of generators of these semi-groups. 

\begin{definition}
The set of {\em $\boxdot$-infinitesimal characters} is the subset of $\left(\bigoplus_{k=0}^{\infty}\mathbb{C}[\mathcal{P}_k/\mathfrak{S}_k]\right)^{*}$ such that for any partitions $p_1$ and $p_2$, 
\begin{align*}
\phi({\bold p_1} \otimes {\bold p_2}) = \phi({\bold p_1}) \epsilon_{\boxdot}({\bold p_2}) + \epsilon_{\boxdot}({\bold p_1}) \phi({\bold p_2}). 
\end{align*}
The set of ${\boxdot}$-infinitesimal characters is denoted by $\mathfrak{X}_{\boxdot}[\mathcal{P}]$
\end{definition}

Using the definitions of $\epsilon_{\boxtimes}$ and $\epsilon_{\boxplus}$, the set of $\boxplus$-infinitesimal characters is the set of linear forms $\phi$ such that $\phi({\bold p}) = 0$ for any composed partition $p$, and the set of  $\boxtimes$-infinitesimal characters is the set of linear forms $\phi$ such that $\phi({\bold p}) = 0$ for any non weakly-irreducible partition $p$ and $\phi({\bold {id}_k}) = k \phi({\bold {id}_1})$ for any integer $k \in \mathbb{N}$. In particular, since we took the convention that $\emptyset$ was composed, any $\boxdot$-infinitesimal character is equal to $0$ on $[\emptyset]$.

In order to study semi-groups of characters and infinitesimal characters, we need a new order and a notion of height. Recall that ${\sf d}({\bold p})$ is the unique integer such that $p \in \mathcal{P}_k$.

\begin{definition}
Let us endow $\cup_{k=0}^{\infty} \mathcal{P}_{k} / \mathfrak{S}_k$ with two orders: 
\begin{itemize}
\item if $\boxdot=\boxplus$, $ {\bold p} \leq_{\boxplus} {\bold p'}$ if ${\sf dg}(\bold{p}) < {\sf dg}(\bold{p'})$ or ${\bold p} = {\bold p'}$, 
\item if $\boxdot = \boxtimes$, $ {\bold p} \leq_{\boxtimes} {\bold p'}$ if there exist an integer $k$, $p_1$ and $p_2$ in $\mathcal{P}_k$ such that $\bold{p} = [p_1]$, $\bold{p'} = [p_2]$ and $p_1 \leq p_2$.
\end{itemize}
\end{definition}

\begin{definition}
The {\em height} of $({\bold p_1}, {\bold p_2})$: 
\begin{itemize}
\item  if $\boxdot=\boxplus$, ${\sf h}_{\boxplus}({\bold p_1}, {\bold p_2}) = {\sf dg}(\bold{p_1}) + {\sf dg}(\bold{p_2})$, 
\item  if $\boxdot=\boxtimes$, ${\sf h}_{\boxtimes}({\bold p_1}, {\bold p_2}) = d(\mathrm{id}_{{\sf dg}(\bold{p_1})}, \bold{p_1} ) + d(\mathrm{id}_{{\sf dg}(\bold{p_2})}, \bold{p_2}). $
\end{itemize}
\end{definition}

We have now gathered all the notions in order to prove the following result. 

\begin{theorem}
\label{th:infini}
Let $\phi \in \left(\bigoplus_{k=0}^{\infty}\mathbb{C}[\mathcal{P}_k/\mathfrak{S}_k]\right)^{*}$. There exists, for any $t \geq 0$, a unique element $e^{\boxdot t \phi}$ of  $\left(\bigoplus_{k=0}^{\infty}\mathbb{C}[\mathcal{P}_k/\mathfrak{S}_k]\right)^{*}$ such that:
\begin{align*}
\frac{d}{dt}_{| t = t _0} e^{\boxdot t \phi} = \phi \boxdot e^{\boxdot t_0 \phi}, 
\end{align*} 
and $e^{\boxdot 0 \phi} = \epsilon_{\boxdot}$. Besides, the following assertions are equivalent: 
\begin{enumerate}
\item $\phi$ is a $\boxdot$-infinitesimal character,
\item for any $t \geq 0$, $e^{\boxdot t \phi} \in \mathcal{X}[\mathcal{P}]$.
\end{enumerate}
\end{theorem}

\begin{proof}
Let us consider $\phi \in \left(\bigoplus_{k=0}^{\infty}\mathbb{C}[\mathcal{P}_k/\mathfrak{S}_k]\right)^{*}$. The system of differential equations: 
\begin{align*}
\frac{d}{dt}_{| t = t _0} e^{\boxdot t \phi} (p) = \left(\phi \boxdot e^{\boxdot t_0 \phi}\right) (p)
\end{align*} 
is triangular when we consider the order $\leq_{\boxdot}$ on $\cup_{k} (\mathcal{P}_k/ \mathfrak{S}_k)$: it implies that the system admits a unique solution. For $\boxdot = \boxtimes$, this is a consequence of Theorem \ref{th:caractK} where we proved that $p_2 \in {\sf K}_{p}(p_1)$ implies that $p_2 \leq p$.

Let us prove the equivalence between the condition $1$ and the condition $2$. Let $p_1$ and $p_2$ be two partitions such that ${\sf h}({\bold p_1}, {\bold p_2}) > 0$. Let us suppose that for any couple of partitions $(p'_1,p'_2)$ such that ${\sf h}_{\boxdot} ({\bold p_1'},{\bold  p_2'}) < {\sf h}_{\boxdot} ({\bold p_1},{\bold  p_2})$, $\phi({\bold p'_1}\otimes {\bold p'_2}) = \epsilon_{\boxdot}({\bold p'_1}) \phi({\bold p'_2}) + \phi({\bold p'_1}) \epsilon_{\boxdot}({\bold p'_2})$ and for any $t \geq 0$, $e^{\boxdot t_0 \phi} ({\bold p_1'}\otimes {\bold p_2'}) = e^{\boxdot t_0\phi} ({\bold p'_1})e^{\boxdot t_0\phi} ({\bold p'_2})$. 

Let $t_0 \geq 0$: 
\begin{align*}
\frac{d}{dt}_{| t = t_0} e^{t \phi}({\bold p_1} \otimes {\bold p_2}) &= \sum \phi(p_1^{(1)} \otimes p_2^{(1)}) e^{t_0 \phi}(p_1^{(2)} \otimes p_2^{(2)}) \\
&= \phi(e_1 \otimes e_2) e^{t_0 \phi}({\bold p_1} \otimes {\bold p_2}) + \phi({\bold p_1} \otimes {\bold p_2}) e^{t_0 \phi}(e_1\otimes e_2) \\ & \ \ \ \ +\sum_{p_1^{(1)} \otimes p_2^{(1)} \notin \{e_1\otimes e_2, p_1 \otimes p_2\}} \phi(p_1^{(1)} \otimes p_2^{(1)}) e^{t_0 \phi}(p_1^{(2)} \otimes p_2^{(2)})
\end{align*}
where $(e_1,e_2) = ([\emptyset], [\emptyset])$ if $\boxdot = \boxplus$ and $(e_1,e_2) = ([\mathrm{id}_{{\sf dg}(\bold{p_1})}], [\mathrm{id}_{{\sf dg}(\bold{p_2})}])$ if $\boxdot = \boxtimes$. We can apply the hypothese about $\phi$ and $e^{\boxdot t_0 \phi}$, $\frac{d}{dt}_{| t = t_0} e^{t \phi}({\bold p_1} \otimes {\bold p_2}) $ is equal to: 

\begin{align*}
 \phi(e_1 \otimes e_2) &e^{t_0 \phi}({\bold p_1} \otimes {\bold p_2}) + \phi({\bold p_1} \otimes {\bold p_2}) e^{t_0 \phi}(e_1\otimes e_2) \\
 & + \sum \phi(p_1^{(1)}) \epsilon_{\boxdot}(p_2^{(1)}) e^{t_0\phi}(p_1^{(2)}) e^{t_0 \phi}(p_2^{(2)}) + \sum \epsilon_{\boxdot}(p_1^{(1)}) \phi(p_2^{(1)}) e^{t_0\phi}(p_1^{(2)}) e^{t_0 \phi}(p_2^{(2)})
  \\& - \phi({\bold p_1})\epsilon_{\boxdot} ({\bold p_2}) e^{t_0 \phi}(e_1) e^{t_0\phi}(e_2) - \phi(e_1) \epsilon_{\boxdot}(e_2) e^{t_0 \phi}({\bold p_1}) e^{t_0 \phi}({\bold p_2})\\
  & - \epsilon_{\boxdot}({\bold p_1}) \phi({\bold p_2}) e^{t_0 \phi}(e_1) e^{t_0 \phi}(e_2) - \epsilon_{\boxdot}(e_1) \phi(e_2)  e^{t_0 \phi}({\bold p_1}) e^{t_0 \phi}({\bold p_2}). 
\end{align*}

But,
\begin{align*}
\sum \phi(p_1^{(1)}) \epsilon_{\boxdot}(p_2^{(1)}) e^{t_0\phi}(p_1^{(2)}) e^{t_0 \phi}(p_2^{(2)}) &= (\phi \boxdot e^{t_0 \phi}({\bold p_1})) (\epsilon_{\boxdot} \boxdot e^{t_0 \phi} ({\bold p_2}))\\& = e^{t_0 \phi}({\bold p_2}) \frac{d}{dt}_{| t = t_0} e^{t \phi}({\bold p_1}),
\end{align*}

Using the fact that $\phi(e_1\otimes e_2) = \phi(e_1)\epsilon_{\boxdot}(e_2) + \epsilon_{\boxdot}(e_1)\phi(e_2)$ and $e^{t_0 \phi}(e_1 \otimes e_{2}) = e^{t_0 \phi}(e_1)e^{t_0 \phi}(e_2)$, we get that $\frac{d}{dt}_{|t = t_0}\left(e^{t \phi}({\bold p_1} \otimes {\bold p_2}) - e^{t \phi}({\bold p_1}) e^{t \phi}({\bold p_2})\right) $ is equal to: 
\begin{align*}
\phi(e_1\otimes e_2)&\left( e^{t_0 \phi}({\bold p_1} \otimes {\bold p_2}) - e^{t_0 \phi}({\bold p_1}) e^{t_0 \phi}({\bold p_2})\right) \\& +  e^{t_0 \phi}(e_1\otimes e_2)  \left( \phi({\bold p_1} \otimes {\bold p_2})  - \phi({\bold p_1}) \epsilon_{\boxdot}({\bold p_2}) - \phi({\bold p_2}) \epsilon_{\boxdot}({\bold p_1})\right).
\end{align*}

It is easy to see that this result allows us to prove, by recurrence on the height of the couple $({\bold p_1},{\bold p_2})$, the equivalence between conditions $1$ and $2$. 
 \end{proof}

\begin{remarque}
Let us remark that when $\boxdot=\boxplus$ and if we suppose that $\phi(\empty) = 0$ these results are already known (\cite{manchon}). But when $\boxdot=\boxtimes$, the usual theory can not be applied since $\left(\bigoplus_{k=0}^{\infty}\mathbb{C}[\mathcal{P}_k / \mathfrak{S}_k], \otimes, \emptyset, \Delta_{\boxtimes}, \epsilon_{\boxtimes}\right)$ is not a filtered (connected Hopf) algebra. 
\end{remarque}

\subsubsection{Characters, infinitesimal characters and transformations}
We will study characters and infinitesimal characters using the different transformations that we defined on $\left(\bigoplus_{k=0}^{\infty} \mathbb{C}[\mathcal{P}_k/\mathfrak{S}_k]\right)^{*}$. 

\begin{proposition}
\label{prop:bijectionM}
The $\mathcal{M}$, $\mathcal{R}$,  $\mathcal{M}^{\to c}$ and  $\mathcal{M}^{c \to}$-transforms are bijections from $\mathcal{X}[\mathcal{P}]$ to itself. 
\end{proposition}

\begin{proof}
Let us prove that the $\mathcal{M}$-transform is a bijection from $\mathcal{X}[\mathcal{P}]$ to itself: it will implies that the $\mathcal{R}$-transform is a bijection, and using the same arguments, we can prove that the   $\mathcal{M}^{\to c}$ and  $\mathcal{M}^{c \to}$-transforms are bijections.

Let $\phi$ be an element of $\left(\bigoplus_{k=0}^{\infty}\mathbb{C}[\mathcal{P}_k/\mathfrak{S}_k]\right)^{*}$, let $p_1$ and $p_2$ be two partitions, respectively in $\mathcal{P}_{k_1}$ and $\mathcal{P}_{k_2}$. Let us suppose that for any $(p'_1,p'_2) \in \mathcal{P}_{k_1} \times \mathcal{P}_{k_2}$ such that ${\sf h}_{\boxtimes}({\bold p'_1}, {\bold p'_2})< {\sf h}_{\boxtimes}({\bold p_1}, {\bold p_2})$, we know that $\phi({\bold p'_1 }\otimes {\bold p'_2}) = \phi({\bold p_1'}) \phi({\bold p'_2})$. We have the equality: 
\begin{align}
\label{eq:equacaractere}
(\mathcal{M}(\phi))({\bold p_1}\otimes {\bold p_2}) - (\mathcal{M}(\phi))({\bold p_1})(\mathcal{M}(\phi))({\bold p_2}) = \phi({\bold p_1}\otimes {\bold p_2}) -  \phi({\bold p_1})\phi({\bold p_2}).
\end{align}
This is a simple consequence of Proposition \ref{multiplicativgeo}, the definition of $\mathcal{M}$ and the hypothesis on $\phi$.

This allows us to prove by recurrence on the height of $({\bold p_1}, {\bold p_2})$ that $\mathcal{M}$ is a bijection on the set of characters. 
\end{proof}

In order to state a similar result for the infinitesimal characters, we need to define the set of additive characters and the set $\mathfrak{X}_{\boxtimes}^{c}[\mathcal{P}]$. Recall Definition \ref{supportexclusive} and the notation ${\sf 0}_k$. 

\begin{definition}
An element $\phi \in \left(\bigoplus_{k=0}^{\infty}\mathbb{C}[\mathcal{P}_k/\mathfrak{S}_k]\right)^{*}$ is an additive character if for any partitions $p_1$ and $p_2$: 
\begin{align*}
\phi({\bold p_1} \otimes {\bold p_2} ) = \phi({\bold p_1}) + \phi( {\bold p_2}).
\end{align*}
The set of additive characters is denoted by $\mathcal{X}_{+}[\mathcal{P}]$. 

Let $\epsilon_{\boxtimes}^{c}$ be the linear form in $\left(\bigoplus_{k=0}^{\infty}\mathbb{C}[\mathcal{P}_k/\mathfrak{S}_k]\right)^{*}$ which sends, for any $k \in \mathbb{N}$, ${\sf 0}_k$ on $1$, and any other partition $p$ on $0$. We recall that ${\sf 0}_{k} = \{\{1,...,k,1',...,k'\}\}$. The set $\mathfrak{X}_{\boxtimes}^{c}[\mathcal{P}]$ is the set of linear forms in $\left(\bigoplus_{k=0}^{\infty}\mathbb{C}[\mathcal{P}_k/\mathfrak{S}_k]\right)^{*}$ such that $\phi({\bold p_1} \otimes {\bold p_2}) = \epsilon_{\boxtimes}^{c} ({\bold p_1}) \phi({\bold p_2}) + \phi({\bold p_1})\epsilon_{\boxtimes}^{c}({\bold p_2})$ for any partitions $p_1$ and $p_2$. 
\end{definition}
Using the definition of $\epsilon_{\boxtimes}^{c}$ and Definition \ref{supportexclusive}, the set $\mathfrak{X}_{\boxtimes}^{c}[\mathcal{P}]$ is the set of linear forms $\phi$ such that $\phi({\bold p})= 0$ for any non exclusive-irreducible partition $p$.

\begin{theorem}
\label{critcondens}
The $\mathcal{M}$-transform is a bijection from : 
\begin{enumerate}
\item $\mathfrak{X}_{\boxplus}[\mathcal{P}]$ to itself, 
\item $\mathfrak{X}_{\boxtimes}[\mathcal{P}]$ to $\mathcal{X}_{+}[\mathcal{P}]$. 
\end{enumerate}
In particular, the $\mathcal{R}$-transform is a bijection from $\mathfrak{X}_{\boxplus}[\mathcal{P}]$ to itself and from $\mathcal{X}_{+}[\mathcal{P}]$ to $\mathfrak{X}_{\boxtimes}[\mathcal{P}]$. 
\end{theorem}

\begin{proof}
The proof uses the same arguments as the proof of Proposition \ref{prop:bijectionM}. We only explains which equation is used instead of Equation (\ref{eq:equacaractere}).
\begin{enumerate}
\item In order to study the assertion on $\mathfrak{X}_{\boxplus}[\mathcal{P}]$, it is enough to see that if for any $(p'_1,p'_2)$ such that ${\sf h}_{\boxtimes}({\bold p'_1}, {\bold p'_2})< {\sf h}_{\boxtimes}({\bold p_1}, {\bold p_2})$, we know that $\phi({\bold p'_1} \otimes{\bold  p'_2}) - \epsilon_{\boxplus}({\bold p'_1}) \phi({\bold p'_2}) - \phi({\bold p'_1}) \epsilon_{\boxplus}({\bold p'_2}) = 0$, then:  
\begin{align*}
\ \ \ (\mathcal{M}(\phi))({\bold p_1}\otimes {\bold p_2}) - \epsilon_{\boxplus}({\bold p_1}) (\mathcal{M}(\phi))&({\bold p_2}) - (\mathcal{M}(\phi))({\bold p_1}) \epsilon_{\boxplus}({\bold p_2}) \\&= \phi({\bold p_1 }\otimes {\bold p_2}) - \epsilon_{\boxplus}({\bold p_1}) \phi({\bold p_2}) - \phi({\bold p_1}) \epsilon_{\boxplus}({\bold p_2}).
\end{align*}
This is a simple consequence of Proposition \ref{multiplicativgeo}, the definition of $\mathcal{M}$ and the fact that $\sum_{p'\leq p} \epsilon_{\boxplus}(p') = \epsilon_{\boxplus}(p)$ for any partition $p$.

\item In order to study the assertion on $\mathfrak{X}_{\boxtimes}[\mathcal{P}]$, it is enough to see that if for any $(p'_1,p'_2)$ such that ${\sf h}_{\boxtimes}({\bold p'_1}, {\bold p'_2})< {\sf h}_{\boxtimes}({\bold p_1}, {\bold p_2})$, we know that $\phi({\bold p'_1} \otimes {\bold p'_2}) - \epsilon_{\boxtimes}({\bold p'_1}) \phi({\bold p'_2}) - \phi({\bold p'_1}) \epsilon_{\boxtimes}({\bold p'_2}) = 0$, then:  
\begin{align*}
\ \ \ \ \ \mathcal{M}(\phi)({\bold p_1}\otimes {\bold p_2}) - \mathcal{M}(\phi)&({\bold p_1})- \mathcal{M}(\phi)({\bold p_2}) \\&= \phi({\bold p_1} \otimes {\bold p_2}) - \epsilon_{\boxtimes}({\bold p_1}) \phi({\bold p_2}) - \phi({\bold p_1}) \epsilon_{\boxtimes}({\bold p_2}). 
\end{align*}
This is a simple consequence of Proposition \ref{multiplicativgeo}, the definition of $\mathcal{M}$ and the fact that $\sum_{p' \leq p} \epsilon_{\boxtimes} (p') = 1$ for any partition $p$.
\end{enumerate}
\end{proof}

\begin{theorem}
The $\mathcal{M}^{\to c}$-transform is a bijection from: 
\begin{enumerate}
\item $\mathfrak{X}_{\boxplus}[\mathcal{P}]$ to itself, 
\item $\mathfrak{X}_{\boxtimes}[\mathcal{P}]$ to $\mathfrak{X}_{\boxtimes}^{c}[\mathcal{P}]$. 
\end{enumerate}
\end{theorem}

\begin{proof}
The proof uses the same arguments as the proof of Proposition \ref{prop:bijectionM}. Again we only explain which equation is used instead of Equation (\ref{eq:equacaractere}). 
\begin{enumerate}
\item In order to study the assertion on $\mathfrak{X}_{\boxplus}[\mathcal{P}]$, it is enough to see that if for any $(p'_1,p'_2)$ such that ${\sf h}_{\boxtimes}({\bold p'_1}, {\bold p'_2})< {\sf h}_{\boxtimes}({\bold p_1}, {\bold p_2})$, we know that $\phi({\bold p'_1} \otimes{\bold  p'_2}) - \epsilon_{\boxplus}({\bold p'_1}) \phi({\bold p'_2}) - \phi({\bold p'_1}) \epsilon_{\boxplus}({\bold p'_2}) = 0$, then:  
\begin{align*}
(\mathcal{M}^{\to c}(\phi))({\bold p_1}\otimes{\bold  p_2}) - \epsilon_{\boxplus}({\bold p_1}) (\mathcal{M}^{\to c}&(\phi))({\bold p_2}) - (\mathcal{M}^{\to c}(\phi))({\bold p_1}) \epsilon_{\boxplus}({\bold p_2}) \\&= \phi({\bold p_1} \otimes{\bold  p_2}) - \epsilon_{\boxplus}({\bold p_1}) \phi({\bold p_2}) - \phi({\bold p_1}) \epsilon_{\boxplus}({\bold p_2}).
\end{align*}
This is a simple consequence of Proposition \ref{multiplicativgeo}, the definition of $\mathcal{M}^{\to c}$ and the fact that $\sum_{p'\sqsupset p} \epsilon_{\boxplus}(p') = \epsilon_{\boxplus}(p)$ for any partition $p$.
\item  In order to study the assertion on $\mathfrak{X}_{\boxtimes}[\mathcal{P}]$,  it is enough to see that if for any $(p'_1,p'_2)$ such that ${\sf h}_{\boxtimes}({\bold p'_1}, {\bold p'_2})< {\sf h}_{\boxtimes}({\bold p_1}, {\bold p_2})$, we know that $\phi({\bold p'_1 }\otimes {\bold p'_2}) - \epsilon_{\boxtimes}({\bold p'_1}) \phi({\bold p'_2}) - \phi({\bold p'_1}) \epsilon_{\boxtimes}({\bold p'_2}) = 0$, then:  
\begin{align*}
(\mathcal{M}^{\to c}(\phi))({\bold p_1}\otimes{\bold  p_2}) - \epsilon^{c}_{\boxtimes}({\bold p_1}) (\mathcal{M}^{\to c}&(\phi))({\bold p_2}) - (\mathcal{M}^{\to c}(\phi))({\bold p_1}) \epsilon^{c}_{\boxtimes}({\bold p_2}) \\&= \phi({\bold p_1} \otimes {\bold p_2}) - \epsilon_{\boxtimes}({\bold p_1}) \phi({\bold p_2}) - \phi({\bold p_1}) \epsilon_{\boxtimes}({\bold p_2}).
\end{align*}
This is a simple consequence of Proposition \ref{multiplicativgeo}, the definition of $\mathcal{M}^{\to c}$ and the fact that $\sum_{p'\sqsupset p} \epsilon_{\boxtimes}(p') = \epsilon^{c}_{\boxtimes}(p)$ for any partition $p$.
\end{enumerate}
\end{proof}

\subsection{Study of $(\bigoplus_{k=0}^{\infty}\mathbb{C}[{\mathcal{A}}_k])^{*}$}

\subsubsection{Projections}
\label{sec:projection}
Recall the definitions in Section \ref{sec:defA}. Let us consider $\mathcal{A}$ in $\{\mathfrak{S}, \mathcal{B}, \mathcal{B}s, \mathcal{H}\}$. Since for any integer $k$, $\mathcal{A}_k$ is a subset of  $\mathcal{P}_k$, any element of $(\bigoplus_{k=0}^{\infty} \mathbb{C}[\mathcal{P}_k])^{*}$ can be restricted to $\bigoplus_{k=0}^{\infty} \mathbb{C}[\mathcal{A}_k]$.

\begin{definition}
Let $\phi$ be an element of $(\bigoplus_{k=0}^{\infty} \mathbb{C}[\mathcal{P}_k])^{*}$, its restriction to $ \bigoplus_{k=0}^{\infty}\mathbb{C}[\mathcal{A}_k]$ is denoted by ${\sf R}_{\mathcal{A}}(\phi)$. We can extend canonically an element of $(\bigoplus_{k=0}^{\infty} \mathbb{C}[\mathcal{A}_k])^{*}$ by defining for any $\phi \in (\bigoplus_{k=0}^{\infty} \mathbb{C}[\mathcal{A}_k])^{*}$, ${\sf E}_{\mathcal{A}}(\phi)$ as the unique element of $(\bigoplus_{k=0}^{\infty} \mathbb{C}[\mathcal{P}_k])^{*}$ such that for any $p \in \cup_{k=0}^{\infty} \mathcal{P}_k$, $\left({\sf E}_{\mathcal{A}}(\phi)\right)(p) = \delta_{p \in \cup_{k=0}^{\infty} \mathcal{A}_k} \phi(p). $
\end{definition}

This definition allows us to define a $\mathcal{M}_{\mathcal{A}}$ and a $\mathcal{R}_{\mathcal{A}}$-transforms. 

\begin{definition}
The $\mathcal{M}_{\mathcal{A}}$-transform is given by: $\mathcal{M}_{\mathcal{A}} = {\sf R}_{\mathcal{A}} \circ \mathcal{M} \circ {\sf E}_{\mathcal{A}}$. It is a bijection from $(\bigoplus_{k=0}^{\infty} \mathbb{C}[\mathcal{A}_k])^{*}$ to itself, whose inverse is the $\mathcal{R}_{\mathcal{A}}$-transform. 
\end{definition}

In order to be more pedagogical, let us explain the equality $\mathcal{M}_{\mathcal{A}} = {\sf R}_{\mathcal{A}} \circ \mathcal{M} \circ {\sf E}_{\mathcal{A}}$: for any $\phi$ in $(\bigoplus_{k=0}^{\infty} \mathbb{C}[\mathcal{A}_k])^{*}$, for any $k$ and any $b \in \mathcal{A}_k$: 
\begin{align*}
(\mathcal{M}_{\mathcal{A}} (\phi))(b) = \sum_{b' \in \mathcal{A}_k | b \leq b} \phi(b'). 
\end{align*}

\begin{remarque}
The definitions of characters, $\boxdot$-convolutions, $\boxdot$-infinitesimal characters and the results which concern these notions and the $\mathcal{M}_{\mathcal{A}}$, $\mathcal{R}_{\mathcal{A}}$-transforms can be extended to $(\bigoplus_{k=0}^{\infty} \mathbb{C}[\mathcal{A}_k])^{*}$.
\end{remarque}

The application ${\sf E}_{\mathcal{A}} \circ {\sf R}_{\mathcal{A}}$ is a projection which ``erases" the values for $p \notin \cup_{k} \mathcal{A}_k$. Let us define three new interesting projections. 

\begin{definition}
Let $\mathcal{C}^{\kappa}_{\mathcal{A}}$,  $\mathcal{C}^{m}_{\mathcal{A}}$ and $\mathcal{C}^{m}_{\mathcal{A}}$ be the three applications on $(\bigoplus_{k=0}^{\infty} \mathbb{C}[\mathcal{P}_k])^{*}$ given by:
\begin{align*}
\mathcal{C}^{\kappa}_{\mathcal{A}} &= {\sf E}_{A}  \circ \mathcal{R}_{\mathcal{A}} \circ {\sf R}_{ \mathcal{A}} \circ \mathcal{M}, \\
\mathcal{C}^{m}_{\mathcal{A}} &= \mathcal{M} \circ \mathcal{C}^{\kappa}_{\mathcal{A}}  \circ \mathcal{R}, \\
\mathcal{C}^{m^{c}}_{\mathcal{A}} &= \mathcal{M}^{\to c} \circ \mathcal{C}^{\kappa}_{\mathcal{A}}  \circ  (\mathcal{M}^{\to c})^{-1}. 
\end{align*}
The application $\mathcal{C}^{\kappa}_{\mathcal{A}}$ is called the {\em cumulant-projection} on $\mathcal{A}$, the application $\mathcal{C}^{m}_{\mathcal{A}}$ is called the {\em moment-projection} on $\mathcal{A}$ and the application $\mathcal{C}^{m^{c}}_{\mathcal{A}}$ is called the {\em exclusive-projection} on $\mathcal{A}$.  
\end{definition}

These applications are projections and ${\sf Im}( \mathcal{C}^{\kappa}_{\mathcal{A}}) = {\sf E}_{\mathcal{A}}((\bigoplus_{k=0}^{\infty} \mathbb{C}[\mathcal{A}_k])^{*})$, ${\sf Im}( \mathcal{C}^{m}_{\mathcal{A}}) = \mathcal{M} \circ {\sf E}_{\mathcal{A}}((\bigoplus_{k=0}^{\infty} \mathbb{C}[\mathcal{A}_k])^{*})$ and ${\sf Im}( \mathcal{C}^{m^{c}}_{\mathcal{A}}) = \mathcal{M}^{\to c} \circ {\sf E}_{\mathcal{A}}((\bigoplus_{k=0}^{\infty} \mathbb{C}[\mathcal{A}_k])^{*})$. These remarks are direct consequences of the following straightforward equality: 
\begin{align*} 
{\sf R}_{ \mathcal{A}} \circ \mathcal{M} \circ  {\sf E}_{A}  \circ \mathcal{R}_{\mathcal{A}} = {\sf Id}_{(\bigoplus_{k=0}^{\infty} \mathbb{C}[\mathcal{A}_k])^{*}}. 
\end{align*}

\begin{notation}
In the following, for any $\mathcal{A}$,  $\mathcal{G}(\mathcal{A})$ will denote the letter given in the following table. 
\begin{center}

\begin{tabular}{|c||c|c|c|c|c|}
\hline 
& & & & & \\
$\mathcal{A}$ & $\ \ \ \mathfrak{S}\ \ \ $ & $\ \ \ \mathcal{B}\ \ \ $ &$\ \ \  \mathcal{B}s \ \ \ $&$ \ \ \ \mathcal{H}\ \ \ $ &$\ \ \  \mathcal{P}\ \ \ $ \\ & & & & &\\
\hline 
& & & & & \\
$\ \ \ \mathcal{G}(\mathcal{A})\ \ \ $ & $U $&$ O $& $B $& $H$ &$ \mathfrak{S}$\\ 
& & & & &\\
\hline 
\end{tabular} \\
\vspace{+10pt}
Table $1$. Notation $\mathcal{G}(\mathcal{A})$.
\end{center}
\end{notation}

 Let us consider an element $\phi$ of $(\bigoplus_{k=0}^{\infty}\mathbb{C}[\mathcal{P}_k])^{*}$.

\begin{definition}
\label{Uinvariant}
We say that $\phi$ is $\mathcal{G}(\mathcal{A})$-invariant if $\phi $ is a fixed point of $\mathcal{C}_{\mathcal{A}}^{m}$.
\end{definition}

Since $\mathcal{C}_{\mathcal{A}}^{m}$ is a projection and using Equation (\ref{eq:decompo}), we get the following lemma. 
\begin{lemme}
The linear form $\phi$ is  $\mathcal{G}(\mathcal{A})$-invariant if and only if one of the following conditions is satisfied: 
\begin{enumerate}
\item $\phi \in \mathcal{M} \circ {\sf E}_{\mathcal{A}}((\bigoplus_{k=0}^{\infty} \mathbb{C}[\mathcal{A}_k])^{*})$,
\item $\mathcal{R}(\phi) \in {\sf E}_{\mathcal{A}}((\bigoplus_{k=0}^{\infty} \mathbb{C}[\mathcal{A}_k])^{*}),$
\item $(\mathcal{M}^{c \to})^{-1} (\phi) \in \mathcal{M}^{\to c} \circ {\sf E}_{\mathcal{A}}((\bigoplus_{k=0}^{\infty} \mathbb{C}[\mathcal{A}_k])^{*})$.
\end{enumerate}
\end{lemme}

The sets ${\sf E}_{\mathcal{A}}((\bigoplus_{k=0}^{\infty} \mathbb{C}[\mathcal{A}_k])^{*})$ and $ \mathcal{M}^{\to c} \circ {\sf E}_{\mathcal{A}}((\bigoplus_{k=0}^{\infty} \mathbb{C}[\mathcal{A}_k])^{*})$ are easy to understand. Recall Definition \ref{setbbarre} where we defined ${\sf Mb}(p)$.

\begin{lemme}
\label{explication}
We have the following characterizations: 
\begin{enumerate}
\item The set ${\sf E}_{\mathcal{A}}((\bigoplus_{k=0}^{\infty} \mathbb{C}[\mathcal{A}_k])^{*})$ is the set of linear forms $\phi$ such that for any $p \notin \cup_{k}\mathcal{A}_k$, $\phi(p) = 0$. 
\item When $\mathcal{A} \in \{ \mathfrak{S}, \mathcal{B} \}$, the set $\mathcal{M}^{\to c} \circ {\sf E}_{\mathcal{A}}((\bigoplus_{k=0}^{\infty} \mathbb{C}[\mathcal{A}_k])^{*})$ is the set of  linear forms $\phi$ such that for any $p \in \cup_{k}\mathcal{P}_k$, 
$\phi(p) = \delta_{p \in \cup_k\overline{\mathcal{A}_k}} \phi({\sf Mb}(p)).$
\end{enumerate}

\end{lemme}

\begin{proof}
The first assertion is straightforward. The second is a direct consequence of Lemma \ref{pcoarsersigma}. Indeed, if $\phi \in\mathcal{M}^{\to c} \circ {\sf E}_{\mathcal{A}}((\bigoplus_{k=0}^{\infty} \mathbb{C}[\mathcal{A}_k])^{*})$, there exists $\phi'$ in $(\bigoplus_{k=0}^{\infty} \mathbb{C}[\mathcal{A}_k])^{*}$ such that for any $p \in \mathcal{P}_k$: 
$$\phi(p) = \sum_{p' \in \mathcal{A}_k \mid p' \sqsupset p} \phi'(p').$$
Using Lemma \ref{pcoarsersigma}, for any $p \in \mathcal{P}_k$, $ \phi(p)= \delta_{p \in \cup_{k} \overline{\mathcal{A}_k}}\phi'({\sf Mb}(p)) =  \delta_{p \in \cup_{k} \overline{\mathcal{A}_k}}\phi({\sf Mb}(p))$. Using similar arguments, the other implication is straitforward. 
\end{proof}

\subsubsection{The moment map and $(\bigoplus_{k=0}^{\infty}\mathbb{C}[{\mathfrak{S}}_k])^{*}$}

We will explain in this section that the $\mathcal{R}_{\mathfrak{S}}$-transform is in fact the usual $\mathcal{R}$-transform in free probabilies (\cite{Voicu1}, \cite{Voicu2}, \cite{Speicher}). Let us begin with a straightfoward lemma. 

\begin{lemme}
\label{identification}
The affine space $\mathcal{X}[\mathfrak{S}]$ can ben identified with the affine space $\mathbb{C}_{1}[[z]]$ of formal power series which constant term is equal to $1$ by the following isomorphism : 
\begin{align*}
\Psi : \mathcal{X}[\mathfrak{S}]& \to \mathbb{C}_{1}[[z]]\\
\phi & \mapsto \sum_{k \in \mathbb{N}} \phi((1,...,k)) z^{k}, 
\end{align*} 
where we recall that $(1,...,k)$ is the $k$-cycle in $\mathfrak{S}_k$. 
\end{lemme}

Let us recall the notion of $\mathcal{R}$-transform in free probabilities, defined on $\mathbb{C}_1[[z]]$, which we will call the $\mathcal{R}_u$-transform. 

\begin{definition}
Let $M(z)$ be a formal power serie in $\mathbb{C}_{1}[[z]]$, that is a formal power serie of the form: 
\begin{align*}
M(z) = 1+\sum_{n=1}^{\infty} a_n z^{n}. 
\end{align*}
Let $C(z)$ be the formal power serie $C(z) = 1+\sum_{n=1}^{\infty}k_{n}z^{n}$ such that $C[zM(z)] = M(z)$.  
The $\mathcal{R}_{u}$-transform of $M$ is $C$. 
\end{definition}

\begin{theorem}
Using the identification $\mathcal{X}[\mathfrak{S}] \simeq \mathbb{C}_{1}[[z]]$ via the application $\Psi$ explained in Lemma \ref{identification}, the following diagram is commutative: 
\begin{align*}
 \xymatrix{\mathcal{X}[\mathfrak{S}]  \ar[rr]^{\mathcal{R}_{\mathfrak{S}}} \ar[d]_{\Psi} &&\mathcal{X}[\mathfrak{S}] \ar[d]^{\Psi}\\
  \mathbb{C}_{1}[[z]] \ar[rr]^{\mathcal{R}_u}  && \mathbb{C}_{1}[[z]] }
 \end{align*}
\end{theorem}

\begin{proof}
The assertion is equivalent to the fact that for any $\phi \in \mathcal{X}[\mathfrak{S}]$, any integer $k$, any $\sigma \in \mathfrak{S}_k$, 
\begin{align*}
\phi(\sigma) = \sum_{\sigma' \in \mathfrak{S}_k | \sigma' \leq \sigma} \prod_{c \text{ cycle of } \sigma} \mathcal{R}_u ( \Psi(\phi))_{\# c}.
\end{align*}

When $\sigma = (1,...,k)$, this is a consequence of the bijection between non-crossing partitions of $k$ elements and the set $[\mathrm{id}_k, (1,…,k)]\cap \mathfrak{S}_k$ and Theorem $2.7$ of \cite{Speic}. In the general case is a consequence of this special case by using the factorization of the geodesics (Lemma \ref{multiplicativgeo}). 
\end{proof}

\section{Observables and convergences of partitions}
\label{sec:obsconv}
In this section, we motivate the definition of the structures in Section \ref{sec:structure}: these structures appear when one studies limits of elements in $\prod\limits_{N=1}^{\infty} \mathbb{C}[\mathcal{P}_k(N)]$. The notion of convergence we are going to use is the good notion to consider when one wants to apply the results to the study of random matrices which are invariant in law by conjugation by the symmetric group (\cite{Gab2}, \cite{Gab3}). 

\subsection{Definitions}
In order to define the notion of convergence on $\prod\limits_{N=1}^{\infty} \mathbb{C}[\mathcal{P}_k(N)]$, we define observables. 

\begin{definition}
Let $N \in \mathbb{N}$, let $p$ be a partition in $\mathcal{P}_k$ and $E \in \mathbb{C}[\mathcal{P}_k(N)]$. The $p$-moment of $E$ is:  
\begin{align*}
m_{p} (E) = \frac{1}{{\sf Tr}_N( p)}{\sf Tr}_N(E\ ^tp),
\end{align*}
where the product is seen in $\mathbb{C}[\mathcal{P}_k(N)]$. 
\end{definition}

Let $p$ be a partition in $\mathcal{P}_k$ and $p'$ be a partition seen in $\mathbb{C}[\mathcal{P}_k(N)]$, using the Equalities $(\ref{lientracenc})$ and $(\ref{eq:lientrace2})$, we have: 
\begin{align}
\label{eq:combimom}
m_{p} (p') = N^{{\sf nc}(p \vee p') - {\sf nc}(p \vee \mathrm{id}_k)}.
\end{align}

From now on, $(E_{N})_{N \in \mathbb{N}}$ is an element of $\prod\limits_{N=1}^{\infty} \mathbb{C}[\mathcal{P}_k(N)]$.

\begin{definition}
\label{weakconv}
The sequence $(E_N)_{N \in \mathbb{N}}$ converges in moments if for any $p \in \mathcal{P}_k$, $m_p(E_N)$ converges when $N$ goes to infinity. If so, we denote by $m_p(E)$ the limit of $m_p(E_N)$ and we denote by $\phi^{m}_{E}$ the linear form in $(\mathbb{C}[\mathcal{P}_k])^{*}$ such that $\phi_{E}^{m}(p) = m_{p}(E)$ for any $p \in \mathcal{P}_{k}$.
\end{definition}
 
In the next section, we give a condition on the coordinates of $(E_N)_{N \in \mathbb{N}}$ which is equivalent to the convergence in moments of $(E_N)_{N \in \mathbb{N}}$.

\subsection{Cumulants and the key result}
For any $p \in \mathcal{P}_k$ and any $F \in \mathbb{C}[\mathcal{P}_k(N)]$, we denote by $F_p$ the coodinate of $F_p$ on the partition $p$: $F = \sum_{p \in \mathcal{P}_k}F_p p. $

\begin{definition}
Let $p$ be a partition in $\mathcal{P}_k$ and let $F \in \mathbb{C}[\mathcal{P}_k(N)]$. The cumulant of $F$ on $p$ is: 
\begin{align*}
\kappa_{p}(F) = N^{{\sf nc}(p) - {\sf nc}(p \vee \mathrm{id}_k)} F_{p}. 
\end{align*}
\end{definition}

\begin{theorem}
\label{equivalence}
 The sequence  $(E_N)_{N \in \mathbb{N}}$ converges in moments if and only if for any $p \in \mathcal{P}_k$, $\kappa_p(E_N)$ converges as $N$ goes to infinity to a number that we denote by $\kappa_{p}(E)$.
 
Let us suppose that $(E_N)_{N \in \mathbb{N}}$ converges in moments, for any $p\in \mathcal{P}_k$: 
\begin{align}
\label{kappamom}
m_p(E) = \sum_{p' \leq p} \kappa_{p'}(E),
\end{align}
\end{theorem}

\begin{remarque}
Let us suppose that $(E_N)_{N \in \mathbb{N}}$ converges in moments, let $\phi_{E}^{\kappa}$ be the linear form in $(\mathbb{C}[\mathcal{P}_k])^{*}$ such that $\phi_{E}^{\kappa}(p) = \kappa_{p}(E)$ for any $p \in \mathcal{P}_k$: the theorem asserts that $\phi_{E}^{m} = \mathcal{M}(\phi_{E}^{\kappa})$, where we recall that $\mathcal{M}$ was defined in Definition \ref{def:moment}. 
\end{remarque}

\begin{proof}
Let $(E_N)_{N \in \mathbb{N}}$ be an element of $\prod\limits_{N \in \mathbb{N}} \mathbb{C}[\mathcal{P}_{k}(N)]$, let $p \in \mathcal{P}_k$ and let $N$ be a positive integer. Using the cumulants of $E_N$, we can calculate the $p$-normalized moment of $E_N$: 
\begin{align*}
m_p(E_N) 
&= \sum_{p' \in \mathcal{P}_k}  \kappa_{p'}(E_N) \frac{1}{N^{{\sf nc}(p')- {\sf nc}(p'\vee \mathrm{id}_k)}} m_{p}\left(  p'\right)\\
&= \sum_{p' \in \mathcal{P}_k} \kappa_{p'}(E_N) N^{{\sf nc}(p \vee p') - {\sf nc}(p \vee \mathrm{id}_k) - {\sf nc}(p') + {\sf nc}(p' \vee \mathrm{id}_k)},
\end{align*}
where we used the Equality (\ref{eq:combimom}). Hence, using Definition \ref{definitiondefect}: 
\begin{align}
\label{egal1}
m_p(E_N) = \sum_{p' \in \mathcal{P}_k} \kappa_{p'}(E_N)  N^{-{\sf df}(p',p)}. 
\end{align}
Let us suppose that for any $p' \in \mathcal{P}_k$,  $\kappa_{p'}(E_N)$ converges to a limit $\kappa_{p'}(E)$. The triangle inequality for $d$ shows that for any $p\in \mathcal{P}_k$, $m_{p}(E_N)$ converges when $N$ goes to infinity and: 
\begin{align*}
\lim_{N \to \infty} m_p(E_N) = \sum_{p' \leq p} \kappa_{p'}(E_N). 
\end{align*}
Now, let us suppose that it converges in moments. We can write the Equation $(\ref{egal1})$ as: 
$$m^{N} = G_{N} \kappa_{N}, $$
where $m^{N}\! =\! \left(m_p(E_N)\right)_{p \in \mathcal{P}_k(N)}, \kappa_{N}\! =\! \left(\kappa_{p}(E_N)\right)_{p \in \mathcal{P}_k(N)},$ and $G_N\! = \left( N^{-{\sf df}(p',p)}\right)_{p,p' \in \mathcal{P}_k(N)}. $
The sequence $\left(G_N\right)_{N \in \mathbb{N}}$ converges to the matrix $G$ of the order $\leq$: since $G$ is invertible, $\kappa_N = G_N^{-1} m^N$ converges to $G^{-1}m$ where $m=\big(\lim\limits_{N \to\infty}m_p(E_N)\big)_{p \in \mathcal{P}_k}$.
\end{proof}

From now on, we will say that $(E_N)_{N \in \mathbb{N}}$ converges if it converges in moments. Besides, for any set of partitions $P \subset  \mathcal{P}_k$, we will use the notation:
\begin{align}
\label{notationsurlim}
\kappa_{P}(E) = \sum_{p \in P}\kappa_{p}(E). 
\end{align}

Using Theorem \ref{th:caractK}, we can generalize easily Equation (\ref{kappamom}): we give one of these generalizations in the following proposition.

\begin{proposition}
\label{gene}
Let us suppose that $(E_N)_{N \in \mathbb{N}}$ converges. For any partitions $p_0$ and $p_1$ in $\mathcal{P}_k$ such that $p_1 \leq p_0$: 
\begin{align}
\label{eq:gene}
m_{\!\!\text{ }^{t}p_1\circ p_0} (E) &= \sum_{p'\leq p_0}  \kappa_{K_{p'}( p_1)}(E).
\end{align}
\end{proposition}

\begin{proof}
By linearity, it is enough to prove that Equation (\ref{eq:gene}) holds for $E_N = \frac{1}{N^{{\sf nc}(p_2) - {\sf nc}(p_2\vee \mathrm{id}_k)}} p_2$ for any integer $N$ and for a given partition $p_2\in \mathcal{P}_k$. For this choice of $(E_N)_{N \in \mathbb{N}}$, for any $p \in \mathcal{P}_k$, 
\begin{align*}
m_{p}(E) &= \delta_{p_2 \leq p}\\
\kappa_{p}(E)&= \delta_{p=p_2}. 
\end{align*}
Thus Equation (\ref{eq:gene}) is equivalent to: 
\begin{align*}
\delta_{p_1\leq p_0} \delta_{p_2 \leq \text{ }^{t}p_1\circ p_0} = \delta_{p_1\circ p_2 \leq p_0} \delta_{p_2 \in {\sf K}_{p_1 \circ p_2}(p_1)}
\end{align*}
which was proved in Theorem \ref{th:caractK}. 
\end{proof}

\subsection{The exclusive world}
In Section \ref{sec:rep}, we defined an other basis of $\mathbb{C}[\mathcal{P}_k]$, namely the exclusive basis. We can define the exclusive moments and the exclusive cumulants in the same way that we defined the moments and the cumulants but using $p^{c}$ instead of $p$. Let $N \in \mathbb{N}$, let $p$ be a partition in $\mathcal{P}_k$ and $E \in \mathbb{C}[\mathcal{P}_k(N)]$. 

\begin{definition}
The $p$-exclusive moment of $E$ is:  
\begin{align*}
m_{p^{c}} (E) = \frac{1}{{\sf Tr}_N( p)}{\sf Tr}_N(E \text{ }\!^tp^{c}). 
\end{align*}
\end{definition}
Let $p$ be a partition in $\mathcal{P}_k$ and $p'$ be a partition seen in $\mathbb{C}[\mathcal{P}_k(N)]$, 
\begin{align}
\label{eq:fondaexclu}
 Tr_{N}(p \text{ }^{t}(p'^{c})) =  \delta_{p\trianglelefteq p'} \frac{N!}{(N-{\sf nc}(p'))!}.
\end{align}

\begin{definition}
The exclusive cumulant of $E$ on $p$ is: 
\begin{align*}
\kappa_{p^{c}}(E) = N^{{\sf nc}(p) - {\sf nc}(p \vee \mathrm{id}_k))} (E)_{p^{c}}, 
\end{align*}
where $(E)_{p^{c}}$ is the coordinate of $E$ on $p^{c}$ in the exclusive basis $(p^{c})_{p \in \mathcal{P}_k}$. 
\end{definition}

The Theorem \ref{equivalence} can be extended to the following theorem. Let $(E_N)_{N \in \mathbb{N}}$  be in $\prod_{N=1}^{\infty} \mathbb{C}[\mathcal{P}_k(N)]$.

\begin{theorem}
\label{th:convexclu}
The sequence $(E_N)_{N \in \mathbb{N}}$ converges if and only if one of the two conditions holds:
\begin{enumerate}
\item for any $p \in \mathcal{P}_k$, $m_{p^{c}}(E_N)$ converges as $N$ goes to infinity to a number that we denote by $m_{p^{c}}(E)$. 
\item for any $p \in \mathcal{P}_k$, $\kappa_{p^{c}}(E_N)$ converges as $N$ goes to infinity to a number that we denote by $\kappa_{p^{c}}(E)$. 
\end{enumerate} 
 
Let us suppose that $(E_N)_{N \in \mathbb{N}}$ converges, for any $p\in \mathcal{P}_k$: 
\begin{align}
\label{kappamom2}
m_{p}(E) &= \sum_{p' \dashv p} m_{p'^{c}}(E), \\
\label{kappamom3}
\kappa_{p^{c}}(E) &= \sum_{p' \sqsupset p} \kappa_{p'}(E).
\end{align}
\end{theorem}

\begin{remarque}
Let us suppose that $(E_N)_{N \in \mathbb{N}}$ converges, let $\phi_{E}^{m^{c}}$ be the linear form in $(\mathbb{C}[\mathcal{P}_k])^{*}$ such that $\phi_{E}^{m^{c}}(p) = m_{p^{c}}(E)$ for any $p \in \mathcal{P}_k$: the theorem asserts that $\phi_{E}^{m} = \mathcal{M}^{c \to}(\phi_{E}^{m^{c}})$.
\end{remarque}

\begin{proof}
Let $(E_N)_{N \in \mathbb{N}}$ be an element of $\prod\limits_{N \in \mathbb{N}} \mathbb{C}[\mathcal{P}_k(N)]$. For any positive integer $N$: 
\begin{align*}
E_N &= \sum_{p \in \mathcal{P}_k}  \frac{\kappa_{p}(E_N)}{N^{{\sf nc}(p) - {\sf nc}(p \vee \mathrm{id}_k)}} p 
\\&= \sum_{p \in \mathcal{P}_k}  \frac{\kappa_{p}(E_N)}{N^{{\sf nc}(p) - {\sf nc}(p \vee \mathrm{id}_k)}} \sum_{p' \in \mathcal{P}_k\mid p\trianglelefteq p'} p'^{c} \\
&= \sum_{p \in \mathcal{P}_k, p' \in \mathcal{P}_k \mid  p \trianglelefteq p'} \kappa_{p}(E_N) N^{-{\sf nc}(p) + {\sf nc}(p\vee \mathrm{id}_k) +{\sf nc}(p')-{\sf nc}(p'\vee \mathrm{id}_k)} \frac{p'^c}{N^{{\sf nc}(p')-{\sf nc}(p'\vee \mathrm{id}_k)}}, 
\end{align*} 
and using Equation (\ref{definitiondefect}): 
\begin{align*}
E_N = \sum_{p' \in \mathcal{P}_k} \left(\sum_{p \in \mathcal{P}_k, p \trianglelefteq p'} \kappa_{p}(E_N) N^{-{\sf df}(p,p')} \right) \frac{p'^c}{N^{{\sf nc}(p')-{\sf nc}(p'\vee \mathrm{id}_k)}}. 
\end{align*}
Thus, for any integer $N$, for any $p' \in \mathcal{P}_k$
\begin{align}
\label{kappapc}
\kappa_{{p'}^{c}}(E_N) = \sum_{p \in \mathcal{P}_k, p \trianglelefteq p'} \kappa_{p}(E_N) N^{-{\sf df}(p,p')}. 
\end{align}
Besides, it is easy to see that:
\begin{align*}
m_{p}(E_N) = \sum_{ p \trianglelefteq p'} N^{{\sf nc}(p'\vee \mathrm{id}_k) - {\sf nc}(p \vee \mathrm{id}_k)}m_{p'^{c}}(E_N). 
\end{align*}
Using the same arguments as the one we used for Theorem \ref{equivalence}, we deduce the equivalence of the conditions for convergence. Besides we have that:
\begin{align*}
m_{p}(E) &= \sum_{ p\trianglelefteq p', p' \leq p} m_{p'^{c}}(E), \\
\kappa_{{p}^{c}}(E_N)& = \sum_{p' \leq p, p' \trianglelefteq p} \kappa_{p'}(E).
\end{align*}
Using Lemma \ref{lemme:carac1}, we deduce the Equalities (\ref{kappamom2}) and (\ref{kappamom3}). 
\end{proof}

Using Theorem \ref{th:convexclu}, we deduce the following theorem. From now on, let us suppose that $(E_N)_{N \in \mathbb{N}}$ converges.

\begin{theorem}
\label{egaliteexclu}
For any $p \in \mathcal{P}_k$, $m_{p^{c}}(E) = \kappa_{p^{c}}(E). $
\end{theorem}

\begin{proof}
Let $\phi_{E}^{\kappa^{c}}$ be the linear form in $(\mathbb{C}[\mathcal{P}_k])^{*}$ such that $\phi_{E}^{\kappa^{c}}(p) = \kappa_{p^{c}}(E)$ for any $p \in \mathcal{P}_k$. Using Theorems \ref{equivalence},  \ref{th:convexclu} and Equation (\ref{eq:decompo}), $\phi_{E}^{\kappa^{c}} = \mathcal{M}^{\to c}(\phi_{E}^{\kappa}) = \mathcal{M}^{\to c} \mathcal{M}^{-1}( \phi_E^{m}) = (\mathcal{M}^{c \to})^{-1} ( \phi_E^{m}) = \phi_{E}^{m^{c}}$.
\end{proof}

This result implies the equality $\phi_{E}^{m^{c}} = \mathcal{M}^{\to c}(\phi_{E}^{\kappa})$ which can be stated also in the following form. 
\begin{theorem}
\label{inversion}
For any $p \in \mathcal{P}_k$, $m_{p^{c}}\left(E\right) = \sum_{p' \sqsupset p } \kappa_{p'}(E).$
\end{theorem}

From this result, we get that if $p$ is a partition in  $\mathcal{P}_k$ which does not have any pivotal block, then: 
\begin{align}
\label{egalitepckappa}
 m_{p^{c}}\left( E\right) = \kappa_{p}\left(E\right). 
\end{align}
In particular, for any $p \in \mathcal{B}_k$, the Equality (\ref{egalitepckappa}) is satisfied.

\subsection{The special case: $\prod_{N=1}^{\infty}\mathbb{C}[\mathcal{A}_k(N)]$}
\subsubsection{Generalization}
\label{sec:generalization}
Using the same notations as in Section \ref{sec:projection}, we consider  $\mathcal{A} \subset \{\mathfrak{S}, \mathcal{B}, \mathcal{B}s, \mathcal{H}, \mathcal{P}\}$. When we consider an element $(E_N)_{N \in \mathbb{N}}$ of $\prod_{N=1}^{\infty} \mathbb{C}[\mathcal{A}_k(N)]$, we can consider the following notion of convergence. 

\begin{definition}
The sequence $(E_N)_{N \in \mathbb{N}}$ converges in $\mathcal{A}$-moments if for any $p \in \mathcal{A}_k$, $m_{p}(E_N)$ converges when $N$ goes to infinity. 
\end{definition}

Then, following a similar proof, Theorem \ref{equivalence} can be generalized easily.
\begin{theorem}
\label{th:equiA}
The sequence $(E_N)_{N \in \mathbb{N}}$ converges in $\mathcal{A}$-moments if and only if for any $p \in \mathcal{A}_k$, $\kappa_{p}(E_N)$ converges as $N$ goes to infinity. Let us suppose that  $(E_N)_{N \in \mathbb{N}}$ converges in $\mathcal{A}$-moments, for any $p \in \mathcal{A}_k$: 
\begin{align*}
m_{p}(E) = \sum_{p' \in \mathcal{A}_k | p' \leq p} \kappa_{p'}(E). 
\end{align*}
\end{theorem}

\begin{remarque}
If we use the same notations as before for $\phi_{E,\mathcal{A}}^{\kappa}$ and $\phi_{E,\mathcal{A}}^{m}$, except that they are elements of $(\mathbb{C}[\mathcal{A}_k])^{*}$, then $\phi_{E,\mathcal{A}}^{m} = \mathcal{M}_{\mathcal{A}}(\phi_{E,\mathcal{A}}^{\kappa})$.
\end{remarque}

\begin{theorem}
\label{equivalence2}
Let us suppose that $(E_N)_{N \in \mathbb{N}}$ converges in $\mathcal{A}$-moments then it converges in moments: for any $p\in \mathcal{P}_{k}$, the limit of $m_{p}(E_N)$ exists. Besides, for any $p \in \mathcal{P}_k$, the following equality holds:
$$ m_{p}(E) = \sum_{p' \in \mathcal{A}_k, p' \leq p} \kappa_{p'}(E).$$
\end{theorem}
\begin{proof}
Indeed, if $\left(E_N\right)_{N \in \mathbb{N}} \in \prod\limits_{N \in \mathbb{N}} \mathbb{C}[\mathcal{A}_{k}(N)]$ converges in $\mathcal{A}$-moments then, by Theorem \ref{th:equiA}, for any $p \in \mathcal{A}_{k}$, $\kappa_{p}(E_N)$ converges. By definition, for any $p \notin  \mathcal{A}_{k}$ and any integer $N$, $\kappa_{p}(E_N)=0$: for any $p \in \mathcal{P}_k$, $\kappa_{p}(E_N)$ converges and by Theorem \ref{equivalence}, for any $p\in \mathcal{P}_{k}$, the limit of $m_{p}(E_N)$ exists. The Equation (\ref{kappamom}) allows us to conclude. 
\end{proof}

Using Definition \ref{setbbarre} and Lemma \ref{pcoarsersigma}, when $\mathcal{A}\in \{\mathfrak{S}, \mathcal{B}\}$, the limit of the exclusive moments of $(E_N)_{N \in \mathbb{N}}$ are easy to compute. 

\begin{theorem}
Let us suppose that $\mathcal{A}\in \{\mathfrak{S}, \mathcal{B}\}$ and that $(E_N)_{N \in \mathbb{N}}$ converges in $\mathcal{A}$-moments. For any $p \in \mathcal{P}_k$: 
\begin{align*}
m_{p^c}(E) = \delta_{p \in \overline{\mathcal{A}_k}} \kappa_{{\sf Mb}(p)}(E).
\end{align*}
\end{theorem}

\subsubsection{Projection by integration}
In this section, we motivate the definitions and results obtained in Section \ref{sec:projection}. The notation $\mathcal{G}(\mathcal{A})$ was set in Section \ref{sec:projection} (Table 1). 
\begin{notation}
For any integer $N$, the notation: 
\begin{itemize}
\item $U(N)$ stands for the unitary group of size $N$, 
\item $O(N)$ stands for the orthogonal group of size $N$, 
\item $H(N)$ stands for the hyperoctahedral group of size $N$, which consists of matrices which have exactly one nonzero enty in each row and each column which is equal to $\pm 1$, 
\item $B(N)$ stands for the orthogonal bistochastic group of size $N$, which consists of orthogonal matrices having sum $1$ in each row and each column. 
\end{itemize}
\end{notation}

Let $(E_N)_{N \in \mathbb{N}}$ be an element of $\prod_{N \in \mathbb{N}} \mathbb{C}[\mathcal{P}_k(N)]$ which converges when $N$ goes to infinity. For any positive integer $N$, we define $$ E^{\mathcal{G}(\mathcal{A})}_N = \int_{\mathcal{G}(\mathcal{A})(N)} g^{\otimes k}\rho_{N}(E_N) (g^{*})^{\otimes k} dg,$$
where $dg$ is the Haar probability measure on $\mathcal{G}(\mathcal{A})(N)$ and $\rho_{N}$ was defined in Section~\ref{sec:rep}. Recall Remark \ref{rq:restrictionak}. 
 
\begin{proposition} 
There exists a sequence $(\mathbb{E}_N)_{N \in \mathbb{N}} \in \prod_{N \in \mathbb{N}} \mathbb{C}[\mathcal{A}_k(N)]$ such that for any positive integer $N$: $\rho_{N}(\mathbb{E}_N) = E^{\mathcal{G}(\mathcal{A})}_N.$ The sequence $(\mathbb{E}_N)_{N \in \mathbb{N}}$ converges as $N$ goes to infinity and the three following equalities hold: 
\begin{align}
\label{eq:projection}
\phi_{\mathbb{E}}^{\kappa} = \mathcal{C}_{\mathcal{A}}^{\kappa} \left( \phi_E^{\kappa} \right), \ \ \ \ \  \phi_{\mathbb{E}}^{m} =  \mathcal{C}_{\mathcal{A}}^{m} \left(\phi_E^{m} \right), \ \ \ \ \  \phi_{\mathbb{E}}^{m^{c}} =  \mathcal{C}_{\mathcal{A}}^{m^{c}} \left( \phi_E^{m^{c}} \right).
\end{align}
\end{proposition}

\begin{proof}
The first part of the proof is a direct consequence of the Schur-Weyl duality for the unitary, orthogonal groups \cite{goodman}. Let us prove that $(\mathbb{E}_N)_{N \in \mathbb{N}}$ converges as $N$ goes to infinity and that Equations \ref{eq:projection} holds. Let $p$ be in $\mathcal{A}_k$ and let $N$ be a positive integer. Using the traciality of ${\sf Tr}^{k}$, $m_{p}(\mathbb{E}_N)$ is equal to:  
\begin{align*}
\frac{1}{N^{{\sf nc}(p \vee \mathrm{id}_k)}} {\sf Tr}^{k}\left(\rho_{N}(\mathbb{E}_n) \rho_{N}(\!\!\text{ }^{t}p) \right)=\frac{1}{N^{{\sf nc}(p \vee \mathrm{id}_k)}} \int_{\mathcal{G}(\mathcal{A})(N)} {\sf Tr}^{k}\left( \rho_{N}^{\mathcal{P}_k}(E_N) (g^{*})^{\otimes k} \rho_{N}(\!\!\text{ }^{t}p) g^{\otimes k}\right)dg. 
\end{align*}
For any $g\in \mathcal{G}(\mathcal{A})(N)$, $(g^{*})^{\otimes k} \rho_{N}(\!\!\text{ }^{t}p) g^{\otimes k} = \rho_N(^{t}p)$, thus $m_{p}(\mathbb{E}_N)=m_{p}(E_N). $ It implies that for any $p\in \mathcal{A}_k$, $m_{p}(\mathbb{E}_N)$ converges as $N$ goes to infinity: by Theorem \ref{equivalence2}, $(\mathbb{E}_{N})_{N \in \mathbb{N}}$ converges and: 
\begin{align*}
{\sf R}_{\mathcal{A}} (\phi_{\mathbb{E}}^{m}) = {\sf R}_{\mathcal{A}} (\phi_{E}^{m}).
\end{align*}
Let us recall that $\phi_{\mathbb{E}}^{\kappa} = {\sf E}_{\mathcal{A}} \circ {\sf R}_{\mathcal{A}} (\phi_{\mathbb{E}}^{\kappa} )$ since $(\mathbb{E}_N)_{N \in \mathbb{N}} \in \prod_{N \in \mathbb{N}} \mathbb{C}[\mathcal{A}_k(N)]$, besides $ {\sf R}_{\mathcal{A}} (\phi_{\mathbb{E}}^{\kappa} ) = \mathcal{R}_{\mathcal{A}} \circ {\sf R}_{\mathcal{A}} (\phi_{\mathbb{E}}^{m})$, thus: 
\begin{align*}
\phi_{\mathbb{E}}^{\kappa} = {\sf E}_{\mathcal{A}} \circ {\sf R}_{\mathcal{A}} (\phi_{\mathbb{E}}^{\kappa}) \!=\!  {\sf E}_{\mathcal{A}} \circ \mathcal{R}_{\mathcal{A}} \circ {\sf R}_{\mathcal{A}} (\phi_{\mathbb{E}}^{m}) = {\sf E}_{\mathcal{A}} \circ \mathcal{R}_{\mathcal{A}} \circ {\sf R}_{\mathcal{A}}(\phi_{E}^{m}) &={\sf E}_{\mathcal{A}} \circ \mathcal{R}_{\mathcal{A}} \circ {\sf R}_{\mathcal{A}} \circ \mathcal{M}  (\phi_{E}^{\kappa})\\
&= \mathcal{C}_{\mathcal{A}}^{\kappa}(\phi_{E}^{\kappa}).
\end{align*}
The two other equalities can be proved with the same kind of computations.
\end{proof}

\subsection{Convergence of the modified algebras}

Let us define a deformation of the partition algebra $\mathbb{C}[\mathcal{P}_k(N)]$ by modifying the multiplication which was set in Definition \ref{multiplication}. This deformation is motivated by the fact that for any $p \in \mathcal{P}_k$, the sequence in $\prod_{N =1}^{\infty} \mathbb{C}[\mathcal{P}_{k}(N)]$ which is constant and equal to $p$ does not converge. Thus, the basis $(p)_{p \in \mathcal{P}_k}$ is not a good basis for the study of the asymptotic of the partition algebra $\mathbb{C}[\mathcal{P}_k(N)]$.

\begin{definition}
\label{MN}
We define the application: 
\begin{align*}
M_N: \mathcal{P}_k &\to \mathcal{P}_k\\
p &\mapsto \frac{1}{N^{{\sf nc}(p) - {\sf nc}(p \vee \mathrm{id}_k)}}p.
\end{align*}
This application can be extended as an isomorphism of vector spaces from $\mathbb{C}[\mathcal{P}_k]$ to itself. 
\end{definition}

The definition of $M_N$ was set such that for any $E \in \mathbb{C}[\mathcal{P}_k(N)]$, $\left(M_k^{N}\right)^{-1}(E)$ is simply equal to $\sum_{p \in \mathcal{P}_k} \kappa_{p}(E) p$: a sequence $(E_N)_{N \in \mathbb{N}} \in \prod_{N=1}^{\infty} \mathbb{C}[\mathcal{P}_k(N)]$ converges if and only if $(M_{N})^{-1}(E_N)$ converges for the usual convergence in finite dimensional vector spaces. 

Let us remark also that for any $p$ in $\mathcal{P}_k$ and any $p'$ in $\mathcal{P}_{k'}$, 
\begin{align}
\label{Mcomp}
M^{N}(p \otimes p') = M^{N}( p) \otimes M^{N}(p'). 
\end{align}
We can now define the deformed algebra $\mathbb{C}[\mathcal{P}_k(N,N)]$. 

\begin{definition}
We endow $\mathbb{C}[\mathcal{P}_k]$ with a structure of associative algebra by taking the pullback of the structure of algebra of $\mathbb{C}[\mathcal{P}_k(N)]$ by $M_{N}$. For any $p_1, p_2$ in $\mathcal{P}_{k}$ the new product of $p_1$ with $p_2$ is given by: 
\begin{align*}
p_1._{\!{}_N} p_2 =  \left(M_{N}\right)^{-1} \big[ M_{N}(p_1) M_{N}(p_2) \big]. 
\end{align*} 
This is the deformed algebra $\mathbb{C}[\mathcal{P}_k(N,N)]$. 
\end{definition}

The application $M_{N}$ can be extended as an isomorphism of algebra from ${\mathbb{C}[\mathcal{P}_k(N,N)]}$ to $\mathbb{C}[\mathcal{P}_k(N)].$ Its extension will be also denoted by $M_{N}$. In the following, we study the limit of the deformed algebras in the following sense. 

\begin{definition}
\label{convalggen}
Let $C$ be a finite set of elements. For any $N \in \mathbb{N}\cup \{\infty\}$, let $A_N$ be an algebra such that $C$ is a linear basis of $L_N$. For any elements $x$ and $y$ of $C$, for each $N \in \mathbb{N}\cup \{\infty\}$, we denote the product of $x$ with $y$ in $L_N$ by $x._{\!{}_N}y$. The algebra $A_N$ converges to the algebra $A_\infty$ when $N$ goes to infinity if for any $x$ and $y$ in $C$, $$x._{\!{}_N}y \underset{N \to \infty}{\longrightarrow}x._{\!{}_\infty} y\ \  \text{ in }\ \  \mathbb{C}[C],$$
for the usual notion of convergence in finite dimensional linear spaces. 
\end{definition}

\begin{theorem}
\label{Convalg}
The deformed algebra $\mathbb{C}[\mathcal{P}_k(N,N)]$ converges, when $N$ goes to infinity, to the deformed algebra $\mathbb{C}[\mathcal{P}_{k}(\infty,\infty)]$ which is the associative algebra over $\mathbb{C}$ with basis $\mathcal{P}_k$ endowed with the multiplication defined by: 
\begin{align*}
\forall p, p' \in \mathcal{P}_k,\ p._{\!{}_\infty}p' =  \delta_{p' \in {\sf K}_{p\circ p'}(p) }\ (p\circ p'),
\end{align*}
where the Kreweras complement was defined in Definition \ref{geo2}.
\end{theorem}

\begin{proof}
For any $N \in \mathbb{N} \cup \{\infty\}$, $\mathcal{P}_k$ is a linear basis of $\mathbb{C}[\mathcal{P}_k(N,N)]$. It is enough to prove that for any $p$ and $p'$ in $\mathcal{P}_k$, $p._{\!{}_N}p'$ converges to $\delta_{p \prec p \circ p'} p \circ p'$. By a straightforward computation: 
\begin{align*}
p._{\!{}_N}p' = N^{d(\mathrm{id}_k,p\circ p') - d(\mathrm{id}_k,p) - d(\mathrm{id}_k, p') + \frac{k + {\sf nc}(p \circ p') - {\sf nc}(p ) - {\sf nc}(p')}{2}+ \kappa(p,p')} (p \circ p') = N^{-\eta(p,p')} (p\circ p'). 
\end{align*}
Using Inequality \ref{autreequation} and Definition \ref{geo2}, $p._{\!{}_N}p' \underset{N \to \infty}{\longrightarrow} \delta_{p' \in {\sf K}_{p\circ p'}(p) }\ (p\circ p').$
\end{proof}

\subsection{Some consequences}

\subsubsection{Combinatorial consequences}

Using the associativity of the product $._{\!{}_N}$ and its limit when $N$ goes to infinity, one can deduce the folllowing proposition. 
\begin{proposition}
\label{prop:transprec}
The relation $\prec$, defined in Definition \ref{geo2}, is transitive.
\end{proposition}

\begin{proof}
Let $p_1$, $p_2$, $p_3$ be three partitions in $\mathcal{P}_k$. Let us consider the product $p_1._{\!{}_N}p_2._{\!{}_N}p_3$. We can compute the limit of this product in two ways by looking either at $(p_1._{\!{}_N}p_2)._{\!{}_N}p_3$ or $p_1._{\!{}_N}(p_2._{\!{}_N}p_3)$. We get two limits which are equal and considering the coefficients, we get: 
\begin{align}
\label{eq:deltaassociatif}\delta_{p_2 \in {\sf K}_{p_1\circ p_2}(p_1)}\delta_{p_3 \in {\sf K}_{p_1\circ p_2 \circ p_3}(p_1\circ p_2)} = \delta_{p_3 \in {\sf K}_{p_2\circ p_3}(p_2)}\delta_{p_2 \circ p_3 \in {\sf K}_{p_1\circ p_2 \circ p_3}(p_1)}.
\end{align}

Let us consider $p_1$, $p'$ and $p''$ in $\mathcal{P}_k$ such that $p_1\prec p'$ and $p' \prec p''$. There exists $p_2$ and $p_3$ such that $p_1 \circ p_2 = p'$, $p_1 \circ p_2 \circ p_3 = p''$, $p_2 \in {\sf K}_{p_1\circ p_2}(p_1)$ and $p_3 \in {\sf K}_{p_1\circ p_2 \circ p_3}(p_1\circ p_2)$. Using Equation (\ref{eq:deltaassociatif}), we get that: 
\begin{align*} 
\delta_{p_3 \in {\sf K}_{p_2\circ p_3}(p_2)}\delta_{p_2 \circ p_3 \in {\sf K}_{p_1\circ p_2 \circ p_3}(p_1)}\ne 0. 
\end{align*}
In particular, $p_2 \circ p_3 \in {\sf K}_{p_1\circ p_2 \circ p_3}(p_1)$. Let $\tilde{p}$ be equal to $p_2 \circ p_3$, then $\tilde{p} \in {\sf K}_{p''}(p_1)$ and $p_1 \circ \tilde{p} = p''$: it proves that $p_1 \prec p''$. 
\end{proof}

\subsubsection{Convergence of the product}

A consequence of Theorem \ref{Convalg} is the continuity of the product for the notion of convergence in moments.

\begin{theorem}
\label{multi1}
Let $(E_N)_{N \in \mathbb{N}}$ and $(F_N)_{N \in \mathbb{N}}$ be two elements of $\prod\limits_{N \in \mathbb{N}}\mathbb{C}[\mathcal{P}_k(N)]$. Let us suppose that $(E_N)_{N \in \mathbb{N}}$ and $(F_N)_{N \in \mathbb{N}}$ converge (in moments), then the sequence $\big(E_NF_N\big)_{N \in \mathbb{N}}$ converges (in moments). Besides:
\begin{align}
\label{1''} \phi_{EF}^{\kappa} = \phi_{E}^{\kappa} \boxtimes \phi_{F}^{\kappa},\ \ \ \ \ \ \ \ \ \ 
 \phi_{EF}^{m} =  \phi_{E}^{\kappa} \boxtimes_{d}^{m} \phi_{F}^{m},\ \ \ \ \ \ \ \ \ \ \phi_{EF}^{m} =  \phi_{E}^{m} \boxtimes_{g}^{m} \phi_{F}^{\kappa}.
\end{align}
\end{theorem}

\begin{remarque}
Using the notations $\kappa_{p}$ and $m_{p}$, the Equations (\ref{1''}) can be written in the following form: for any $p_0\in \mathcal{P}_k$: 
\begin{align}
\label{1}\kappa_{p_0}(EF) &= \sum_{p \in \mathcal{P}_k, p \prec p_0} \kappa_{p}(E)\kappa_{{\sf K}_{p_0}( p)}(F), \\
\label{2}\ m_{p_0} (EF) &= \sum_{p \in \mathcal{P}_k, p \leq p_0} \kappa_{p}(E) m_{\!\!\text{ }^{t}p \circ p_0}(F),\\
\label{mom2}
m_{p_0}(E F) &= \sum_{p \in \mathcal{P}_k, p\leq p_0} m_{p_0\circ \!\!\text{ }^{t}p}(E)\kappa_p(F). 
\end{align}
\end{remarque}

\begin{proof}
By definition for any integer $N$, $(M_{N})^{-1}(E_NF_N)=\left(M_{N}\right)^{-1}(E_N)._{\!{}_N} \left(M_{N}\right)^{-1}(F_N).$ But $\left(M_{N}\right)^{-1}(E_N)$ and $\left(M_{N}\right)^{-1}(F_N)$, seen as elements of $\mathbb{C}[\mathcal{P}_k]$, converge when $N$ goes to infinity. Besides, the algebra $\mathbb{C}[\mathcal{P}_k(N,N)]$ converges to $\mathbb{C}[\mathcal{P}_{k}(\infty, \infty)]$, as it was proved in Theorem \ref{Convalg}. Thus $(M_{k}^{N})^{-1}(E_NF_N)$ converges when $N$ goes to infinity. Again this shows that $\left(E_NF_N\right)_{N \in \mathbb{N}}$ converges. Besides: 
\begin{align*}
(M_{N})^{-1}(E_NF_N)&=\sum_{p\in \mathcal{P}_k} \kappa_{p_0}(E_NF_N) p_0, \\
\left(M^{N}\right)^{-1}(E_N)._N \left(M_{N}\right)^{-1}(F_N) &= \sum_{p\in \mathcal{P}_k, p' \in \mathcal{P}_k} \kappa_{p}(E_N)\kappa_{p'}(F_N) p._{\!{}_N}p'. 
\end{align*}
Using the formula for the limit of $._{\!{}_\infty}$ in Theorem \ref{Convalg}, for any $p_0 \in \mathcal{P}_k$: 
\begin{align*}
\kappa_{p_0}(EF) = \sum_{p \in \mathcal{P}_k, p \prec p_0}  \kappa_{p}(E)\kappa_{{\sf K}_{p_0}( p)}(F),
\end{align*}
which gives us the first equality of (\ref{1''}) (or equivalently Equation (\ref{1})).

The other equalities are consequences of the Equation $(\ref{eq:metconvol})$ since, by Theorem \ref{equivalence}, $\phi_{EF}^{m} = \mathcal{M}(\phi_{EF}^{\kappa})$, $\phi_{E}^{m} = \mathcal{M}(\phi_{E}^{\kappa})$ and $\phi_{F}^{m} = \mathcal{M}(\phi_{F}^{\kappa})$.
\end{proof}

\subsubsection{Convergence of multiplicative semi-groups}

\label{semi-groups}

\begin{definition}
A family $\left((E^N_{t})_{N}\right)_{t \geq0}$ of elements of $\prod_{N \in \mathbb{N}} \mathbb{C}[\mathcal{P}_k(N)]$ is a semi-group if there exists $(H_N)_{N \in \mathbb{N}} \in \prod_{N \in \mathbb{N}} \mathbb{C}[\mathcal{P}_k(N)]$, called the generator, such that for any $t \geq 0$, for any integer $N$: 
\begin{align*}
\frac{d}{dt}_{\mid t=t_0} E^N_t = H_N E^N_{t_0}. 
\end{align*}
\end{definition}

Let us suppose that $\left((E^N_{t})_{N}\right)_{t \geq0}$ is a semi-group in $\prod_{N \in \mathbb{N}} \mathbb{C}[\mathcal{P}_k(N)]$ whose generator is $(H_N)_{N \in \mathbb{N}}$. 

\begin{definition}
\label{convergencesemigroup}
The semi-group $\left((E^N_{t})_{N}\right)_{t \geq0}$ converges if and only if for any $t \geq 0$, $E^N_t$ converges as $N$ goes to infinity. 
\end{definition}

The next theorem shows that a semi-group in $\prod_{N \in \mathbb{N}} \mathbb{C}[\mathcal{P}_k(N)]$ converges if the initial condition and the generator converge. 

\begin{theorem}
\label{semigroup}
The semi-group $\left((E^N_{t})_{N}\right)_{t \geq0}$ converges if the sequences $(E^N_{0})_{N \in \mathbb{N}}$ and $(H_N)_{N \in \mathbb{N}}$ converge as $N$ goes to infinity. Besides, we have the three differential systems of equations: for any $t_0 \geq 0$: 
\begin{align}
\label{eq:diff}
\frac{d}{dt}_{| t=t_0}\phi_{E_t}^{\kappa} = \phi_{H}^{\kappa} \boxtimes \phi_{E_{t_0}}^{\kappa}, \ \ \ 
\frac{d}{dt}_{| t=t_0}\phi_{E_t}^{m} = \phi_{H}^{\kappa} \boxtimes_{d}^{m} \phi_{E_{t_0}}^{m}, \ \ \ 
\frac{d}{dt}_{| t=t_0}\phi_{E_t}^{m} = \phi_{H}^{m} \boxtimes_{g}^{m} \phi_{E_{t_0}}^{\kappa}.
\end{align}

\end{theorem}

\begin{remarque}
Using the notations $\kappa_{p}$ and $m_{p}$, the Equations (\ref{eq:diff}) can be written in the following form: for any $p \in \mathcal{P}_k$, for any $t_0 \geq 0$: 
\begin{align}
\label{1'}\frac{d}{dt}_{\mid t=t_0}\kappa_{p} (E_t) &= \sum_{p_1 \in \mathcal{P}_k, p_1 \prec p} \kappa_{p_1} (H) \kappa_{{\sf K}_p(p_1)}(E_{t_0}), \
\\
\label{2'}\frac{d}{dt}_{\mid t=t_0}m_{p} (E_t) &= \sum_{p_1\in \mathcal{P}_k, p_1 \leq p } \kappa_{p_1} (H) m_{\!\!\text{ }^{t}p_1 \circ p}(E_{t_0}),\\
\label{3'} \frac{d}{dt}_{\mid t=t_0}m_{p}(E_t) &= \sum_{p_1 \in \mathcal{P}_k,  p_1\leq p} m_{p\circ \!\!\text{ }^{t}p_1}(H) \kappa_{p_1}(E_{t_0}).
\end{align}
\end{remarque}

\begin{proof}
Let us suppose that $(E^N_{0})_{N \in \mathbb{N}}$ and $(H_N)_{N \in \mathbb{N}}$ converge. For any integer $N$ and any $t\geq 0$, we define ${\bold E_t^N} = (M_{N})^{-1}(E_t^{N})$ and ${\bold H_N}=(M_{N})^{-1}(H_{N}).$ Since $M_{N}$ is a morphism of algebra, the family $\big(({\bold E_t^N})_{N \in \mathbb{N}}\big)_{t \geq 0}$ is a semi-group in $\prod\limits_{N \in \mathbb{N}}\mathbb{C}[\mathcal{P}_k(N,N)]$ and its generator is $\big({\bold H_N}\big)_{N \in \mathbb{N}}$. For any $t_0 \geq 0$:
\begin{align*}
\frac{d}{dt}_{\mid t=t_0} \sum_{p_0 \in \mathcal{P}_k} \kappa_{p_0}(E^{N}_t) p_0 = \bigg(\sum_{p \in \mathcal{P}_k} \kappa_{p}(H_{N})p\bigg) ._{\!{}_N} \bigg(\sum_{p' \in \mathcal{P}_k} \kappa_{p'}(E^{N}_{t_0})p'\bigg). 
\end{align*}
Thus, for any $p_0 \in \mathcal{P}_k$: 
\begin{align*}
\frac{d}{dt}_{\mid t=t_0}&\kappa_{p_0}(E^{N}_t) = \sum_{p,p' \in \mathcal{P}_k,\ p\circ p' =p_0} \kappa_{p}(H_{N}) \kappa_{p'}(E^{N}_t)N^{-\eta(p,p')}. 
\end{align*}

When $N$ goes to infinity, because of the hypotheses and since the $\prec$-defect is always positive, this differential system converges: $\kappa_{p}(E^{N}_t)$ converges for any $p\in \mathcal{P}_k$ and any real $t \geq 0$. Besides, for any $t_0\geq 0$, for any $p \in \mathcal{P}_k$: 
\begin{align*}
\frac{d}{dt}_{\mid t=t_0}\kappa_{p} (E_t) &= \sum_{p_1 \in \mathcal{P}_k, p_1 \prec p} \kappa_{p_1} (H) \kappa_{{\sf K}_p(p_1)}(E_{t_0}). 
\end{align*}
This gives us the first equation in (\ref{eq:diff}). The other equations are again consequences of Proposition \ref{eq:metconvol}.
\end{proof}

\begin{remarque}
\label{rq:genersemigroup}
Let the letter $\mathcal{A}$ stand either for $\mathfrak{S}$ or $\mathcal{B}$. If $((E_t^N)_{N \in \mathbb{N}})_{t \geq 0}$ is a semi-group in $\prod_{N \in \mathbb{N}} \mathbb{C}[\mathcal{A}_k(N)]$, we can state a more general result: if the sequences $(E^N_{0})_{N \in \mathbb{N}}$ and $(H_N)_{N \in \mathbb{N}}$ converge in $\mathcal{A}$-moments as $N$ goes to infinity then $((E_t^N)_{N \in \mathbb{N}})_{t \geq 0}$ converges in $\mathcal{P}$-moments.
\end{remarque}

\section{Fluctuations}

\label{sec:Fluctuations}
In this section, we generalize Section \ref{sec:obsconv} in order to study the asymptotic developments of the moments and cumulants. This would allow us to generalize the notions of $\mathcal{M}$-, $\mathcal{R}$-transforms, $\boxtimes$ and $\boxplus$ convolutions which would be defined on $(\bigoplus_{k=0}^{\infty} \mathbb{C}[(\mathcal{P}_k/\mathfrak{S}_k) \times \{0,...,n\}])^{*}$: the definitions are straightfoward after reading this section. Yet the characters of $(\bigoplus_{k=0}^{\infty} \mathbb{C}[(\mathcal{P}_k/\mathfrak{S}_k) \times \{0,...,n\}])^{*}$ are not interesting in the asymptotic study of random matrices which are invariant in law by conjugation by the symmetric group thus we will not spend more time to explain these generalizations, only will we give short definitions in \ref{notation:notationsfluctu} in order to simplify results in the following article \cite{Gab2}. This section can also be easily generalized for the sets of partitions $\mathcal{A}$ when $\mathcal{A} \in \{\mathfrak{S}, \mathcal{B}, \mathcal{B}s, \mathcal{H}\}$.

Let $(E_N)_{N \in \mathbb{N}}$ be an element of $\prod_{N \in \mathbb{N}} \mathbb{C}[\mathcal{P}_k(N)]$. We  define and study a notion of strong convergence up to the $n^{th}$ order of fluctuations.

\begin{definition}
\label{convergencemomentfluctuation}
The sequence $(E_N)_{N \in \mathbb{N}}$ converges in moments up to the $n^{th}$ order of fluctuations if for any $p \in \mathcal{P}_k$, for any $i \in \{0,...,n\}$, there exist a real $m_{p}^{i}(E)$ such that: 
\begin{align*}
N^{n} \left( m_{p}(E_N) - \sum_{i=0}^{n-1} \frac{m_{p}^{i}(E)}{N^{i}} \right) \underset{ N \to \infty}{\longrightarrow} m_{p}^{n}(E)
\end{align*}
\end{definition}

\subsection{Cumulants of higher orders and the key result}
The generalization of Theorem \ref{equivalence} is given in the following theorem. 

\begin{theorem}
\label{equivalencefortmomentfluctu}
The sequence  $(E_N)_{N \in \mathbb{N}}$ converges in moments up to the $n^{th}$ order of fluctuations if and only if for any $p \in \mathcal{P}_k$, for any $i \in \{0,...,n\}$, there exists a real $\kappa_{p}^{i}(E)$ such that for any $p \in \mathcal{P}_k$: 
\begin{align}
\label{eq:cumufluctu}
N^{n} \left( \kappa_{p}(E_N) - \sum_{i=0}^{n-1} \frac{\kappa_{p}^{i}(E)}{N^{i}} \right) \underset{ N \to \infty}{\longrightarrow} \kappa_{p}^{n}(E)
\end{align}
Let us suppose that  $(E_N)_{N \in \mathbb{N}}$ converges in moments up to the $n^{th}$ order of fluctuations, for any $p \in \mathcal{P}_k$, for any $i_0 \in \{0,…,n\}$: 
\begin{align}
\label{cumulantmomentfluctuation}
m^{i_0}_p(E) = \sum_{p' \in \mathcal{P}_k, \text{\sf df}{(p',p)}\leq i_0}\kappa_{p'}^{i_0- \text{\sf df}{(p',p)}}(E).
\end{align}
\end{theorem}

\begin{proof}
Let us suppose that for any $p \in \mathcal{P}_k$, for any $i \in \{0,...,n\}$, there exist a real $\kappa_{p}^{i}(E)$ such that: 
\begin{align}
\label{eq:cumu1}
\kappa^{n}_{p}(E_N) := N^{n}\left( \kappa_{p}(E_N) - \sum_{i=0}^{n-1} \frac{\kappa_{p}^{i}(E)}{N^{i}} \right) \underset{ N \to \infty}{\longrightarrow} \kappa_{p}^{n}(E)
\end{align}

Let us denote for any $i \leq n-1$ and any $p\in \mathcal{P}_k$, $\kappa_{p}^{i}(E_N) = \kappa_{p}^{i}(E)$. This change of notations allows us to write the Equation (\ref{eq:cumu1}) as following: 
\begin{align*}
E_N = \sum_{p \in \mathcal{P}_k} \sum_{i=0}^{n}\frac{\kappa^{i}_{p}(E_N)}{N^{i}}\frac{p}{N^{{\sf nc}(p) - {\sf nc}(p \vee \mathrm{id}_k)}}. 
\end{align*}
We can compute the $p$-moment of $E_N$: 
\begin{align*}
m_{p}(E_N) = \frac{1}{{\sf Tr}_N( p)}{\sf Tr}_N(E_N \!\text{ }^{t}p) &= \sum_{p' \in \mathcal{P}_k} \sum_{i=0}^{n} \kappa_{p'}^{i}(E_N) \frac{1}{N^{i+{\sf df}(p',p)}} \\
&= \sum_{j=0}^{n-1} \left(\sum_{(p',i) \in \mathcal{P}_k \times \{0,…,n-1\}, i+{\sf df}(p',p) = j} \kappa^{i}_{p'}(E_N)\right) \frac{1}{N^{j}}  \\
&\ \ \ \ \  + \left(\sum_{(p',i) \in \mathcal{P}_k \times \{0,…,n\}, i+{\sf df}(p',p) \geq n} \frac{\kappa^{i}_{p'}(E_N)}{N^{i+{\sf df}(p',p) - n}}\right)\frac{1}{N^{n}}.
\end{align*}
For any $N \in \mathbb{N}$, any $j \in \{0,…,n-1\}$ and any $p \in \mathcal{P}_k$: 
\begin{align*}
m_{p}^{j}(E_N) =  \sum_{(p',i) \in \mathcal{P}_k \times \{0,…,n-1\}, i+{\sf df}(p',p) = j} \kappa^{i}_{p'}(E_N)
\end{align*}
and $$m_{p}^{n}(E_N) = \sum_{(p',i) \in \mathcal{P}_k \times \{0,…,n\}, i+{\sf df}(p',p) \geq n} \frac{\kappa^{i}_{p'}(E_N)}{N^{i+{\sf df}(p',p) - n}}, $$
so that, for any $p \in \mathcal{P}_k$ and any $N\in \mathbb{N}$:
$$m_{p}(E_N) = \sum_{j=0}^{n-1} \frac{m^{i}_{p}(E_N)}{N^{j}} + \frac{m_{p}^{n}(E_N)}{N^{n}}.$$
For any $p \in \mathcal{P}_k$ and any $i \leq n-1$, $m_{p}^{i}(E_N)$ does not depend on $N$ and $m_{p}^{n}(E_N)$ converges to 
$$\sum_{p'\in \mathcal{P}_k, {\sf df}(p',p)\leq n} \kappa^{n-{\sf df}(p',p)}_{p'}(E)$$
when $N$ goes to infinity. This shows that $(E_N)_{N \in \mathbb{N}}$ converges in moments up to the $n^{th}$ order of fluctuations and Equation (\ref{cumulantmomentfluctuation}) holds. 

Let us suppose now that $(E_N)_{N \in \mathbb{N}}$ converges in moments up to the $n^{th}$ order of fluctuations. We prove that Equations (\ref{eq:cumufluctu}) and  (\ref{cumulantmomentfluctuation}) hold by recurrence on $n$. Since $(E_N)_{N \in \mathbb{N}}$ converges, by Theorem \ref{equivalence}, $\kappa_{p}(E_N)$ converges to $\kappa_{p}(E)$. Let us suppose that Equation (\ref{eq:cumufluctu}) holds for $l < n$: for any $p \in \mathcal{P}_k$, any $i \in \{0,...,l\}$, there exists $\kappa_{p}^{i}(E)$ and for any $N$ there exists $\kappa_{p}^l(E_N)$ which converges to $\kappa_{p}^{l}(E)$ such that: 
\begin{align*}
E_N= \sum_{p \in \mathcal{P}_k} \left(\sum_{j=0}^{l-1}\frac{\kappa_{p}^j(E)}{N^{j}} + \frac{\kappa_{p}^{l}(E_N)}{N^{l}} \right) p, 
\end{align*}
Using the computations we already did: 

\begin{align*}
m_p(E_N) &= \sum_{j=0}^{l-1} \left(\sum_{(p',i) \in \mathcal{P}_k \times \{0,…,l-1\}, i+{\sf df}(p',p) = j} \kappa^{i}_{p'}(E)\right) \frac{1}{N^{j}}  \\
&\ \ \ \ \  + \left(\sum_{(p',i) \in \mathcal{P}_k \times \{0,…,l\}, i+{\sf df}(p',p) \geq l} \frac{\kappa^{i}_{p'}(E_N)}{N^{i+{\sf df}(p',p) - l}}\right)\frac{1}{N^{l}}. 
\end{align*}
Using Equation (\ref{cumulantmomentfluctuation}) for $i_0 \leq l$, we get: 
\begin{align*}
m_p(E_N) = \sum_{j=0}^{l} \frac{m_p^{j}(E)}{N^{j}}&+\sum_{p' \in \mathcal{P}_k, {\sf df}(p',p) = 0} \frac{\kappa_{p'}^{l}(E_N)-\kappa_{p'}^{l}(E)}{N^{l}}
\\& + \sum_{(p',i) \in \mathcal{P}_k\times \{0,...,l\}, i+{\sf df}(p',p)-l = 1} \frac{\kappa_{p'}^{i}(E_N)}{N^{l+1}} + o\left(\frac{1}{N^{l+1}}\right). 
\end{align*}
Since $(E_N)_{N \in \mathbb{N}}$ converges in moments up to the order $l+1$ of fluctuations: for any $p' \in \mathcal{P}_k$,
\begin{align*} 
N^{l+1}\left(m_{p}(E_N) - \sum_{j=0}^{l}\frac{m_{p}^{j}(E)}{N^{j}}\right)
\end{align*}
converges as $N$ goes to infinity. This implies that for any $p \in \mathcal{P}_k$, 
\begin{align*}
\sum_{p' \in [\mathrm{id}_k, p]} N(\kappa_{p'}^{l}(E_N)-\kappa_{p'}^{l}(E))
\end{align*}
converges as $N$ goes to infinity. By inverting the order $\leq$, for any $p \in \mathcal{P}_k$: 
\begin{align*}
N(\kappa_{p'}^{l}(E_N)-\kappa_{p'}^{l}(E))
\end{align*}
converges as $N$ goes to infinity. Thus Equation (\ref{eq:cumufluctu}) holds for $n=l+1$, and using the first part of the proof, Equation  (\ref{cumulantmomentfluctuation}) holds for $n=l+1$. A recurrence allows us to finish the proof. 
\end{proof}

From now on, we will say that $(E_N)_{N \in \mathbb{N}}$ converges up to the $n^{th}$ order of fluctuations if it converges in moments up to the $n^{th}$ order of fluctuations. Let us remark that Theorems \ref{th:equiA} and \ref{equivalence2} are easily generalized in this setting of fluctuations.

\subsection{The $N$-development algebra of order $n$}
\subsubsection{Definition}
We can also generalize the algebra $\mathbb{C}[\mathcal{P}_k(N,N)]$ in order to study the algebraic fluctuations. We need to consider a formal variable $X$. 

\begin{definition}
The $N$-development algebra of order $n$ of $\mathcal{P}_k$, denoted by $\mathbb{C}_{(n)}[\mathcal{P}_k(N)]$, is the associative algebra generated by the elements of the form: $\frac{p}{X^{i}},$ where $p\in \mathcal{P}_k$ and $i \in \{0,…, n\}$. The product is defined such that, for any $p$ and $p'$ in $\mathcal{P}_k$, and any $i$ and $j$ in $ \{0,…, n\}$: 
\begin{align*}
\frac{p}{X^{i}}. \frac{p'}{X^{j}} = \frac{1}{N^{\max(i+j+\eta(p,p')-n,0)}}\frac{p\circ p'}{X^{\min(i+j+\eta(p,p'),n)}}.
\end{align*}
\end{definition}

The $\prec$-defect is non-negative, thus for any $i, j \in \{0,…, n\}$ and any $p, p' \in \mathcal{P}_k$, $min(i+j+\eta(p,p'), m )\geq 0$. This implies that $\frac{p\circ p'}{X^{\min(i+j+\eta(p,p'),m)}}$ is an element of the canonical basis of $\mathbb{C}_{(n)}[\mathcal{P}_k(N)]$. This shows that the product is well defined. 

When $n=0$, we recover the algebra $\mathbb{C}[\mathcal{P}_k(N,N)]$: the application $\Phi : \mathcal{P}_k \to \mathbb{C}_{(n)}[\mathcal{P}_k(N)]$ which sends $p$  on $\frac{p}{X^{0}}$ can be extended as an isomorphism of algebra between $\mathbb{C}[\mathcal{P}_k(N,N)]$ and $\mathbb{C}_{(0)}[\mathcal{P}_k(N)]$. Indeed if $p$ and $p'$ are two partitions in $\mathcal{P}_k$, 
\begin{align*} 
\Phi(p) \Phi(p') = \frac{p}{X^{0}} \frac{p'}{X^{0}}
=\frac{1}{N^{\eta(p,p')}}\frac{p\circ p'}{X^{0}}
=\frac{1}{N^{\eta(p,p')}} \Phi(p \circ p') = \Phi(p._{\!{}_N} p'). 
\end{align*}

\begin{notation}
\label{coordonnes}
For any $E \in \mathbb{C}[\mathcal{P}_k(N,N)]$, we denote by $\kappa_{p}^{i}(E)$ the coordinate of $E$ on $\frac{p}{X^{i}}$. 
\end{notation}

We will use a slight modification of the usual notion of convergence. 
\begin{definition}
A sequence $({E_N})_{N \in \mathbb{N}}$ in $\prod_{N \in \mathbb{N}} \mathbb{C}_{(n)}[\mathcal{P}_k(N)]$ converges if and only if for any $i \in \{0,…,n-1\}$, and any $p \in \mathcal{P}_k$, $\kappa_{p}^i (E_N)$ does not depend on $N$ and for any $p\in \mathcal{P}_k$, $\kappa_{p}^{n}(E_N)$ converges when $N$ goes to infinity. We denote then $\kappa^{n}_{p}(E)$ the limit of $\kappa^n_{p}(E_N)$. 
\end{definition}

\subsubsection{Convergence of $\mathbb{C}_{(n)}[\mathcal{P}_k(N)]$}
As for $\mathbb{C}[\mathcal{P}_k(N,N)]$, the algebra  $\mathbb{C}_{(n)}[\mathcal{P}_k(N)]$ converges as $N$ goes to infinity. Recall Definition \ref{convalggen} where we define the notion of convergence of a sequence of algebras. Theorem \ref{Convalg} has the following generalization. 

\begin{theorem}
\label{Convalg2}
The algebra $\mathbb{C}_{(n)}[\mathcal{P}_k(N)]$ converges, when $N$ goes to infinity, to the $\infty$-development algebra of order $n$, denoted by $\mathbb{C}_{(n)}[\mathcal{P}_{k}(\infty)]$, which is the associative algebra over $\mathbb{C}$ with basis $(\frac{p}{X^{i}})_{i\in \{0,...,n\}, p \in \mathcal{P}_k}$, endowed with the multiplication defined by: 
\begin{align*}
\forall p, p' \in \mathcal{P}_k, \forall i, j \in  \{0,…, n\},\ \  \frac{p}{X^{i}} \frac{p'}{X^{j}} = \delta_{i+j+\eta(p,p')\leq n}\frac{p \circ p'}{X^{i+j+\eta(p,p')}}. 
\end{align*}
\end{theorem}

\begin{proof}
Let $p$ and $p'$ be two partitions in $\mathcal{P}_k$ and let $i$ and $j$ be two positive integers. 
$$\frac{p}{X^i} \frac{p'}{X^{i'}} =\frac{1}{N^{\max(i+j+\eta(p,p')-n,0)}}\frac{p\circ p'}{X^{\min(i+j+\eta(p,p'),n)}} \underset{N \to \infty}{\longrightarrow} \delta_{i+j+\eta(p,p')\leq n}\frac{p \circ p'}{X^{i+j+\eta(p,p')}}, $$
where the first product is seen in $\mathbb{C}_{(n)}[\mathcal{P}_k(N)]$.  
\end{proof}

\subsubsection{Convergences: $\mathbb{C}_{(n)}[\mathcal{P}_k(N)]$, multiplication and semi-groups}
\label{sec:convalgebric} 
 
Using Theorem \ref{Convalg2} and using similar ideas as for the $n=0$ case, we deduce the two followings results. 

\begin{proposition}
\label{convfluctu2}
Let $(E_N)_{N \in \mathbb{N}}$ and $(F_N)_{N \in \mathbb{N}}$ be elements of $\prod_{N \in \mathbb{N}}\mathbb{C}_{(n)}[\mathcal{P}_k(N)]$ which converge. The sequence $(E_NF_N)_{N \in \mathbb{N}}$ converges. For any $i_0 \in \{0,…,n\}$ and for any $p_0 \in \mathcal{P}_k$:  

\begin{align*}
\kappa_{p_0}^{i_0}(EF) = \sum_{p,p' \in \mathcal{P}_k, i,i' \in \mathbb{N} | p \circ p' = p_0, i+i'+\eta(p,p') = i_0} \kappa_{p}^i(E)\kappa_{p'}^{i'}(F).
\end{align*}
\end{proposition}

The good behavior of the product, given by Proposition \ref{convfluctu2}, implies a criterion for the convergence of semi-groups in $\prod_{N \in \mathbb{N}}\mathbb{C}_{(n)}[\mathcal{P}_k(N)]$. Let $\left((E^N_{t})_{N}\right)_{t \geq0}$ be semi-group in  $\prod_{N \in \mathbb{N}}\mathbb{C}_{(n)}[\mathcal{P}_k(N)]$, which generator is denoted by $(H_N)_{N \in \mathbb{N}}$. By definition, it converges if and only if for any $t \geq 0$, $\left(E^N_t\right)_{N \in \mathbb{N}}$ converges.

\begin{proposition}
\label{semigroupalgebric}
The semi-group $\left((E^N_{t})_{N}\right)_{t \geq0}$ converges if the sequences $(E^{N}_0)_{N \in \mathbb{N}}$ and $(H_N)_{N \in \mathbb{N}}$ converge. Besides, for any $t_0\geq 0$, for any $p \in \mathcal{P}_k$ and any $i \in \{0, …, n\}$,
\begin{align*}
\frac{d}{dt}_{\mid t=t_0} \kappa_{p_0}^{i_0}(E_t) =  \sum_{p,p' \in \mathcal{P}_k, i,i' \in \mathbb{N} | p \circ p' = p_0, i+i'+\eta(p,p') = i_0} \kappa_{p}^i(H)\kappa_{p'}^{i'}(E_{t_0}).
\end{align*}
\end{proposition}

\subsubsection{ $\prod_{N=1}^{\infty}\mathbb{C}_{(n)}[\mathcal{P}_k(N)]$ and $\prod_{N=1}^{\infty}\mathbb{C}[\mathcal{P}_k(N)]$}
\label{sec:link}

Let us consider an element $(E_N)_{N \in \mathbb{N}}$ in $\prod_{N=1}^{\infty}\mathbb{C}[\mathcal{P}_k(N)]$ which converges up to the $n^{th}$ order of fluctuations. Using Theorem \ref{equivalencefortmomentfluctu}, we can associate a real number $\kappa_{p}^{i}(E)$ for any $p \in \mathcal{P}_k$ and any $i \in \{0,...,n\}$ such that for any $p \in \mathcal{P}_k$, Equation (\ref{eq:cumufluctu}) holds. Let us denote by $\kappa^{n}_{p}(E_N)$ the left hand side of Equation (\ref{eq:cumufluctu}). 

\begin{definition}
The lift of the sequence $(E_N)_{N \in \mathbb{N}}$ in $\prod_{N \in \mathbb{N}} \mathbb{C}_{(n)}[\mathcal{P}_k(N)]$, denoted by $({\bold {E}_N})_{N \in \mathbb{N}} $ is:
\begin{align*}
{\bold {E}_N} = \sum_{p \in \mathcal{P}_k} \left(\left( \sum_{i=0}^{n-1} \kappa_{p}^{i}(E) \frac{p}{X^{i}}\right) + \kappa^{n}_{p}(E_{N}) \frac{p}{X^{n}}\right). 
\end{align*}
\end{definition}

By definition, $({\bold E_N})_{N \in \mathbb{N}}$ converges as $N$ goes to infinity and for any $N \in \mathbb{N}$, one has $\mathcal{E}^{N}_{(n)}({\bold {E}_N}) = E_N$, where $\mathcal{E}^{N}_{(n)}$ is the evaluation morphism: 
\begin{align*}
\mathcal{E}^{N}_{(n)}:\ \ \ \mathbb{C}_{(n)}[{\mathcal{P}_k} (N)]  \ \ \ \ \ \ &\to \ \ \ \ \ \ \ \ \ \ \ \ \ \mathbb{C}[\mathcal{P}_k(N)]\\
\sum_{p \in \mathcal{P}_k} \sum_{i=0}^{n} \kappa_{p}^i(E) \frac{p}{X^{i}}\ &\mapsto \sum_{p \in \mathcal{P}_k} \sum_{i=0}^{n} \kappa_{p}^i(E) \frac{1}{N^{i}}\frac{p}{N^{{\sf nc}(p)- {\sf nc}(p \vee \mathrm{id}_k))}}. 
\end{align*}
The fact that  $\mathcal{E}^{N}_{(n)}$ is a morphism of algebra follows from a simple calculation. Besides, by definition, if a sequence $({\bold E_N})_{N \in \mathbb{N}}$ in $\prod_{N =1}^{\infty} \mathbb{C}_{n}[\mathcal{P}_k(N)]$ converges then $\mathcal{E}^{N}_{(n)}({\bold {E}_N})$ converges up to the $n^{th}$ order of fluctuations. 

\subsection{Convergences: $\mathbb{C}[\mathcal{P}_k(N)]$, multiplication and semi-groups}
Using the discussion in Section \ref{sec:link}, we can use Section \ref{sec:convalgebric} in order to state results for $\mathbb{C}[\mathcal{P}_k(N)]$. Let us consider $(E_N)_{N \in \mathbb{N}}$ and $(F_N)_{N \in \mathbb{N}}$ two elements of  $\prod_{N \in \mathbb{N}} \mathbb{C} [\mathcal{P}_k(N)]$ which converge up to the $n^{th}$ order of fluctuations.

\begin{theorem}
\label{convfluctuprod}
The sequence $\left(E_NF_N\right)_{N \in \mathbb{N}}$ converges up to the $n^{th}$ order of fluctuations. Besides, for any $i_0 \in \{0,…,n\}$, for any $p_0 \in \mathcal{P}_k$: 
\begin{align}
\label{produitfluctuationformule1}
\kappa_{p_0}^{i_0}(EF) &=\sum_{p,p' \in \mathcal{P}_k, i,i' \in \mathbb{N} | p \circ p' = p_0, i+i'+\eta(p,p') = i_0}  \kappa_{p}^i(E)\kappa_{p'}^{i'}(F),\\ 
\label{produitfluctuationformule2}
m_{p_0}^{i_0}(EF) &=  \sum_{p \in \mathcal{P}_k} \sum_{i+j+{\sf df}(p,p_0) = i_0} \kappa_{p}^{i}(E)m^{j}_{\!\!\text{ }^{t}p \circ p_0}(F),\\
\label{produitfluctuationformule3}
m_{p_0}^{i_0}(EF) &= \sum_{p \in \mathcal{P}_k} \sum_{i+j+{\sf df}(p,p_0) = i_0} m^{j}_{p_0\circ\! \text{ }^{t}p }(E) \kappa_{p}^{i}(F).
\end{align}
\end{theorem}

\begin{proof}
Let $({\bold {E}_N})_{N \in \mathbb{N}}$ (respectively  $({\bold {F}_N})_{N \in \mathbb{N}}$) be the lifts of $(E_N)_{N \in \mathbb{N}}$ (resp. $(F_N)_{N \in \mathbb{N}}$) in $\prod_{N \in \mathbb{N}} \mathbb{C}_{(n)}[\mathcal{P}_k(N)]$. The two sequences $({\bold {E}_N})_{N \in \mathbb{N}}$ and $({\bold {F}_N})_{N \in \mathbb{N}}$ converge. According to Proposition \ref{convfluctu2}, the sequence $({\bold {E}_N}{\bold {F}_N})_{N \in \mathbb{N}}$ converges. For any $i_0 \in \{0, …, n\}$ and for any $p_0 \in \mathcal{P}_k$: 

\begin{align}
\label{eq}
\kappa_{p_0}^{i_0}({\bold E}{\bold F}) = \sum_{p,p' \in \mathcal{P}_k, i,i' \in \mathbb{N} | p \circ p' = p_0, i+i'+\eta(p,p') = i_0} \kappa_{p}^i({\bold E})\kappa_{p'}^{i'}({\bold F}).
\end{align}
This implies that the sequence $\left(\mathcal{E}_{(n)}^{N}({\bold E_N}{\bold F_N})\right)_{N \in \mathbb{N}}$ converges in moment up to the $n^{th}$ order of fluctuations. Since $\mathcal{E}_{(n)}^{N}$ is a morphism of algebra, $\mathcal{E}_{(n)}^{N}({\bold E_N}) = E_N$ and $\mathcal{E}_{(n)}^{N}({\bold F_N}) = F_N$, for any $N \in \mathbb{N}$, $\mathcal{E}_{(n)}^{N}({\bold E_N}{\bold F_N}) =  E_NF_N.$ We deduce that $\left(E_NF_N\right)_{N \in \mathbb{N}}$ converges up to the $n^{th}$ order of fluctuations. The Equation (\ref{produitfluctuationformule1})  is a consequence of Equation (\ref{eq}). The Equations (\ref{produitfluctuationformule2}) and  (\ref{produitfluctuationformule3}) are consequences of Theorem \ref{th:dautresvaleurs}. 
\end{proof}

Let us suppose that $\left( \left( E_t^{N}\right)_N\right)_{ t \geq 0}$ is a semi-group in $\prod_{N \in \mathbb{N}} \mathbb{C}[\mathcal{P}_k(N)]$ whose generator is $\left(H_N\right)_{N \in \mathbb{N}}$. We say that $\left( \left( E_t^{N}\right)_N\right)_{ t \geq 0}$ converges up to the $n^{th}$ order of fluctuations if and only if for any $t \geq 0$, $\left( E_t^{N}\right)_{N \in \mathbb{N}}$ converges up to the $n^{th}$ order of fluctuations. The following result is a generalization of Theorem \ref{semigroup}: it is a direct consequence of Proposition \ref{semigroupalgebric} and some usual arguments. 

\begin{theorem}
\label{semigroupfluctu}
The semi-group $\left( \left( E_t^{N}\right)_N\right)_{ t \geq 0}$ converges up to the $n^{th}$ order of fluctuations if the sequences $(E_0^{N})_{N \in \mathbb{N}}$ and $(H_N)_{N \in \mathbb{N}}$ converge up to the $n^{th}$ order of fluctuations. Besides, we have the two differential systems of equations: for any $p_0 \in \mathcal{P}_k$, for any $t_0 \geq 0$, for any $ i_0 \in \{0,…,n\}$: 
\begin{align}
\frac{d}{dt}_{\mid t=t_0}\kappa_{p_0}^{i_0}(E_t) &= \sum_{p,p' \in \mathcal{P}_k, i,i' \in \mathbb{N} | p \circ p' = p_0, i+i'+\eta(p,p') = i_0}  \kappa_{p}^i(H)\kappa_{p'}^{i'}(E_{t_0}), \\
\frac{d}{dt}_{\mid t=t_0} m_{p_0}^{i_0} (E_t)  &=  \sum_{p \in \mathcal{P}_k} \sum_{i+j+{\sf df}(p,p_0) = i_0} \kappa_{p}^{i}(H)m^{j}_{\!\!\text{ }^{t}p \circ p_0}(E_{t_0}),\\
\frac{d}{dt}_{\mid t=t_0} m_{p_0}^{i_0} (E_t)  &= \sum_{p \in \mathcal{P}_k} \sum_{i+j+{\sf df}(p,p_0) = i_0} m^{j}_{p_0\circ\! \text{ }^{t}p }(H) \kappa_{p}^{i}(E_{t_0}).
\end{align}
\end{theorem}

The Remark \ref{rq:genersemigroup} can be generalized to fluctuations.

\begin{notation}\label{notation:notationsfluctu}
As explained at the beginning of Section \ref{sec:Fluctuations}, for sake of clarity, in the following article \cite{Gab2}, we will use some $\boxplus$ and $\boxtimes$ convolutions in the setting of fluctuations of higher order. We will define them on $(\bigoplus_{k=0}^{\infty} \mathbb{C}[\mathcal{P}_k])^{*}$. 

For any set $X$, $\mathcal{E}(X)$ is the notation for the set of sets of $X$. For any $I \in \mathcal{E}(X)$,  let $I^{c}$ be the complement of $I$ in $X$. Recall that $p_{I}$ is the extraction of $p$ to $I$ defined in the beginning of Section \ref{sub:irredu}. Let $k$ be an integer and $p \in \mathcal{P}_k$, the set of factorizations $\mathfrak{F}_{2}(p)$ is the set of $(p_1,p_2,I) \in \mathcal{P} \times \mathcal{P} \times \mathcal{E}(\{1,...,k,1',...,k'\})$ such that $p_{ I} = p_1$, $p_{ I^{c}} = p_2$  and ${\sf nc}(p_1)+{\sf nc}(p_2) = {\sf nc}(p)$.

Let $\phi_1$ and $\phi_2$ in $(\bigoplus_{k=0}^{\infty} \mathbb{C}[\mathcal{P}_k])^{*}$, for any $p_0 \in \cup_{k} \mathcal{P}_k$ and any $i \in \{0,...,n\}$: 
\begin{align*}
\phi_1 \boxplus \phi_2 (p_0,i_0) &= \!\sum_{(p_1, p_2, I ) \in \mathfrak{F}_2(p)} \sum_{i_1,i_2 \in \mathbb{N}|i_1+i_2 = i_0}\!\!\!\!\! \phi_1(p_1,i_1) \phi_{2}(p_2, i_2), \\
\phi_1 \boxtimes \phi_2( p_0,i_0 ) &= \sum_{p_1,p_2 \in \mathcal{P}_k, i_1, i_2 \in \mathbb{N} | p_1\circ p_2=p_0, i_1+i_2+\eta(p_1,p_2) = i_0} \phi_1(p_1,i_1) \phi_2(p_2,i_2).
\end{align*}
\text{ }\\
\end{notation}

\section{Conclusion}
By introducing a distance and a new order on partitions, we defined new structures on $\mathcal{P}= \cup_{k} \mathcal{P}_k$ and on $(\mathbb{C}[\mathcal{P}])^{*}$. In \cite{Gab2}, we will define the notion of $\mathcal{P}$-tracial algebras which are algebras endowed with multi-linear observables. These observables are compatible with the product and are indexed by the partitions in $\mathcal{P}$. The structures defined and studied in this article will allow us to study  $\mathcal{P}$-tracial algebras and to define a notion of $\mathcal{P}$-freeness. 

Besides, we emulated in this article the theory of random matrices. Using these results, we will define in \cite{Gab2} the notion of finite dimensional cumulants for random matrices. This will allow us to get various results on asymptotics of random matrices. 

In \cite{Gab3}, we will use these results in order to study the asymptotic of general random walks on the symmetric group. This will allow us to define the $\mathfrak{S}(\infty)$-master field which is in some sense the limit of the $\mathfrak{S}(N)$-Yang-Mills measure.\\

{\em Acknowledgements.} The author would like to gratefully thank A. Dahlqvist, Yvain Bruned,  Tom Halverson, Arun Ram and G. C\'{e}bron for the useful discussions, his PhD advisor Pr. T. L\'{e}vy for his helpful comments and his postdoctoral supervisor, Pr. M. Hairer, for giving him the time to finalize this article. 

The first version of this work has been made during the PhD of the author at the university Paris 6 UPMC. This final version of the paper was completed during his postdoctoral position at the University of Warwick where the author is supported by the ERC grant, “Behaviour near criticality”, held by Pr. M. Hairer.

\bibliographystyle{plain}
\bibliography{biblio}

\end{document}